\documentclass[11pt,reqno]{amsart}

\usepackage{amsmath}
\usepackage{amsthm}
\usepackage{graphicx}
\usepackage{mathtools}
\usepackage{amsfonts}
\usepackage{amssymb}
\usepackage{hyperref}
\usepackage{bm}
\usepackage[margin=1 in]{geometry}
\usepackage{enumerate}
\usepackage[normalem]{ulem}
\usepackage{tikz}
\usepackage{xcolor}
\usetikzlibrary{matrix,arrows,positioning,automata}
  \usepackage{cancel}

\numberwithin{equation}{section}

\newtheorem{theorem}{Theorem}[section]
\newtheorem{lemma}[theorem]{Lemma}
\newtheorem{proposition}[theorem]{Proposition}
\newtheorem{corollary}[theorem]{Corollary}
\newtheorem{assumption}[theorem]{Assumption}

\newtheorem{convention}[theorem]{Convention}

\theoremstyle{definition}

\newtheorem{remark}[theorem]{Remark}

\def\E{{\mathbb E}}

\def\R{{\mathbb R}}
\def\N{{\mathbb N}}

\def\FF{{\mathbb F}}
\def\P{{\mathcal P}}

\def\L{{\mathcal L}}

\def\A{{\mathbb A}}
\def\F{{\mathcal F}}
\def\tr{{\mathrm{Tr}}}
\def\Var{{\mathrm{Var}}}

\newcommand{\bx}{\mathbf{x}}
\newcommand{\bX}{\mathbf{X}}
\newcommand{\bW}{\mathbf{W}}
\newcommand{\sA}{\mathcal{A}}
\newcommand{\bP}{\mathbb{P}}
\newcommand{\sF}{\mathcal{F}}
\newcommand{\vdt}{\mathcal{V}_{\text{dist}}}
\newcommand{\sL}{\mathcal{L}}
\newcommand{\sP}{\mathcal{P}}
\newcommand{\sR}{\mathcal{R}}
\newcommand{\spt}{\sP_2(\R^d)}
\def\ds{\displaystyle}
\newcommand{\be}{\begin{equation}}
\newcommand{\ee}{\end{equation}}

\newcommand{\linf}{L^{\infty}}
\newcommand{\sG}{\mathcal{G}}
\newcommand{\by}{\mathbf{y}}

\newcommand{\bz}{\mathbf{z}}
\newcommand{\bY}{\mathbf{Y}}
\newcommand{\bZ}{\mathbf{Z}}
\newcommand{\vd}{V_{\text{dist}}}
\newcommand{\ad}{\mathcal{A}_{\text{dist}}}
\newcommand{\stwo}{\mathcal{S}^2}

\newcommand{\ltwo}{L^2}
\newcommand{\jol}{J_{\text{OL}}}
\newcommand{\cv}{\mathcal{V}}
\newcommand{\cu}{\mathcal{U}}
\newcommand{\ope}{\text{op}}
\newcommand{\linfty}{L^{\infty}}

\newcommand{\calpha}{C^{\alpha}}

\newcommand{\calphaloc}{\calpha_{\text{loc}}}

\newcommand{\wass}{\mathcal{W}}
\newcommand{\mf}{\text{MF}}

\title[Approximately optimal distributed controls]{Approximately optimal distributed stochastic controls beyond the mean field setting}
\author{Joe Jackson and Daniel Lacker} 
\address{The University of Texas at Austin}
\email{jjackso1@utexas.edu}
\address{Columbia University}
\email{daniel.lacker@columbia.edu}

\thanks{J.J. is supported by the NSF under Grant No. DGE1610403. D.L. is partially supported by the NSF CAREER award DMS-2045328. Any opinions, findings and conclusions or recommendations expressed in this material are those of the authors and do not necessarily reflect the views of the NSF}

\begin{document}

\begin{abstract}
We study high-dimensional stochastic optimal control problems in which many agents 
cooperate
to minimize a convex cost functional.
We consider both the full-information problem, in which each agent observes the states of all other agents, and the distributed problem, in which each agent observes only its own state.
Our main results are sharp non-asymptotic bounds on the gap between these two problems, measured both in terms of their value functions and optimal states.
Along the way, we develop theory for distributed optimal stochastic control in parallel with the classical setting, by characterizing optimizers in terms of an associated stochastic maximum principle and a Hamilton-Jacobi-type equation.
By specializing these results to the setting of mean field control, in which costs are (symmetric) functions of the empirical distribution of states, we derive the optimal rate for the convergence problem in the displacement convex regime.
\end{abstract}

\maketitle

\setcounter{tocdepth}{1}
\tableofcontents

\section{Introduction} \label{sec.intro}

Consider a stochastic optimal control problem, on a finite time horizon $T > 0$, of the following form. Each of $n$ agents $i=1,\ldots,n$ controls a state process $X^i_t$, with values in $\R^d$, governed by the stochastic dynamics
\begin{equation}
dX^i_t = \alpha^i(t,\bX_t)dt + dW^i_t.  \label{intro:SDE}
\end{equation}
Here $W^1,\ldots,W^n$ are independent Brownian motions, and $\bm\alpha = (\alpha^1,\ldots,\alpha^n)$ is a Markovian (feedback) control which depends on the entire vector $\bX_t=(X^1_t,\ldots,X^n_t)$ of states. That is, $\bm\alpha$ belongs to the set of full-information controls, defined as follows (and more precisely in Section \ref{se:probformulation}). \newline \newline 
\textbf{Full-information controls}: $\sA$ denotes the set of $\bm{\alpha} = (\alpha^1,\ldots,\alpha^n)$, where $\alpha^i : [0,T] \times (\R^d)^n \to \R^d$ is measurable for each $i$, and the SDE \eqref{intro:SDE} admits a unique strong solution. \newline \newline 
The full-information control problem is to minimize a functional of the following form:
\begin{align}
V := \inf_{\bm{\alpha} \in \sA} J(\bm{\alpha}), \qquad J( \bm{\alpha}) := \E\bigg[\int_0^T \bigg(\frac{1}{n} \sum_{i = 1}^n L^i(X_t^i, \alpha^i(t,\bX_t)) + F(\bX_t) \bigg) dt + G(\bX_T) \bigg]. \label{intro:fullinfo}
\end{align}
Precise assumptions on $L^i$, $F$, and $G$ are deferred to Section \ref{se:probformulation}.
It will be important later that the running cost is additively separable as a function of the controls, as opposed to taking the more general form $f(\bX_t,\bm{\alpha}(t,\bX_t))$ with arbitrary dependence on all $n$ states and controls; see Section \ref{se:outlook} for additional discussion of possible generalizations. 
We have chosen to focus on Markovian (closed-loop) controls throughout the paper, but a parallel and largely equivalent story could be told using open-loop controls (processes adapted to a given filtration).

The framework of \emph{mean field control} provides an important special case of our setup as well as a point of reference for our work.
The \emph{mean field case} arises when $L^i$ does not depend on $i$, and when $F$ and $G$ are symmetric:
\begin{align}
\textbf{Mean field case: } \ \ L^i=L, \ \ F(\bx) = \sF(m^n_{\bx}), \ \ G(\bx) = \sG(m^n_{\bx}), \quad m^n_{\bx} := \frac{1}{n}\sum_{i=1}^n \delta_{x_i}, \label{intro:MF1}
\end{align}
for some nice functions $\sF,\sG : \P(\R^d) \to \R$, where $\P(\R^d)$ is the space of probability measures on $\R^d$.
Here and throughout we write $\bx=(x^1,\ldots,x^n)$ for a generic element of $(\R^d)^n$.
In the mean field case, it is known \cite{lacker2017limit,djete2022mckean} under general assumptions on $(L,\sF,\sG)$ that the value function $V$ converges as $n\to\infty$  to the value of the corresponding mean field control problem:
\begin{align}
\begin{split}
V_{\mathrm{MF}} = \inf_\alpha \E &\bigg[ \int_0^T\big( L(X_t,\alpha(t,X_t)) + \sF(\L(X_t))\big)dt + \sG(\L(X_T))\bigg], \\
\text{where } \ \ &dX_t = \alpha(t,X_t)dt + dW_t, \quad \L(X_t)=\mathrm{Law}(X_t).
\end{split} \label{intro:MF2}
\end{align}
The idea is that the state processes should become approximately i.i.d.\ as $n\to\infty$, and their empirical measure $m^n_{\bX_t}$ should thus be close to the common law $\L(X_t)$ of the state processes, by a law of large numbers. 
If $\alpha^{\mathrm{MF}} : [0,T] \times \R^d \to \R^d$ denotes an optimal control for this mean field control problem, then the controls $(\alpha^{\mathrm{MF},i})_{i=1}^n \in \sA$ defined by $\alpha^{\mathrm{MF},i}(t,\bx) := \alpha^{\mathrm{MF}}(t,x_i)$ should be nearly optimal for the original control problem \eqref{intro:fullinfo}.

Notably, the approximately optimal controls $\bm{\alpha}^{\mathrm{MF}}$ constructed from the mean field limit are \emph{distributed}, in the sense that agent $i$'s control depends only on the state of agent $i$, not the other agents. Indeed, one of the primary motivations for mean field control theory is that it provides a recipe for constructing near-optimal distributed controls for the often less tractable high-dimensional control problem.

Outside of the highly symmetric mean field case \eqref{intro:MF1}, it is not as clear how to construct near-optimal distributed controls for the control problem \eqref{intro:fullinfo}.
There have been some recent proposals to extend the mean field framework to accommodate certain models with heterogeneous interactions which possess certain asymptotic structure, by taking advantage of the theory of graphons \cite{aurell2022stochastic,bayraktar2022propagation,gao2020linear,lacker2022label}.
For instance, suppose the cost function $G$ (and similarly $F$) takes the form
\begin{align}
G(\bx) = \frac{1}{n}\sum_{i=1}^n G_1(x^i) + \frac{1}{n }\sum_{i,j=1}^n J_{ij}G_2(x^i-x^j), \label{intro:hetero}
\end{align}
for some functions $G_1,G_2 : \R^d \to \R$ and some symmetric \emph{interaction matrix} $J$. 
If $J$ converges as $n\to\infty$ to a graphon in a suitable sense, then a similar (but more involved) recipe should apply as in the mean field case: One can first solve a corresponding limiting problem, and then use it to construct near-optimal controls for the $n$-agent problem which are \emph{distributed}.
We must stress that both the mean field and graphon-based approaches require specific asymptotic structure for the cost functions $(L^i,F,G)$.

The main goal of this paper is to propose a new perspective on the construction of near-optimal distributed controls, which is non-asymptotic in nature and imposes no such structural assumptions.
We achieve this by studying directly the problem of optimization over distributed controls.
\newline \newline 
\textbf{Distributed controls:} $\ad$ denotes the set of $(\alpha^1,\ldots,\alpha^n) \in \sA$ for which $\alpha^i(t,x^1,\ldots,x^n)=\alpha^i(t,x^i)$ depends only on the $i^\text{th}$ state variable, for each $i$.
\newline \newline 
The distributed optimal control problem is defined by
\begin{align}
\vd := \inf_{\bm{\alpha} \in \ad} J(\bm{\alpha}). \label{intro:dstr}
\end{align}
Note that $V \le \vd$, because $\vd$ minimizes the same functional over a smaller class of admissible controls.
The quantity $\vd-V$ can be seen as a measure of the degree of suboptimality of distributed controls for the original full-information problem.
When it can be shown that $\vd-V < \epsilon$, it follows that there exist $\epsilon$-optimal distributed controls for the original problem \eqref{intro:fullinfo}.

The term ``distributed" is used broadly in the control theory literature, with a variety of different meanings. For instance, an early paper \cite{lions1973optimal} by J.L.\ Lions defines ``the control of distributed systems" in great generality as ``systems for which the state can be described by a solution of a partial differential equation."
Our terminology is much more specific and chosen to be consistent with its usage in the literature on mean field games and control.

\subsection{Near-optimality of distributed controls}

The most substantial part of our work derives quantitative bounds on $\vd-V$.
The simplest special case of our bound takes the following form (say, in $d=1$), stated more precisely in Theorem \ref{thm.est1}:
\begin{align}
0 \le \vd - V \le Cn\sum_{1 \le i < j \le n}\big( \|\partial_{ij} F\|_\infty^2 + \|\partial_{ij} G\|_\infty^2 \big), \label{intro:mainbound1}
\end{align}
for an explicit constant $C$ which crucially does not depend on $n$. It is important that the constant $C$ is explicit, so that our non-asymptotic bounds specialize properly to various asymptotic regimes such as the mean field case \eqref{intro:MF1} and heterogeneous interactions like \eqref{intro:hetero}.
In general, the constant $C$ in \eqref{intro:mainbound1} depends on the $L^\infty$-norms of the second derivatives of $L^i$ and its convex conjugate (Hamiltonian), as well as (the $L^\infty$-norms of) the operator norms of the Hessians of $F$ and $G$, where the latter are crucially dimension-free in most examples.
The case of quadratic cost $L^i(x,a)=|a|^2/2$ permits certain simplifications, and in Corollary \ref{co:quadratic} we obtain a particularly clean bound, stated here under the additional assumption of non-random initial states:
\begin{align}
 (\vd - V)^{1/2} \le  \bigg(nT^4\sum_{1 \le i < j \le n} \|\partial_{ij} F\|_\infty^2\bigg)^{1/2} +  \bigg(nT^2\sum_{1 \le i < j \le n} \|\partial_{ij} G\|_\infty^2\bigg)^{1/2}. \label{intro:mainbound1-quadratic} 
\end{align}

In addition, as a by-product of our proof of \eqref{intro:mainbound1}, we obtain a quantitative result in the spirit of propagation of chaos. Specifically, considering the optimal states $(X^1,\ldots,X^n)$ of the full-information problem, we find a distributed state process $(\widehat{X}^1,\ldots,\widehat{X}^n)$ such that, for each $k \le n$, the joint law of $(X^i)_{i \in S}$ is close to that of $(\widehat{X}^i)_{i \in S}$ in quadratic Wasserstein distance for ``most" choices of $S \subset \{1,\ldots,n\}$ of cardinality $k$. See Theorem \ref{thm.est2} for a precise statement.

The two main assumptions behind our main bound \eqref{intro:mainbound1} are (1) the boundedness of the second (but not the first) derivatives of $F$, $G$, and the Hamiltonians $H^i(x,p) := \sup_{a \in \R^d}\big( - a \cdot p - L^i(x,a)\big)$, and (2) the convexity of $F$, $G$, and $L^i$ for each $i$.
In fact, in the precise Theorem \ref{thm.est1} below, a sharper bound is given in which the $L^\infty$ norms in \eqref{intro:mainbound1} are replaced by certain $L^2$ norms. The boundedness assumptions can be relaxed in the case of quadratic Hamiltonian, as discussed in Remark \ref{re:unboundedness}. The convexity assumptions are difficult to remove but we give some modest results in this direction in Section \ref{sec.nonconvex} under additional smallness assumptions.

The most important feature of the bound \eqref{intro:mainbound1} is that it involves only the cross-derivatives, $\partial_{ij} F$ and $\partial_{ij}G$ for $i \neq j$, and none of the derivatives $\partial_{ii}F$ or $\partial_{ii}G$ appear.
The bound is non-asymptotic, but certain asymptotic regimes illustrate that it is quite sharp:

\subsubsection{The mean field case} \label{se:intro:MFC}

A first special case we highlight is the mean field case \eqref{intro:MF1}, treated in detail in Section \ref{sec.mfc}. In the bound \eqref{intro:mainbound1}, each cross-derivative becomes $\partial_{ij}G(\bx) = n^{-2} D_m^2 \sG(m^n_{\bx},x_i,x_j)$, which is $O(1/n^2)$ if the second Lions (Wasserstein) derivative $D_m^2\sG$ is bounded, and similarly for $F$. Summing up, the bound \eqref{intro:mainbound1} becomes $|V-\vd|=O(1/n)$. We will show separately (and in fact with no need for convexity assumptions) that $|\vd-V_{\mathrm{MF}}|=O(1/n)$, where $V_{\mathrm{MF}}$ was defined in \eqref{intro:MF2}. This then implies the optimal rate of $|V-V_{\mathrm{MF}}|=O(1/n)$ for the convergence problem; see Theorem \ref{thm.mf} for a precise statement. This appears to be a new result at this level of generality, though we stress that it relies crucially on some smoothness and most importantly the \emph{displacement convexity} of $\sF$ and $\sG$. It is expected, though not documented, that the same assumptions would lead to the existence of a smooth solution of the Hamilton-Jacobi equation on $\P_2(\R^d)$ which is (at least formally) satisfied by the value function of the mean field control problem; a smooth solution of this equation can be used to prove the same $O(1/n)$ convergence rate using the method developed in \cite{cdll} for mean field games and explained in \cite{germain2022rate} in the control setting. Interestingly, our method makes no use of this Hamilton-Jacobi equation.

It is natural to wonder if our method could lead to convergence rates in the difficult non-convex regime, but so far we are only able to handle small time horizon (Section \ref{sec.nonconvex}). In contrast, the recent work \cite{cdjs} obtains convergence rates in the non-convex regime without any restrictions on the time horizon, although the rates are suboptimal. See also \cite{Cecchin2021FiniteSN} for similar results in the finite state space.

It is interesting to note that several prior works on mean field games used the distributed control problem as a simplification of the full-information setup, because it is easier to establish rigorously the convergence of distributed equilibria of the $n$-player game to the mean field limit (i.e., the convergence problem). Indeed, the initial work of Lasry-Lions \cite{lasry2007mean} adopted this perspective, leaving open the full-information convergence problem; see also \cite{cardaliaguetol,feleqi2013derivation}, but note that none of these authors use the term ``distributed."
In our context, they studied the convergence of $\vd$ to $V_{\mathrm{MF}}$.

The full-information convergence problem (i.e., the convergence of $V$ to $V_{\mathrm{MF}}$), has by now been resolved in a qualitative sense, in quite general settings for mean field games \cite{lacker2020convergence,djete2021large} and control \cite{lacker2017limit,djete2022mckean}. The existing quantitative results for games  \cite{cdll,DelLacRam} rely on the analysis of the so-called \emph{master equation}, and the results for control  \cite{germain2022rate,cdjs,cardaliaguet2022regularity} rely  similarly on the Hamilton-Jacobi equation solved by the mean field value function,  though we mention that alternative BSDE techniques of \cite{possamai2021non} have yielded quantitative convergence results for mean field games. 
In our notation, these works studied  directly the convergence of $V$ to $V_{\mathrm{MF}}$, without any intermediate use of $\vd$.

The two approximations $V_{\mathrm{MF}} \approx \vd$  and $\vd \approx V$ are interestingly distinct in nature. The former is more probabilistic, essentially reducing to a rate of convergence of (smooth functionals of) i.i.d.\ empirical measures. The latter is more control-theoretic in nature, and it is here that convexity plays a critical role. See Remark \ref{re:differentapproximations} for additional details.

\subsubsection{Heterogeneous interactions}  \label{se:intro:hetero}

Our most important contribution, beyond proving the expected optimal rate for the mean field convergence problem, is to move beyond the mean field setting, with our main bound on $|V-\vd|$ applying in quite general asymmetric settings.
This connects with a recent and growing literature on large-population games and control problems in which the pairwise interactions between agents are modeled by a (potentially weighted) graph, as was mentioned in the paragraph surrounding \eqref{intro:hetero}. 
A key advantage of our approach over the prior graphon-based work discussed above is its quantitative and non-asymptotic nature, which makes it more readily applicable to sparse graphs.

We illustrate this point via the example \eqref{intro:hetero} of heterogeneous pairwise interactions modeled by an interaction matrix $J$, assumed to be symmetric with zeros on the main diagonal.
In this case, $\partial_{ij}G(\bx) = -\frac{1}{n } J_{ij}(D^2 G_2(x^i-x^j)+D^2 G_2(x^j-x^i))$ for $i \neq j$. Thus 
\begin{align}
n\sum_{1 \le i < j \le n} \|\partial_{ij} G\|_\infty^2 \le 2\|D^2 G_2\|_\infty^2   \mathrm{Tr}(J^2) / n . \label{intro:heterobound}
\end{align}
This reveals that, in an asymptotic regime where $n\to\infty$, we need $\tr(J^2) = o(n)$ in order to guarantee that $|V-\vd| \to 0$. This same condition on $J$ appeared in \cite{basak2017universality} in the study of Ising and Potts models on large graphs.

A common special case is when $J$ is $1/m$ times the adjacency matrix of a $m$-regular graph (with $m=n$ falling into the mean field case), which models each agent as interacting  with the average of its neighbors in the graph. In this case $\mathrm{Tr}(J^2)=n/m$, and so the right-hand side of \eqref{intro:heterobound} vanishes if $m\to\infty$ as $n\to\infty$. Building on this, we show in Corollary \ref{cor.mf-hetero} that in fact $|V-V_{\mathrm{MF}}|=O(1/m)$, where $V_{\mathrm{MF}}$ is the mean field value function defined in \eqref{intro:MF2} with $\sG(m) := \langle m, G_1\rangle + \langle m \otimes m, G_2(\cdot-\cdot)\rangle$. That is, as long as $m$ diverges, the control problem set on an $m$-regular graph converges to the usual mean field control problem. This insensitivity of the mean field limit with respect to changes in the $n$-agent interaction matrix is a ``universality" phenomenon, which has been observed in various settings and is expected for sufficiently dense matrices $J$ with row averages (approximately) equal to 1; cf.\ \cite[Remark 3.12]{lacker2022label} for mean field games, \cite{oliveira2019interacting} for interacting diffusions, and \cite[Section 2.1]{basak2017universality} for Ising and Potts models.
The regime of bounded degree $m$ (as $n\to\infty$) is  very different, and one cannot expect distributed controls to be approximately optimal; see \cite{lacker2022case} for an in-depth discussion of this denseness/sparseness threshold in a game-theoretic context.

\subsubsection{The Cole-Hopf case}

A noteworthy special case is the following:
\begin{align}
\textbf{Cole-Hopf case: } \ \ F \equiv 0, \ \ L^i(x,a)=|a|^2/2 \text{ for each } i, \ \ G \text{ is general.} \label{intro:ColeHopf}
\end{align}
In this special case, our bound \eqref{intro:mainbound1-quadratic} reduces to the recent result \cite[Corollary 2.14]{LacMukYeu}, which was proven (for $d=1$) using ideas from the theory of nonlinear large deviations.
The full-information control problem admits a well known semi-explicit solution, via Cole-Hopf transformation:
\begin{align}
V &= -\frac{1}{n}\log \int_{(\R^d)^n} e^{nG}\,d\gamma_T = \inf_{m\in \P((\R^d)^n)} \bigg( \int_{(\R^d)^n} G\,dm + \frac{1}{n}H(m\,|\,\gamma_T)\bigg), \label{intro:LMYapproach}
\end{align}
where $\gamma_T$ is the centered Gaussian measure with covariance matrix $TI$ and $H$ is relative entropy.
This can be interpreted as a ``static" formulation of the control problem, an optimization over $m \in \P((\R^d)^n)$ corresponding to the time-$T$ law of the state process.
The method of \cite{LacMukYeu} is based on a similar ``static" reformulation of the distributed control problem, recognizing that $\vd$ equals the same infimum as in \eqref{intro:LMYapproach} but restricted to product measures $m=m^1\otimes \cdots \otimes m^n$, and then using functional inequalities for log-concave measures. An initial motivation for our work was to generalize their result beyond the Cole-Hopf case, where the static reformulation is no longer available and a completely different approach is required.
We should mention, however, that the static approach of \cite{LacMukYeu} leads to more detailed information about the structure of the optimal distributed control, discussed in Remark \ref{re:brownianbridge} below.

\subsubsection{Outline of the rest of the introduction}

The remainder of the introduction is divided into roughly three parts. First, we describe some general theory around the distributed control problem, with a verification theorem and a maximum principle. Second, we sketch the proof of the main bound \eqref{intro:mainbound1}. Lastly, we discuss some possible generalizations.

\subsection{Toward a theory of distributed stochastic control}

The first part of our paper develops some general theory for the distributed control problem \eqref{intro:dstr}.
There is a well-developed theory for the full-information problem \eqref{intro:fullinfo}, as is summarized in several textbooks \cite{yong1999stochastic,fleming2006controlled,pham2009continuous}.
But the distributed control problem lies outside of the scope of classical theory, due to its atypical information constraint.
Nor may we turn to the literature on stochastic control under partial observations \cite{bensoussan1992stochastic}, in which there is a single, common set of information on which all controls $(\alpha^1,\ldots,\alpha^n)$ are based.


The perspective we adopt is to ``lift" the distributed control problem from a state process in $(\R^d)^n$ to a state process in $\P(\R^d)^n$, the space of vectors of probability measures on $\R^d$.
The space $\P(\R^d)^n$ can be identified with the space of product measures on $(\R^d)^n$, and it appears naturally here because the  state processes $(X^1,\ldots,X^n)$ are \emph{independent} precisely when the control is distributed (and the initial states are independent).
Letting $\mu^i_t$ denote the law of $X^i_t$, for any distributed control $\bm{\alpha}=(\alpha^1,\ldots,\alpha^n) \in \ad$ we may write
\begin{align*}
J(\bm{\alpha}) = \int_0^T \int_{(\R^d)^n}\bigg(\frac{1}{n} \sum_{i = 1}^n L^i(x_i, \alpha^i(t,x_i)) + F(\bx) \bigg) \prod_{i=1}^n \mu^i_t(dx_i)\,dt + \int_{(\R^d)^n} G(\bx)  \prod_{i=1}^n \mu^i_T(dx_i).
\end{align*}
We study this control problem over the (non-random) state process $(\mu^1_t,\ldots,\mu^n_t)_{t \in [0,T]}$, by defining the value function as a function on $[0,T] \times \P(\R^d)^n$.

The first main tool we need is a verification theorem for an associated nonlinear partial differential equation (PDE) on $[0,T] \times \P(\R^d)^n$. 
Focusing on the Cole-Hopf case for simplicity, the PDE for the distributed value function $\vdt : [0,T] \times \P(\R^d)^n \to \R$ takes the following form:
\begin{align}
-\partial_t \vdt(t,\bm m) &+ \frac12\sum_{i=1}^n \int_{\R^d} \Big(n|D_{m^i}\vdt(t,\bm m,y)|^2 - \mathrm{Tr} \big(D_y D_{m^i} \vdt(t,\bm m,y)\big)\Big) m^i(dy) = 0, \label{intro:vdt-PDE} \\
 \vdt(T,\bm m) &= \int_{(\R^d)^n} G\, d(m^1 \otimes \cdots \otimes m^n). \nonumber
\end{align}
Here we abbreviate $\bm m=(m^1,\ldots,m^n)$ for a generic element of $\P(\R^d)^n$, and $D_{m^i}$ denotes the Lions (Wasserstein) derivative with respect to the variable $m^i \in \P(\R^d)$; the precise definition is recalled in Section \ref{se:analysisonP2}.
Our (fairly straightforward) verification result (Proposition \ref{prop.verification}) shows that a smooth solution of this PDE must equal the value function of the distributed control problem, viewed as a function of the starting time and the $n$ initial distributions of the state processes.

Second, we prove a stochastic maximum principle, which leads to a description of optimizers for $\vd$ in terms of a forward-backward stochastic differential equation (FBSDE) of McKean-Vlasov type. Again focusing on the Cole-Hopf case \eqref{intro:ColeHopf} for simplicity, this equation takes the form
\begin{align}
\label{eq.mpdist-intro} 
 dX_t^i &= - n Y_t^i dt + dW_t^i, \quad  dY_t^i =  Z_t^i dW_t^i, \qquad X_0^i = x^i, \quad Y_T^i = \E[G(\bX_T)\,|\,X^i_T].
\end{align} 
Note that $(X^i,Y^i)_{i=1}^n$ must be independent, and so the conditional expectation $\E[G(\bX_T)\,|\,X^i_T]$ is just an integral over the law of the components $j \neq i$.
Under convexity assumptions, we show that this FBSDE is well-posed and characterizes the unique optimal distributed control; see Propositions \ref{prop.mpfbsde} and \ref{existunique}. This characterization is later used to prove Theorem \ref{thm.est3}, which gives $L^2$-estimates between the optimal controls of the distributed and full-information problems.

There are some notable omissions in our theory of distributed stochastic control.
We do not state a dynamic programming principle.
It is also natural to expect that our verification theorem can be complemented with a viscosity solution theory.
We chose not to develop the theory here to the utmost generality, but rather just enough to illustrate the form it should take, and to serve our primary goal of obtaining bounds like \eqref{intro:mainbound1}.

\subsection{Comparing the control problems} \label{se:intro:proofsketch}

Here we outline the main ideas going into the bound \eqref{intro:mainbound1}, focusing on the Cole-Hopf case defined in \eqref{intro:ColeHopf} and with $d=1$. Associated to the original (full-information) control problem \eqref{intro:fullinfo} we associate the usual value function $V : [0,T] \times (\R^d)^n \to \R$, defined for any initial time and any non-random state $\bx \in (\R^d)^n$. We begin by ``lifting" this value function to random but independent initial states, defining $\cv : [0,T] \times \P(\R^d)^n \to \R$ by
\begin{align}
\cv(t,\bm m) := \int_{(\R^d)^n} V(t,\cdot)\, d(m^1 \otimes \cdots \otimes m^n). \label{intro:Vlift}
\end{align}
We show (in Lemma \ref{lem.vlifteqn}) that $\cv$ solves exactly the same PDE \eqref{intro:vdt-PDE} except with the $0$ on the right-hand side replaced by the ``error term" $-E(t,\bm m)$, where
\begin{align*}
E(t,\bm m) := \frac{n}{2} \sum_{i = 1}^n \E\big[ |D_i V(t,\bm\xi)|^2 - |\E[D_i V(t,\bm\xi) \,|\, \xi^i]|^2 \big] = \frac{n}{2}\sum_{i=1}^n \E\,\Var(D_iV(t,\bm\xi)\,|\, \xi^i),
\end{align*}
where $\bm\xi=(\xi^1,\ldots,\xi^n) \sim m^1 \otimes \cdots \otimes m^n$, and where the variance of a random vector is defined as the sum of the variances of the components.
Using the verification argument mentioned before, we deduce that $\cv(t,\bm m)$ is the value function for a certain distributed optimal control problem in which this error term $E$ appears as a running cost. The optimal state process $\widehat{\bX}_s=(\widehat{X}^1_s,\ldots,\widehat{X}^n_s)$ for this problem $\cv(t,\bm m)$ turns out to be (by Lemma \ref{lem.vlifteqn}) the unique solution of the McKean-Vlasov  system
\begin{align}
d\widehat{X}^i_s = -n\E[D_i V(s,\widehat{\bX}_s)\,|\,\widehat{X}^i_s]ds + dW^i_s, \ \ s \in (t,T], \ \ \widehat{X}^i_t \sim m^i \text{ independent}. \label{intro:xhatdef}
\end{align}
Crucially, the processes $\widehat{X}^1,\ldots,\widehat{X}^n$ are independent, and this conditional expectation is really just shorthand for an expectation with respect to $(\widehat{X}^j_s)_{j \neq i}$.

Knowing that $\vdt$ and $\cv$ solve the same PDE up to an error term, we deduce (in Lemma \ref{lem.suffest}) via a form of comparison principle that
\begin{align}
0 \le \vdt(t,\bm m) - \cv(t,\bm m) \le \int_t^T E(s,\bm m_s)\,ds, \label{intro:pf:Ebound}
\end{align}
where $\bm m_s=(m^1_s,\ldots,m^n_s)$ with $m^i_s=\mathrm{Law}(\widehat{X}^i_s)$, where $\widehat{X}^i$ is defined in \eqref{intro:xhatdef}.
Estimating $E(t,\bm m)$ appears at first to be a challenging prospect, and we do not expect it to admit a useful bound which is uniform in $\bm m \in \P(\R^d)^n$. Indeed, focusing on time $t=T$,  in the mean field case \eqref{intro:MF1} we can expect at best $\|D_iV(T,\cdot)\|_\infty = \|D_iG\|_\infty = O(1/n)$ for each $i$ if $\sG$ is Wasserstein-Lipschitz, which leads to a useless bound of $\|E(T,\cdot)\|_\infty = O(1)$.
Remarkably, however, the dynamics of $E(s,\bm m_s)$ \emph{along the curve $(\bm m_s)_{s \in [t,T]}$} turn out to be tractable: Abbreviating $V=V(s,\widehat{\bX}_s)$, we show that
\begin{align*}
\frac{d}{ds}E(s,\bm m_s) &= n^2\E\bigg[ \sum_{i,j=1}^n \big(D_i V - \E[D_i V\,|\, \widehat{X}^i_s]\big)\big(D_j V - \E[D_j V\,|\, \widehat{X}^j_t]\big) D_{ij}V \bigg] \\
	&\quad + \frac{n}{2}\E\bigg[\sum_{i,j=1}^n|D_{ij}V|^2 -  \sum_{i=1}^n |\E[D_{ii}V\,|\, \widehat{X}^i_s]|^2 \bigg].
\end{align*}
The second line is nonnegative. When $G$ is convex, it is well known (and easy to show) that $V(t,\bx)$ is convex in $\bx$, and so the first line is also nonnegative. We deduce from \eqref{intro:pf:Ebound} that
\begin{align*}
\vdt(t,\bm m) - \cv(t,\bm m) \le (T-t) E(T,\bm m_T) = (T-t) \frac{n}{2}\sum_{i=1}^n \E\,\Var(D_iG(\widehat{\bX}_T)\,|\, \widehat{X}^i_T).
\end{align*}
Using the convexity of $V$ we can show that the drift of $\widehat{X}^i$ is monotone, and thus $\widehat{X}^i_T$ satisfies the Poincar\'e inequality $\Var(\varphi(\widehat{X}^i_T)) \le (T-t)\E|D\varphi(\widehat{X}^i_T)|^2$ for any $\varphi \in C^1$ and $i=1,\ldots,n$. The Poincar\'e inequality tensorizes to show that
\begin{align*}
\Var(D_iG(\widehat{\bX}_T)\,|\, \widehat{X}^i_T) \le (T-t)\sum_{j=1,\, j \neq i}^n\E\big[|D_{ij}G(\widehat{\bX}_T)|^2 \,|\,\widehat{X}^i_T\big].
\end{align*}
Finally, combining these inequalities yields
\begin{align*}
\vdt(t,\bm m) - \cv(t,\bm m) \le n(T-t)^2\sum_{1 \le i < j \le n} \E|D_{ij}G(\widehat{\bX}_T)|^2 .
\end{align*}
Taking $t=0$ yields  \eqref{intro:mainbound1}, in the Cole-Hopf case. See Theorem \ref{thm.est1} for the general version of this argument, beyond the Cole-Hopf case.

The convexity of $V$ is the decisive ingredient in the above argument, and no other quantitative information is needed of $(V,F,G,...)$.
In fact, even if $V$ is not convex but satisfies a lower bound of the form $D^2V \ge -(c/n)I$ in semidefinite order, then we may use Gronwall's inequality to estimate $E$ and get a similar inequality but with a different (but still $n$-independent) constant in place of $(T-t)$. The challenge in the non-convex regime is that this lower bound on $D^2V$ is difficult to obtain. In Section \ref{sec.nonconvex} we present some results based on different FBSDE arguments and certain smallness conditions in a non-convex regime, as long as $L^i(x,a)$ is sufficiently convex in $a$ relative to the non-convexity of $F$ and $G$ and/or the length of the time horizon $T$.


By quite different arguments using the stochastic maximum principle in place of the PDEs, we give in Theorem \ref{thm.est3} an $L^2$-estimate between the optimal controls in the full-information and distributed problem. Consider the solutions $(\bX,\bY,\bZ)$ and $(\overline{\bX},\overline{\bY},\overline{\bZ})$ of the FBSDEs corresponding to the full-information and distributed problems; the former is standard, and the latter was mentioned in \eqref{eq.mpdist-intro}  above (in the Cole-Hopf case). The idea of the proof of Theorem \ref{thm.est3}  is to apply It\^o's formula to $(\bX_t - \overline{\bX}_t) \cdot (\bY_t - \overline{\bY}_t)$ and use convexity and a well-timed conditioning. Ultimately, this leads to a bound on 
\begin{align*}
\frac{1}{n}\sum_{i=1}^n \E\int_0^T |\alpha^i_t-\overline\alpha^i_t|^2\,dt
\end{align*}
by (a different constant times) the right-hand side of \eqref{intro:mainbound1}, where  $(\alpha^1,\ldots,\alpha^n)$ and  $(\overline\alpha^1,\ldots,\overline\alpha^n)$ are the optimal controls of the full-information and distributed problems, respectively.

\subsection{Outlook} \label{se:outlook}

Our work raises a number of natural questions about generalizations and variants. 

The most obvious questions of generalization pertain to the form of the dynamics and cost functions. A more general running cost might take the form $F(\bm\alpha_t,\bX_t)$ instead of $F(\bX_t) + \frac{1}{n}\sum_{i=1}^n L^i(X^i_t,\alpha^i_t)$, depending in a general way on all $n$ states and controls. At this level of generality, though, the optimal controls $(\alpha^1,\ldots,\alpha^n)$ have a less tractable structure because the $n$ Hamiltonians do not decouple; this is reminiscent of the common \emph{separability} assumption imposed on the Hamiltonian in mean field game theory.
It might be possible to allow for more general state process dynamics of the form
\begin{align}
dX^i_t = b(X^i_t,\alpha^i_t)dt + \sigma(X^i_t,\alpha^i_t)dW^i_t, \label{intro:eq:genstate}
\end{align}
for sufficiently nice $(b,\sigma)$, although this would complicate our convexity arguments. 
Still more challenging would be \emph{coupled} state processes, in which $dX^i_t$ depends directly on (some of) the other states $(X^j)_{j \neq i}$. In a sense, our methods fundamentally rely on the equivalence between ``distributed controls" and ``independent state processes," ensured by having dynamics of the form \eqref{intro:eq:genstate}. This equivalence renders $\spt^n$ the natural state space for our PDE methods.
In our view, though, the practical relevance of distributed controls is not as clear when the state processes are coupled.

It is not obvious how to extend our methods to allow for \emph{common noise}, i.e., an additional Brownian motion $dB_t$ appearing in \eqref{intro:eq:genstate} which is common among all agents $i=1,\ldots,n$.
There are then at least two reasonable choices of what ``distributed controls" should be. One choice would keep the same definition, with $\alpha^i=\alpha^i(t,X^i_t)$ depending only on the agent's own state. A second choice would allow $\alpha^i=\alpha^i(t,X^i_t,B)$ to depend on the common noise, perhaps even its history; this is a natural choice because it accommodates the approximately optimal controls that one can construct in the setting of mean field control.

Another variant of our work, which seems approachable though perhaps not as widely applicable, would replace the set $\ad$ of distributed controls by the set $\sA_{\mathrm{det}}$ of deterministic controls, where $\alpha^i : [0,T] \to \R^d$ is non-random. For such a control, the state process $\bX_t$ is Gaussian, with covariance matrix given by $t$ times the identity. If $L^i(x,a)=|a|^2/2$, then we expect that an analogous bound to \eqref{intro:mainbound1} will hold, but with $\vd$ replaced by $\inf_{\bm\alpha \in \sA_{\mathrm{det}}}J(\bm\alpha)$, and with the summation on the right-hand side now \emph{including the diagonal terms} $\partial_{ii}F$ and $\partial_{ii}G$. This is shown in \cite[Remark 2.16]{LacMukYeu} in the Cole-Hopf case. Intuitively, if the second derivatives of $F$ and $G$ are all small, then $F$ and $G$ are nearly linear, and for linear $(F,G)$ one can show that the optimal full-information control is deterministic. It is not clear, however, how to generalize this to non-quadratic $L^i$.

\subsection{Organization of the paper}

We begin in Section \ref{se:notation} by summarizing some useful notation and conventions, the most important part being Section \ref{se:analysisonP2} which reviews the calculus on $\spt$ and introduces our somewhat non-standard notation for vectors of measures and product measures.
Section \ref{se:probformulation} gives the precise assumptions and setups of the full-information and distributed control problems. It is in Section \ref{se:probformulation} that we present our verification result (Proposition \ref{prop.verification}) and stochastic maximum principle (Section \ref{se:maxprinciple}).
Section \ref{sec.approx}, the most substantial of the paper, presents the precise main results; Theorem \ref{thm.est1} gives the most general form of our estimate on $V-\vd$ as in \eqref{intro:mainbound1} above, and Theorems \ref{thm.est2} and \ref{thm.est3} give additional bounds between the state and control processes of the two control problems.
Section \ref{sec.mfc} specializes to the mean field case, proving an optimal rate for the convergence problem, and Section \ref{sec.hetero} derives an analogous result in the setting of heterogenous interactions. Finally, Section \ref{sec.nonconvex} gives some variants of our main estimates when convexity is replaced by an appropriate form of smallness.

\section{Notation and preliminaries} \label{se:notation}
This section discusses the notation and terminology we will use throughout the paper, most of which is standard but compiled here for ease of reference. 

\subsection{Probabilistic set-up}
Fix numbers $n, d \in \N$, a terminal time $T > 0$, and  a probability space $(\Omega, \sF, \bP)$ hosting a standard $n$-dimensional Brownian motion $\bW = (W^1,...,W^n)$. We write $\mathbb{F} = (\sF_t)_{0 \leq t \leq T}$ for the (augmented) filtration generated by $\bW$ and $\sF_0$, $\sF_0 \subset \sF$ being a  complete atomless $\sigma$-algebra independent of $\bW$. Write $\L(Z)$ for the law of a random variable $Z$.

\subsection{Basic notation}
We will be working frequently with the space $(\R^d)^n$. Because we will often make quantitative statements about functions on $(\R^d)^n$ and their gradients and Hessians, and we need constants which exhibit sharp dependence on the dimension $n$, we take care in this short section to be absolutely clear about certain notational conventions.

We will use bold to denote elements of this space or processes taking values in this space, e.g., $\bx$ will denote an element of $(\R^d)^n$ and $\bX$ will denote a process taking values in $(\R^d)^n$. We will use superscripts to denote the components of $\bx$ in the following way: $\bx = (x^1,...,x^n)$, where each $x^i \in \R^d$. If necessary, we can further use subscripts to write $x^i = (x^i_1,...,x^i_d)$. As a rule, we will use $i$ or $j$ to denote indices which run from $1$ to $n$, and $k$ or $l$ to denote indices which run from $1$ to $d$. 

When working with functions $u = u(\bx) : (\R^d)^n \to \R$, we denote by $D_{x^i_k} u$ the derivative of $u$ in the argument $x^i_k$, and $D_{x^i_k x^j_l} u$ the mixed partial derivative in the arguments $x^i_k$, $x^j_l$. We frequently use the more compact notation $D_i u$ and $D_{ij} u$, $1 \leq i,j \leq n$, defined as follows. $D_i u = D_i u(\bx) : (\R^d)^n \to \R^d$ is the gradient of $u$ in $x^i$, given by
\begin{align*}
D_{i} u = (D_{x^i_1} u,\ldots, D_{x^i_d} u)^T, 
\end{align*}
and similarly $D_{ij} u = D_{ij}u(\bx) : (\R^d)^n \to \R^{d \times d}$ is given by $(D_{ij} u)_{kl} = D_{x^i_k x^j_l} u$. When we write $D^2 u$, we mean the element of $\R^{nd \times nd}$ written in blocks as 
\begin{align*}
    D^2 u = \begin{pmatrix} 
    D_{11} u & \dots  & D_{1n}u \\
    \vdots & \ddots & \vdots\\
    D_{n1} u & \dots  & D_{nn} u. 
    \end{pmatrix}
\end{align*}
That is, $D^2 u$ is the usual Hessian of $u$ when viewed in a natural way as a map $\R^{nd} \to \R$. Given $\bx,\by,\bz \in (\R^d)^n$, we will abuse notation slightly by writing 
\begin{align*}
(\by)^T D^2 u(\bx) \, \bz = \sum_{i,j = 1}^n (y^i)^T D_{ij} u(\bx) \, z^j = \sum_{i,j = 1}^n \sum_{k,l = 1}^d y^i_k D_{x^i_k x^j_l} u(\bx) \, z^j_l. 
\end{align*}
We make frequent use of the usual semidefinite order between symmetric matrices; that is, $D^2 u \leq CI_{nd \times nd}$ means that for any $\bx,\by  \in (\R^d)^n$ we have 
\begin{align*}
\by^T D^2 u(\bx) \by \leq C \sum_{i = 1}^n |y^i|^2.
\end{align*}
We will use this notation to mean the same even for non-symmetric matrices.

We write $|\cdot|$ to denote the usual Euclidean norms on both $\R^d$ and $(\R^d)^n$ i.e. $|\bx|^2 = \sum_{i = 1}^n |x^i|^2 = \sum_{i = 1}^n \sum_{k = 1}^d |x^i_k|^2$. We will write $|\cdot|$ to denote the Frobenius norm, while $|\cdot|_{\text{op}}$ is the operator norm for any finite-dimensional space of matrices:
\begin{align*}
|A| = \tr(AA^\top)^{1/2}, \quad |A|_{\ope} = \sup_{|x| = 1} |Ax|. 
\end{align*}

\subsection{Spaces of functions} \label{se:functionspaces}
Let $E$ denote some Euclidean space. Given a function $u : E \to \R$, we say that $u$ is $C^j$ if the derivatives the derivatives of $u$ up to order are continuous. For $\alpha \in (0,1]$, we say that $u$ is $C^{\alpha}$, and write $u \in \calpha$ if the norm 
\begin{align*}
\|u\|_{\calpha} \coloneqq \|u\|_{\linf} + \sup_{x,x' \in E, x \neq x'} \frac{|u(x) - u(x')|}{|x - x'|}
\end{align*}
is finite. We say that $u$ is locally $C^{\alpha}$, and write $u \in \calphaloc$, if for any compact $K \subset E$ we have 
\begin{align*}
\sup_{x, x' \in K, x \neq x'} \frac{|u(x) - u(x')|}{|x - x'|} < \infty. 
\end{align*}
Likewise, given $k \in \N$ and $\alpha \in (0,1]$ we write $u \in C^{k,\alpha}$ if $u$ and its derivatives up to order $k$ are in $C^{\alpha}$, and  we say that $u \in C^{k,\alpha}_{\text{loc}}$ if $u \in C^k$ and the order $k$ derivatives are in $C^{\alpha}_{\text{loc}}$. 

For a function $u = u(t,x) : [0,T] \times E \to \R$, we say that $u$ is $C^{1,2}$ if $\partial_t u$, $Du$, and $D^2u$ exist and are continuous on $[0,T] \times \R^d$. We say that $u \in \calpha$ if the norm 
\begin{align*}
\|u\|_{\calpha} \coloneqq \|u\|_{\linf} + \sup_{t,t' \in [0,T], x, x' \in E, t \neq t', x \neq x'} \frac{|u(t,x) - u(t',x')|}{|t -t'|^{\alpha/2} + |x - x'|^{\alpha}}
\end{align*}
is finite. We say that $u \in \calphaloc$ if for each compact $K \subset E$,
\begin{align*}
\sup_{t,t' \in [0,T], x, x' \in K, t \neq t', x \neq x'} \frac{|u(t,x) - u(t',x')|}{|t -t'|^{\alpha/2} + |x - x'|^{\alpha}} < \infty. 
\end{align*}
 We say that $u \in C^{1,\alpha}$ if $u$ and $Du$ are $C^{\alpha}$. We say that $u \in C^{2,\alpha}$ if $u$, $\partial_t u$, $Du$, and $D^2u$ are $\calpha$.

\subsection{Analysis on the space of probability measures} \label{se:analysisonP2}

We will also be working with the space $\sP_2(\R^d)^n=(\sP_2(\R^d))^n$, where $\sP_2(\R^d)$ denotes the Wasserstein space of probability measures with finite second moment. We denote by $m$ a generic element of $\spt$ and by $\bm m = (m^1,\ldots,m^n)$ a generic element of $\spt^n$. We use angled brackets to denote integration when convenient: $\langle m, g \rangle := \int g \,dm$. For $i = 1,\ldots,n$, we denote by $\bm m^{-i}$ the element of $\spt^{n-1}$ given by $\bm m^{-i} = (m^1,\ldots,m^{i-1},m^{i+1},\ldots,m^n)$. In an abuse of notation, we identify elements of $\spt^n$ (or $\spt^{n-1}$) with the corresponding product measures in $\sP_2((\R^d)^n)$ (or $\sP_2((\R^d)^{n-1})$). That is, $\bm m \in \spt^n$ is identified with $m^1 \otimes \cdots \otimes m^n \in \sP_2((\R^d)^n)$, so that, for example if $f : (\R^d)^n \to \R$, we may write 
\begin{align*}
\langle \bm m, f \rangle = \int_{(\R^d)^n} f(x^1,...,x^n) \prod_{i=1}^n m^i(dx^i)
\end{align*}
whenever the integral is well-defined.
We also use the convention that with $f$, $\bm m$ as above, the expression $\langle \bm m^{-i}, f \rangle$ denotes the function $\R^d \to \R$ given by 
\begin{align} \label{anglebrackets}
\langle \bm m^{-i}, f \rangle (x) = \int_{(\R^d)^{n-1}} f(x^1,\ldots,x^{i-1},x,x^{i+1},\ldots,x^n) \prod_{j=1,\,j \neq i}^n m^j(dx^j). 
\end{align}
Likewise if $f : (\R^d)^n \to \R^k$ for some $k$, $\langle \bm m^{-i}, f \rangle : \R^d \to \R^k$ can be defined component-wise by \eqref{anglebrackets}

We will use a calculus for functions $\spt^n \to \R$ which is inherited from the calculus on $\spt$ applied coordinatewise. Given a function $\cu = \cu(m) : \sP_2(\R^d) \to \R$, we refer to \cite[Chapter 5]{cardelbook1} or \cite[Section 2.1]{CST} for the definitions of the (order 1 and 2) Lions derivatives (a.k.a.\ \emph{intrinsic} derivative)
\begin{align*}
D_m \cu = D_m \cu(m,y) : \spt \times \R^d \to \R^d, \quad D_y D_m \cu = D_y D_m \cu(m,y) : \spt \times \R^d \to \R^{d \times d}
\end{align*}
and 
\begin{align*}
D_{mm} \cu = D_{mm} \cu(m, y, \overline{y}) : \spt \times \R^d \times \R^d \to \R^{d \times d}
\end{align*}
We recall that in general $D_m \cu(m,\cdot)$ is uniquely defined only on the support of $m$, but we always fix a continuous version if it exists.  
Note that some sources (such as \cite{cardelbook1,CST}) use the notation $\partial_\mu$ where we use $D_m$.

With these definitions in hand, we can consider a continuous map $\cu = \cu(t,\bm m) = \cu(t,m^1,..,m^n) : [0,T] \times \spt^n \to \R$. We call such a map $C^{1,2}$ if the (usual) time derivative $\partial_t \cu$ exists, if for each $(t, \bm m^{-i})$ the map $m^i \mapsto \cu(t,m^1,\ldots,m^n)$ admits Lions derivatives up to order two,
and if there are versions of these derivatives which are continuous on all of $[0,T] \times \spt^n \times \R^d$.
Moreover, if $\cu$ is $C^{1,2}$ we denote by 
\begin{align*}
D_{m^i} \cu = D_{m^i} \cu(t, \bm m, y) : [0,T] \times \spt^n \to \R^d
\end{align*}
the derivative of the map $m^i \mapsto \cu(t,m^1,\ldots,m^n)$, and likewise we denote by 
\begin{align*}
D_y D_{m^i} \cu = D_y D_{m^i} \cu(t,\bm m,y) : [0,T] \times \spt^n \times \R^d \to \R^{d \times d}
\end{align*}
the derivative of the map $y \mapsto D_{m^i} \cu(t,\bm m,y)$. 

\section{Problem formulations and value functions} \label{se:probformulation}

We now give a more precise discussion of the control problem stated in the introduction.
Our data consists of the $n$ ``Lagrangian" functions
\begin{align*}
L^i = L^i(x,a) : \R^d \times \R^d \to \R, \quad i = 1,...,n, 
\end{align*}
together with the running and terminal cost functions 
\begin{align*}
    \quad F = F(\bx) : (\R^d)^n \to \R, \quad G = G(\bx) : (\R^d)^n \to \R. 
\end{align*}
It will also be convenient to work with the Hamiltonians associated to $L^i$, i.e., the maps $H^i = H^i(x,p) : \R^d \times \R^d \to \R$ defined by 
\begin{align} \label{hdef}
H^i(x,p) = \sup_{a \in \R^d} \big(- a \cdot p -  L^i(x,a)\big).
\end{align}
Under mild assumptions of convexity and regularity on $L^i$, as in the following Assumption \ref{assump.conv}, the unique optimizer in \eqref{hdef} is given by $a = -D_p H^i(x,p)$.

Our main results operate under the following assumptions on the data $(G,F,L^i)$, which will be in force throughout this section and Section \ref{sec.approx}, and which make use of the function spaces summarized in Section \ref{se:functionspaces}. 

\begin{assumption} \label{assump.conv} 
The function $F, G : (\R^d)^n \to \R$ are bounded from below and convex. Moreover, $F$ is $C^2$, $G$ is in $C^{2,\alpha}_{\text{loc}}$ for some $\alpha \in (0,1)$ and both $F$ and $G$ have bounded derivatives of order two (but not necessarily of order one).

The functions $L^i, H^i : \R^d \times \R^d \to \R$ are $C^2$ with bounded derivatives of order two (but not necessarily of order one). Moreover, $L^i$ is bounded from below and satisfies
\begin{align} \label{lsrict}
    D^2 L^i(x,a) = \begin{pmatrix} 
    D_{xx} L^i(x,a) & D_{xa} L^i(x,a) \\
    D_{ax} L^i(x,a) & D_{aa} L^i(x,a)
    \end{pmatrix}
    \geq C_L \begin{pmatrix} 0 & 0 \\
    0 & I_{d \times d}    
    \end{pmatrix}, \quad \text{for all } x,a \in \R^d
\end{align}
or equivalently 
\begin{align} \label{lstrictequiv}
\big(D_x L^i(x,a) - D_x L^i(\bar{x}, \bar{a})\big) \cdot (x - \bar{x}) + \big(D_a L^i(x,a) - D_a L^i(\bar{x}, \bar{a})\big) \cdot (a - \bar{a}) \geq C_L |a - \bar{a}|^2, 
\end{align}
for all $x,\bar{x}, a, \bar{a} \in \R^d$ and for some constant $C_L > 0$.
\end{assumption} 

\begin{remark}
Our main structural condition is the convexity of the maps $F$, $G$, and $L^i$. The fact that $F$ and $G$ have bounded second derivatives is important for our method, indeed our main estimates (Theorems \ref{thm.est1}, \ref{thm.est2}, \ref{thm.est3}) involve the second derivatives of $F$ and $G$, though see Remark \ref{re:unboundedness} for some discussion of when and how it might be relaxed. The assumption that $L^i$ and $H^i$ have bounded second derivatives is not essential, and it would be more natural to assume, e.g., that the second derivatives of $H^i$ are bounded on $\R^d \times B_R$ for each $R > 0$. Our method can be used to obtain good estimates in this setting provided that we have good Lipschitz estimates on the value function of the control problem \eqref{controlstandard}, which in turn is possible when $F$ and $G$ are Lipschitz. Since we are working in a convex setting, we find it more natural to assume $F$ and $G$ grow quadratically, and so we cannot assume our value functions are Lipschitz, which is why we enforce the boundedness of second derivatives of $L^i$ and $H^i$.
\end{remark}

\subsection{The full-information control problem} 
As in the introduction, the set $\sA$ of full-information controls is defined as the set of $\bm{\alpha} = (\alpha^1,\ldots,\alpha^n)$, where $\alpha^i : [0,T] \times (\R^d)^n \to \R^d$ is measurable for each $i$, and the SDE
\begin{align} \label{xdynamics}
dX_s^i = \alpha^i(s,\bX_s) ds + dW_s^i, \quad s \in [t,T], \quad \bX_t \sim \bm m
\end{align}
admits a unique strong solution $\bX=(X^1,\ldots,X^n)$ satisfying $\E\int_t^T|\alpha^i(s,\bX_s)|^2\,ds < \infty$, for each $(t,\bm m) \in [0,T] \times \spt^n$. Recall here that we identify $\bm m=(m^1,\ldots,m^n)$ with the product measure $m^1 \otimes \cdots \otimes m^n$, so that $\bX_t \sim \bm m$ means that $X^i_t \sim m^i$ are independent.
The process $\bX$ is called the state process associated with $\bm\alpha$ (and initial position $(t,\bm m)$).

For each $(t,\bm m) \in [0,T] \times \spt^n$ and $\bm\alpha \in \sA$, define  the cost functional
\begin{align}
\mathcal{J}(t,\bm m, \bm\alpha) := \E\bigg[\int_t^T \bigg(\frac{1}{n} \sum_{i = 1}^n L^i(X_s^i, \alpha^i(s,X_s^i)) + F(\bX_s) \bigg) ds + G(\bX_T)\bigg], \label{def:scriptJ}
\end{align}
with $\bX$ being the state process associated with $\bm\alpha$, i.e., the unique solution of \eqref{xdynamics}.
The lifted value function $\cv : [0,T] \times \spt^n \to \R$ is defined by
\begin{align}
\cv(t,\bm m) := \inf_{\bm\alpha \in \sA} \mathcal{J}(t,\bm m,\bm\alpha), \label{def:Vlift}
\end{align}
Alternatively, we can define the standard value function $V : [0,T] \times (\R^d)^n \to \R$ via non-random initial positions:
\begin{align} \label{controlstandard}
V(t,\bx) := \cv(t,\delta_{\bx}) = \inf_{\bm\alpha \in \sA} \mathcal{J}(t,\delta_{\bx}, \bm\alpha),
\end{align}
where $\delta_{\bx}$ is identified with the vector $(\delta_{x^1},\ldots,\delta_{x^n}) \in \spt^n$.
A standard dynamic programming argument yields the identity
\begin{align}
\cv(t,\bm m) = \langle \bm m, V(t,\cdot) \rangle.  \label{def:Vlift-control}
\end{align}

Under Assumption \ref{assump.conv}, the value function $V$ is well-defined and is in fact a classical solution to the HJB equation
\be\label{eq.hjbclassic}
\left\{\begin{array}{l}
\ds - \partial_t V - \frac{1}{2} \Delta V + \frac{1}{n} \sum_{i = 1}^n H^i(x^i, nD_i V) = F(\bx), \quad (t,\bx) \in [0,T) \times (\R^d)^n, \\[1.2mm]
\ds  V(T,\bx) = G(\bx), \quad \bx \in (\R^d)^n.
\end{array}\right.
\ee
It is the unique classical solution, at least in the class of functions
satisfying the growth constraint 
\begin{align} \label{hjbgrowth}
|D V(t,\bx)| \leq C(1 + |\bx|), \quad \bx \in (\R^d)^n.
\end{align}
This point is fairly standard, but it is difficult to find a reference which covers the setting when $F$ and $G$ are not Lipschitz. We provide a brief sketch here. The fact that $V$ is a viscosity solution of \eqref{eq.hjbclassic} is well-known, so the first issue is to argue that it is $C^{1,2}$, hence a classical solution. For this, one can first use control-theoretic arguments (see Lemma \ref{lem.spectral} below) to show that that $V$ is continuous and in fact is $C^{1,1}_{\text{loc}}$, and then notice that the restriction $V^R$ of $V$ to $[0,T] \times B_R$ is the unique viscosity solution of
\be\label{eq.hjbtrunc}
\left\{\begin{array}{l}
\ds - \partial_t V^R - \frac{1}{2} \Delta V^R + \frac{1}{n} \sum_{i = 1}^n H^i(x^i, nD_i V^R) = F(\bx), \quad (t,\bx) \in [0,T) \times B_R, \\[1.2mm]
\ds  V^R(T,\bx) = G(\bx), \quad \bx \in B_R, \quad 
\ds V^R(t,\bx) = V(t,\bx), \quad \bx \in \partial B_R
\end{array}\right.
\ee
Interior regularity results for linear parabolic equations (see e.g. Theorem 9 in Section 3.4 of \cite{friedman}) together with the fact that $G \in C^{2,\alpha}_{\text{loc}}$ then allow one to deduce that $V \in C^{1,2}$. Control-theoretic arguments (see Lemma \ref{lem.spectral} below) again give that $V$ satisfies the growth condition \eqref{hjbgrowth}. Uniqueness of classical solutions satisfying this growth condition is also standard, since this condition is strong enough for the usual verification argument to apply.

\subsection{The distributed control problem} 

As in the introduction, the set $\ad$ of distributed controls is defined as the set of $\bm{\alpha} = (\alpha^1,\ldots,\alpha^n) \in \sA$ such that $\alpha^i(t,\bx)=\alpha^i(t,x^i)$ depends only on the $i^\text{th}$ coordinate, for each $i$. 
We can define the \emph{distributed value function}
\begin{align} \label{controldist}
\vdt(t,\bm m) := \inf_{\bm\alpha \in \ad} \mathcal{J}(t,\bm m,\bm\alpha), 
\end{align}
In analogy with the full-information value function $V$ given in \eqref{controlstandard}, it seems natural at first to define a distributed value function $\vd : [0,T] \times (\R^d)^n \to \R$ for deterministic initial positions, via
\begin{align*}
\vd(t,\bx) := \vdt(t,\delta_{\bx}) := \inf_{\bm\alpha \in \ad} \mathcal{J}(t,\delta_{\bx},\bm\alpha).
\end{align*}
However, the dynamic programming identity \eqref{def:Vlift-control} breaks down here, and 
\begin{align*}
\vdt(t,\bm m) \neq \langle \bm m, \vd(t,\cdot) \rangle, \ \text{ in general}.
\end{align*}

It will be a consequence of Proposition \ref{prop.verification} below that, if $\vdt$ is $C^{1,2}$, then  it solves the infinite-dimensional PDE
\be\label{eq.hjbdist}
\left\{\begin{array}{l}
\ds - \partial_t \vdt - \frac{1}{2} \sum_{i=1}^n \langle m^i, \tr (D_y D_{m^i} \vdt) \rangle + \frac{1}{n} \sum_{i=1}^n \langle m^i, H^i(\cdot, nD_{m^i} \vdt) \rangle  = \langle \bm m, F \rangle  \\ 
\ds  \vdt(T,\bm m) = \langle \bm m, G \rangle,  
\end{array}\right.
\ee
for $(t, \bm m) \in [0,T) \times \spt^n$, where we abbreviated $\vdt=\vdt(t,\bm m)$ in the first line, and
\begin{align*}
\langle m^i, H^i(\cdot, nD_{m^i} \vdt) \rangle &:= \int_{\R^d} H^i(y, nD_{m^i} \vdt(t,\bm m,y ))\,m^i(dy), \\
\langle m^i, \tr (D_y D_{m^i} \vdt) \rangle &:= \int_{\R^d}\tr (D_y D_{m^i} \vdt(t,\bm m,y ))\,m^i(dy).
\end{align*}

\subsection{The PDE approach} \label{sec.dpp}

In this section we provide a verification result connecting the distributed control problem \eqref{controldist} with the equation \eqref{eq.hjbdist}. 
In fact, our main estimates in Section \ref{sec.approx} will make no use of the PDE \eqref{eq.hjbdist}, favoring the control-theoretic definition of $\vdt$. 
However, we will need crucially a verification theorem for the PDE associated with the lifted version $\cv(t,\bm m)$ of the full-information value function.
With this in mind, we present in this section a verification theorem for a generalized distributed control problem which will encompass both $\vdt$ and $\cv$.

Suppose we are given functions $\sF, \sG : \spt^n \to \R$.
Consider the value function $\widehat{\cv} = \widehat{\cv}(t,\bm m) : [0,T] \times \spt^n$  
given by 
\begin{equation} \label{vgendef}
\widehat{\cv}(t,\bm m) = \inf_{\bm\alpha \in \ad} \widehat{\mathcal{J}}(t,\bm m,\bm\alpha), 
\end{equation}
where 
\begin{equation*}
\widehat{\mathcal{J}}(t,\bm m, \alpha) = \E\bigg[\int_t^T \bigg(\frac{1}{n} \sum_{i = 1}^n L^i(X_s^i, \alpha^i(s,X_s^i)) + \sF(\bm m_s) \bigg) ds + \sG(\bm m_T) \bigg]
\end{equation*}
with $\bX = (X^1,...,X^n)$ being defined as in \eqref{xdynamics}.
Note that $\vdt$ arises as a special case of this problem by setting  $\sF(\bm m) = \langle \bm m, F \rangle$ and  $\sG(\bm m) = \langle \bm m , G \rangle$. 
The relevant PDE is then 
\be\label{eq.hjbgen}
\left\{\begin{array}{l}
\ds - \partial_t \widehat{\cv} - \frac{1}{2} \sum_{i=1}^n \langle m^i, \tr (D_y D_{m^i} \widehat{\cv}) \rangle + \frac{1}{n} \sum_{i=1}^n \langle m^i, H^i(y, nD_{m^i} \widehat{\cv}) \rangle  = \sF(\bm m)  \\ 
\ds  \widehat{\cv}(T,\bm m) = \sG(\bm m_T),
\end{array}\right.
\ee
for $(t, \bm m) \in [0,T] \times \spt^n$. Before stating our verification result, we first state the relevant version It\^o's formula needed for its proof, which is a straightforward extension of the $n=1$ case given in \cite[Theorem 5.92]{cardelbook1}.

\begin{lemma} \label{lem.ito}
Let $\cu = \cu(t, \bm m) : [0,T] \times \spt^n \to \R$ be $C^{1,2}$, and let $\bX_t = (X^1,...,X^n)$ be an It\^o process of the form 
\begin{align*}
dX_t^i = b_t^i dt + dW_t^i, \quad X_0 = \xi \in L^2
\end{align*}
with $\E \int_0^T |b_t|^2 dt  < \infty$.
Suppose that for each  bounded set $K \subset \spt^n$  we have
\begin{align} \label{itosuff}
\sup_{t \in [0,T]} \sup_{\bm m \in K} \langle m^i, |D_y D_{m^i} \cu(t,\bm m,\cdot)|^2 \rangle < \infty, \quad i=1,\ldots,n. 
\end{align}
Set $\bm m_t = (\sL(X_t^1),...,\sL(X_t^n))$. Then 
\begin{align*}
\frac{d}{dt} &\cu(t,\bm m_t) = \partial_t \cu(t,\bm m_t) + \E\bigg[\frac{1}{2} \sum_{i = 1}^n \tr( D_y D_{m^i} \cu(t,\bm m_t, X_t^i))  + \sum_{i = 1}^nD_{m^i} \cu(t,\bm m_t, X_t^i) \cdot b_t^i\bigg]. 
\end{align*}
In particular, if $b_t^i = b^i(t,X_t^i)$, and we write $\bm m_t = (\sL(X_t^1),...,\sL(X_t^n))$, we have 
\begin{align*}
\frac{d}{dt} \cu(t,\bm m_t) = \partial_t \cu(t,\bm m_t) + \frac{1}{2} \sum_{i = 1}^n \langle m^i, \tr(D_y D_{m^i} \cu(t,\bm m_t, \cdot) ) \rangle + \sum_{i = 1}^n \langle m^i, b(t, \cdot) \cdot  D_{m^i} \cu(t,\bm m_t, \cdot) \rangle. 
\end{align*}
\end{lemma}

With this lemma in hand, we now give the verification result which connects the partial information control problem \eqref{vgendef} to the PDE \eqref{eq.hjbgen}.

\begin{proposition} \label{prop.verification}
Suppose that $H^i$ satisfies the conditions appearing in Assumption \ref{assump.conv} and $\sF$, $\sG$ 
are continuous functions $\spt^n \to \R$. Suppose further that $\cu : [0,T] \times \spt^n \to \R$ is $C^{1,2}$ and satisfies the PDE \eqref{eq.hjbgen} as well as \eqref{itosuff} for each bounded set $K \subset \spt^n$.
Finally, suppose that for each $t \in [0,T]$ and $\bm m \in \spt^n$ the McKean-Vlasov SDE 
\begin{align*}
    dX_s^i &= - D_p H(X_s^i, nD_{m^i} \cu(s,\bm m_s,X_s^i)) ds + dW_s^i, \ \ s \in [t,T], \quad \bX_t \sim \bm m , \\
    \bm m_s &= (\sL(X_s^1),...,\sL(X_s^n))
\end{align*} 
admits a strong solution such that 
\begin{align} \label{feedbackdef}
\alpha^i(s,x) \coloneqq - D_pH^i(x, nD_{m^i} \cu(s,\bm m_s, x))
\end{align}
is admissible, i.e., $(\alpha^1,\ldots,\alpha^n) \in \ad$. Then in fact $\cu = \widehat{\cv}$, and the control  \eqref{feedbackdef} is an optimizer for \eqref{vgendef}, unique in the sense that any other optimizer $(\beta^1,\ldots,\beta^n) \in \ad$ satisfies $\alpha^i = \beta^i$ a.e.\ on $[t,T] \times \R^d$ for each $i=1,\ldots,n$.
\end{proposition}
\begin{proof}
Fix a candidate control $\bm\alpha \in \ad$, and denote by $\bX$ the corresponding state process: 
\begin{align*}
X_s^i = X^i_t + \int_t^s \alpha^i(u,X_u^i) du + (W_s^i - W_t^i), \quad t \leq s \leq T.
\end{align*}
Set $\bm m_s=(m^1_s,\ldots,m^n_s) = (\sL(X_s^1),...,\sL(X_s^n))$, and use It\^o's formula in the form of Lemma \ref{lem.ito}, along with the PDE \eqref{eq.hjbdist} satisfied by $\cu$ and the definition of $H^i$, to compute 
\begin{align*}
\frac{d}{ds} \cu(s,\bm m_s) &= \partial_t \cu(s,\bm m_s) + \frac{1}{2} \sum_{i = 1}^n \langle m^i_s, \tr(D_y D_{m^i} \cu(s,\bm m_s, \cdot)) \rangle +    \sum_{i = 1}^n \langle m^i_s, \alpha^i(s, \cdot) \cdot D_{m^i} \cu(s,\bm m_s, \cdot) \rangle \\
&= \frac{1}{n} \sum_{i = 1}^n \langle m^i_s, H^i(\cdot, nD_{m^i} \cu(s,\bm m_s, \cdot)) 
 + \alpha^i(s, \cdot) \cdot  nD_{m^i} \cu(s,\bm m_s, \cdot) \rangle - \sF(\bm m_s) \\
 &= \frac{1}{n}\sum_{i = 1}^n \langle m^i_s, H^i(\cdot, nD_{m^i} \cu(s,\bm m_s, \cdot)) 
 + \alpha^i(s, \cdot) \cdot n D_{m^i} \cu(s,\bm m_s, \cdot) + L^i(\cdot, \alpha^i(s,\cdot)) \rangle 
 \\ & \hspace{1cm} - \frac{1}{n} \sum_{i = 1}^n \langle m^i, L^i(\cdot, \alpha^i(s,\cdot) \rangle - \sF(\bm m_s) \\
 &\geq - \frac{1}{n} \sum_{i = 1}^n \langle m^i, L^i(\cdot, \alpha^i(s,\cdot) )\rangle - \sF(\bm m_s), 
\end{align*}
with equality if and only if $\alpha^i$ satisfies \eqref{feedbackdef}. Integrating this differential inequality, we find that
\begin{align*}
\cu(t,\bm m_t) &\leq \int_t^T \bigg(\frac{1}{n} \sum_{i = 1}^n \langle m^i_s, L^i(\cdot, \alpha^i(s,\cdot)) \rangle + \sF(\bm m_s) \bigg) ds + \cu(T, \bm m_T) \\
&= \E\bigg[\int_t^T \bigg(\frac{1}{n} \sum_{i = 1}^n L^i(X_s^i, \alpha^i(s,X_s^i)) + \sF(\bm m_s) \bigg) ds + \sG(\bm m_T)\bigg], 
\end{align*}
with equality if and only if $\alpha^i$ satisfies \eqref{feedbackdef}. 
\end{proof}

\subsection{An alternative description of the optimal distributed control} \label{sec:dstr-opt-necessary}

An occasionally useful necessary condition for optimal distributed controls arises from the following simple observation: If $\bm\alpha=(\alpha^1,\ldots,\alpha^n) \in \ad$ is optimal for \eqref{controldist}, then $\alpha^i$ is optimal for a standard control problem with $\R^d$-valued state process, for each $i=1,\ldots,n$.
Specifically, given $i$, this control problem is 
\begin{align} \label{videf}
\inf_{\beta} \E\bigg[\int_t^T \bigg( \frac{1}{n} L^i(X_s, \beta(s,X_s)) + F^i(s,X_s) \bigg)ds +  G^i(X_T)\bigg], 
\end{align}
subject to 
\begin{align*}
dX_s = \beta(s,X_s) ds + dW_s^i, \quad X_t \sim m^i, 
\end{align*}
and with 
\begin{align*}
F^i(s,\cdot) = \langle \bm m_s^{-i}, F \rangle, \quad G^i(\cdot) = \langle \bm m_T^{-i}, G \rangle.
\end{align*}
Here as usual we use the notation $\bm m_s = (\sL(X_s^1),...,\sL(X_s^n))$.
Thus we must have 
\begin{align*}
\alpha^i(t,x) = - D_p H^i(t,nDv^i(t,x)), 
\end{align*}
where $v^i$ is the value function associated to the control problem \eqref{videf}, in other words the unique classical solution of 
\begin{align}
\begin{cases} \label{vieqn}
 - \partial_t v^i - \frac12\Delta v^i + \frac{1}{n} H^i(x, nD v^i) = F^i(t,x), \quad (t,x) \in [0,T) \times \R^d, \\
 v^i(T,x) = G^i(x), \quad x \in \R^d
\end{cases}
\end{align}
satisfying an appropriate growth condition.
It is worth noting that \eqref{vieqn} is  analogous to PDEs which have appeared for distributed Nash equilibria in mean field games \cite{cardaliaguetol,feleqi2013derivation,lasry2007mean}.

\subsection{The maximum principle} \label{se:maxprinciple}

As is well known \cite[Section 6.4.2]{pham2009continuous}, the optimizers of the full-information control problem \eqref{controlstandard} can be characterized in terms of an FBSDE,
\be \label{eq.mpstand}
\begin{cases}
\ds dX_t^i = - D_p H^i(X_t^i, n Y_t^i) dt + dW_t^i, \\
\ds  dY_t^i = - \bigg(\frac{1}{n}D_x L^i\big(X_t^i, -D_pH^i(X_t^i,nY_t^i)) + D_{i}F(\bX_t)\bigg) dt + \bZ_t^i  d\bW_t, \\
X_0^i = x^i, \quad Y_T^i = D_iG(\bX_T). 
\end{cases}
\ee
A solution to \eqref{eq.mpstand} is a triple of progressively measurable process $(\bX, \bY, \bZ)$ with values in $(\R^d)^n \times (\R^d)^n \times (\R^{d \times d})^{n \times n}$ satisfying \eqref{eq.mpstand} as well as
\begin{align*}
\E\bigg[ \sup_{t \in [0,T]}\big(|\bX_t|^2 + |\bY_t|^2\big) + \int_0^T |\bZ_t|^2\,dt \bigg] < \infty,
\end{align*}
where we write $\bZ = (\bZ^1,...,\bZ^n)$ with $\bZ^i = (Z^{i1},..., Z^{in})$ and $Z^{ij} \in \R^{d \times d}$, and $\bZ_t^i  d \bW_t = \sum_{j = 1}^n Z_t^{ij}  dW_t^j$.

The aim of this section is to show that optimizers of \eqref{controldist} can be characterized via a similar equation, in particular the McKean-Vlasov FBSDE 
\be \label{eq.mpdist}
\begin{cases}
\ds dX_t^i = - D_p H^i(X_t^i, n Y_t^i) dt + dW_t^i, \
\ds \\
\ds dY_t^i = -\bigg( \frac{1}{n}D_x L^i\big(X_t^i, -D_pH^i(X_t^i,nY_t^i)) + \sF^i(X_t^i, \bm m_t)\bigg) dt + Z_t^i dW_t^i, \\
X_0^i = \xi^i \sim m^i, \quad Y_T^i = \sG^i(X_T^i, \bm m_T),  \qquad \bm m_t = (\sL(X_t^1),...,\sL(X_t^n)),
\end{cases}
\ee
where we define $\sF^i, \sG^i : \R^d \times \spt^n \to \R^d$ given by
\begin{align*}
\sF^i(x,\bm m) = \langle \bm m^{-i}, D_i F \rangle(x), \quad \sG^i(x,\bm m) = \langle \bm m^{-i}, D_i G \rangle(x)
\end{align*}
and $(\xi^i)_{i = 1,...,n}$ are independent square-integrable $\sF_0$-measurable initial conditions.
A solution process $(\bX, \bY, \bZ)$, which is again implicitly assumed to be square-integrable on $[0,T] \times \Omega$, this time takes values in $(\R^d)^n \times (\R^d)^n \times (\R^{d \times d})^n$. A notable difference between \eqref{eq.mpstand} and \eqref{eq.mpdist} is that the $\bZ$ process in the latter has only $n$ instead of $n^2$ components, because the equations $i=1,\ldots,n$ are decoupled.

In order to streamline the presentation, it will be helpful to introduce an open-loop formulation of the distributed control problem, which in the end will be equivalent under Assumption \ref{assump.conv}. More precisely, treating the initial state $\bm \xi = (\xi^1,...,\xi^n)$ as fixed, we consider the problem 
\begin{align} \label{controldistol}
\inf_{\bm\alpha }   \jol(\bm\alpha),   \qquad \jol(\bm\alpha) := \E\bigg[\int_0^T \bigg(\frac{1}{n} \sum_{i = 1}^n L^i(X_t^i, \alpha_t^i) ds + F(\bX_t) \bigg)dt + G(\bX_T) \bigg], 
\end{align}
with the infimum taken over all square-integrable, adapted process $\bm\alpha=(\alpha^1,\ldots,\alpha^n)$ such that $\alpha^i$ is adapted to the  augmented filtration $\mathbb{F}^i$ generated by $W^i$ and $\xi^i$, and $X^i$ is given by 
\begin{align*}
X^i_t = \xi^i + \int_0^t \alpha_s^i ds + W_t^i.
\end{align*}
We start by showing that the FBSDE \eqref{eq.mpdist} is a necessary condition for optimality. 

\begin{proposition} \label{prop.mpnec}
Suppose that Assumption \ref{assump.conv} holds, and that $\bm\alpha=(\alpha^1,\ldots,\alpha^n)$ is a minimizer of \eqref{controldistol}. Then there is a solution $(\bX,\bY,\bZ)$ of \eqref{eq.mpdist} such that $\alpha^i_t = - D_p H(X_t^i, n Y_t^i)$ a.s.\ for a.e.\ $t$ and each $i=1,\ldots,n$.
\end{proposition}
\begin{proof}
The proof follows from the same observation which led to Lemma \ref{lem.vichar}: Let $\bm\alpha$ be optimal for \eqref{controldistol}, and let $\bX^*$ be the corresponding optimal state process. Define $\bm m_t = (\sL(X^{*,1}_t),...,\sL(X^{*,n}_t))$ and
\begin{align*}
F^i(t,\cdot) = \langle \bm m_t^{-i}, F \rangle, \quad G^i(\cdot) = \langle \bm m_T^{-i}, G \rangle.
\end{align*}
Then, for each $i$, $\alpha^i$ must be optimal for the control problem 
\begin{align} \label{olprob}
\inf_{\beta} \E\bigg[\int_0^T \bigg( \frac{1}{n} L^i(X_s, \beta_s) + F^i(s,X_s) \bigg)ds +  G^i(X_T)\bigg], 
\end{align}
subject to  $dX_t = \beta_t \, dt + dW_t^i$ and $X_0= \xi^i$.
The result then follows from the standard stochastic maximum principle, see for example \cite[Theorem 4.12]{carmonabsde}.
\end{proof}

We now show that in fact any solution of \eqref{eq.mpdist} yields an optimizer of \eqref{controldist}.

\begin{proposition} \label{prop.mpsuff}
Suppose that Assumption \ref{assump.conv} holds and that $(\bX,\bY,\bZ)$ is a solution of \eqref{eq.mpdist}. Then $\alpha_t^i = - D_pH^i(X_t^i, n Y_t^i)$ for $i=1,\ldots,n$ defines an optimizer for \eqref{controldistol}. 
\end{proposition}
\begin{proof}
Let $\bm\alpha=(\alpha^1,\ldots,\alpha^n)$ be as given in the Proposition, and let $\bm m_t=(\sL(X^1_t),\ldots,\sL(X^n_t))$. Suppose that $\bm{\overline{\alpha}} = (\overline{\alpha}^1,...,\overline{\alpha}^n)$ is any competitor, and let $\overline{\bX}$ be the corresponding state process. Note that $X^i$ and $\overline{X}^i$ are adapted to the filtration $\FF^i=(\F^i_t)_{t \in [0,T]}$ generated by $W^i$ and $\xi^i$. In particular, $X^1,\ldots,X^n$ are independent, and so
\begin{align*}
\sF^i(X^i_t,\bm m_t) &= \E[D_iF(\bX_t)\,|\,\F^i_t], \qquad Y^i_T = \sG^i(X^i_T,\bm m_T) = \E[D_iG(\bX_T)\,|\,\F^i_T].
\end{align*}
By convexity of $G$,
\begin{align*}
\E[G(\overline{\bX}_T)& - G(\bX_T)] \geq \E[DG(\bX_T) \cdot  (\overline{\bX}_T - \bX_T)] \\
&= \sum_{i = 1}^n \E[D_i G(\bX_T) \cdot   (\overline{X}_T^i - X_T^i)] \\
 &= \sum_{i = 1}^n \E[Y^i_T  \cdot  (\overline{X}_T^i - X_T^i) ] \\
 &= \sum_{i = 1}^n \E\int_0^T \bigg(Y_t^i \cdot (\overline{\alpha}_t^i - \alpha_t^i) - \frac{1}{n}(\overline{X}_t^i - X_t^i)\cdot D_x L^i(X_t^i, \alpha_t^i) -  (\overline{X}_t^i - X_t^i)\cdot \sF^i(X^i_t,\bm m_t) \bigg) dt  \\
 &= \sum_{i = 1}^n \E\int_0^T \bigg(Y_t^i \cdot (\overline{\alpha}_t^i - \alpha_t^i) - \frac{1}{n}(\overline{X}_t^i - X_t^i)\cdot D_x L^i(X_t^i, \alpha_t^i) -  (\overline{X}_t^i - X_t^i)\cdot D_iF(\bX_t) \bigg) dt .
\end{align*}
Thus 
\begin{align*}
\jol(\overline{\alpha}) - \jol(\alpha) &\geq \E \int_0^T \sum_{i=1}^n  \bigg(Y_t^i \cdot (\overline{\alpha}_t^i - \alpha_t^i) - \frac{1}{n}(\overline{X}_t^i - X_t^i)\cdot D_x L^i(X_t^i, \alpha_t^i) -  (\overline{X}_t^i - X_t^i)\cdot D_iF(\bX_t) \bigg) dt  \\ 
&\quad + \E \int_0^T \bigg( \sum_{i=1}^n \frac{1}{n} ( L^i(\overline{X}_s^i, \overline{\alpha}_s^i) -  L^i(X_s^i, \alpha_s^i) )
     +  F(\overline{\bX}_t) - F(\bX_t) \bigg) dt  \\
    & \geq 0
\end{align*}
where the last inequality uses convexity of $F$, convexity of $(x,a) \mapsto L^i(x,a)$, and the fact that $\alpha_t^i$ minimizes $a \mapsto  \tfrac{1}{n}L^i(X_t^i, a) + Y_t^i \cdot  a$. This shows that $\bm\alpha$ is optimal. 
\end{proof}

The following proposition states that \eqref{eq.mpdist} has a unique solution, and thus the maximum principle can be used to produce a solution to our distributed control problem. The proof is straightforward but tedious, and so is delayed to Appendix B.

\begin{proposition} \label{prop.mpfbsde}
Under Assumption \ref{assump.conv}, the FBSDE \eqref{eq.mpdist} has a unique solution. In particular, there is a unique optimizer for \eqref{controldistol}.
\end{proposition}

Let us mention that in fact (abusing notation slightly) the unique open loop optimizer for \eqref{controldistol} necessarily has the form $\alpha^i_t = \alpha^i(t,X_t^i)$ for some $\bm\alpha = (\alpha^1,...,\alpha^n) \in \ad$, which is thus the unique optimizer to the corresponding closed loop problem. Indeed, the open and closed loop formulations of the control problem \eqref{videf} are well-known to be equivalent, so the same argument leading to Lemma \ref{lem.vichar} shows that we must have $\alpha^i_t = - D_pH^i(X_t^i, n D v^i(t,\bX_t))$, where $v^i$ solves \eqref{vieqn}. We summarize this discussion in the following Proposition:

\begin{proposition}
\label{existunique}
Suppose Assumption \ref{assump.conv} holds. Then for any $(t,\bm m) \in [0,T] \times \spt^n$, there exists $\bm\alpha = (\alpha^1,...,\alpha^n) \in \ad$ which is optimal in the definition of $\vdt(t,\bm m)$, i.e., such that 
\begin{align*}
\vdt(t,\bm m) = \mathcal{J}(t,\bm m, \bm\alpha). 
\end{align*}
It is unique in the sense that if $\bm\beta$ is any other optimizer then, for each $i$, $\alpha^i = \beta^i$ a.e.\ on $[t,T] \times \R^d$, for each $i=1,\ldots,n$.
\end{proposition}

\begin{remark} \label{re:brownianbridge}
In the Cole-Hopf case \eqref{intro:ColeHopf}, say with $d=1$ and $\bX_0=0$ and $T=1$ for simplicity, a somewhat more concrete description of the optimal distributed control was discussed in \cite[Remark 2.15]{LacMukYeu}. For the probability measure $P$ on $\R^n$ with density proportional to $e^{-nG(\bx)}$, it is shown that there exists a unique minimizer $Q^*$ of $H(\cdot\,|\,P)$ over the set of product measures. Then, the optimal state process $\bX$ for the distributed control problem is characterized by $\bX_0=0$, $\bX_1 \sim Q^*$, and the conditional law of $(\bX_t)_{t \in [0,1]}$ given $\bX_1=x$ being the law of the Brownian bridge from $0$ to $x$. It is not clear if this description could be recovered from the PDE \eqref{vieqn} or the FBSDE \eqref{eq.mpdist}.
\end{remark}

\section{Near-optimality of distributed controls}
\label{sec.approx}

This section states and proves our most general bounds on $|\cv(t,\bm m) - \vdt(t,\bm m)|$.
These bounds will involve explicit constants depending on the regularity of the data $L^i$, $H^i$, $F$, and $G$, as well as certain concentration properties of the initial distribution $\bm m$. 
We thus begin in Section \ref{se:constants}  by introducing some terminology and notational conventions for these constants, and Section  \ref{se:mainestimates} will then state the main results in detail.

\subsection{Functional inequalities and explicit constants} \label{se:constants}

We will make frequent use of two functional inequalities satisfied by the laws of various controlled state processes, the  Poincar\'e  and transport inequalities.

For a probability measure $m$ on $\R^k$, we say that $m$ satisfies a \emph{Poincar\'e inequality} with constant $C$ if
\begin{equation} \label{poindef}
\Var_m(g) := \langle m,g^2\rangle -  \langle m, g\rangle^2 \leq C \langle m,|Dg|^2\rangle,
\end{equation}
for all bounded Lipschitz functions $g : \R^k \to \R$.
We call the smallest constant $C$ such that \eqref{poindef} holds \emph{the Poincar\'e constant} of $m$; if there is no such constant, then the Poincar\'e constant is $\infty$. We recall that convention that $\bm m = (m^1,...,m^n)$ is identified with the product measure $m^1 \otimes \cdots \otimes m^n$. Thus, when we say the Poincar\'e constant of $\bm m$, we mean the Poincar\'e constant of $m^1 \otimes \cdots \otimes m^n$, which, because Poincar\'e inequalities tensorize, is the same as the maximum of the Poincar\'e constant of the marginals $m^i$. 

 A probability measure $\mu$ on a separable metric space $(E,d)$ is said to satisfy the $T_2$ inequality with constant $c$ if
\begin{equation}
\wass_2^2(\mu,\nu) \le cH(\nu\,|\,\mu), \quad \forall \nu \in \P_2(E). \label{def:T2inequality}
\end{equation}
Here $\wass_2^2$ is the quadratic Wasserstein distance and $H$ the relative entropy, defined as usual by
\begin{align*}
\wass_2^2(\mu,\nu) &:= \inf\{ \E [d^2(X,Y)] : X \sim \mu, \ Y \sim \nu\}, \\
H(\nu\,|\,\mu) &:= \int_E \frac{d\nu}{d\mu}\log\frac{d\nu}{d\mu}\,d\mu, \ \text{if } \nu \ll \mu, \ \ \ H(\nu\,|\,\mu)=\infty \ \text{otherwise}.
\end{align*}
The $T_2$ inequality  is satisfied (with finite constant) for $\mu$ being a Dirac, a Gaussian, or any strongly log-concave measure, to name but a few examples; see \cite{gozlan2010transport} for additional information about these well-studied inequalities.

Lastly, we introduce some shorthand notation for explicit but complicated constants which appear in our main results.
Recall from Section \ref{se:notation} that $|\cdot|$ and $|\cdot|_{op}$ denote the Frobenius and operator norms, respectively.

\begin{convention}
We use $L$ or $H$ without a superscript to describe bounds which apply to $L^i$ or $H^i$ uniformly with respect to $i$. More precisely, we will use the quantities
\begin{align*}
\|D_{pp} H\|_{\infty} = \max_{i = 1,...,n} \| |D_{pp} H^i|_{\ope} \|_{\linfty(\R^d \times \R^d)}, \\
\|D_{xp} H\|_{\infty} = \max_{i = 1,...,n} \| |D_{xp} H^i|_{\ope} \|_{\linfty(\R^d \times \R^d)},  \\
\|D_{xx} L\|_{\infty} = \max_{i = 1,...,n} \| |D_{xx} L^i|_{\ope} \|_{\linfty(\R^d \times \R^d)}. 
\end{align*}
We will denote by $C_F$ and $C_G$ two constants such that the spectral bounds 
\begin{align}
0 \leq D^2 F(\bx) \leq \frac{C_F}{n} I_{nd \times nd}, \quad
0 \leq D^2 G(\bx) \leq \frac{C_G}{n} I_{nd \times nd} \label{asmp:specFG}
\end{align}
hold for all $\bx \in (\R^d)^n$. When we write $\|D_{ij} F\|_{\linf}$ or $\|D_{ij} G\|_{\linf}$, we mean the $\linf$ norm on $(\R^d)^n$ of the Frobenius norm on $\R^{d \times d}$, e.g.
\begin{align*}
\| D_{ij} G \|_{\linf}^2 = \||D_{ij} G|\|_{\linfty((\R^d)^n)}^2 = \sup_{\bx \in (\R^d)^n}  \sum_{k,l = 1}^d |D_{x^i_k x^j_l} G(\bx)|^2  .
\end{align*}
We will denote by $C_S$ the constant 
\begin{align}
C_S = C_G + T(\|D_{xx} L\|_{\infty} + C_F). \label{def:CS}
\end{align}
We denote by $C_P$ the constant 
\begin{align*} 
C_P = \frac{\exp\bigg(2T \big(\|D_{xp} H\|_{\infty} + \|D_{pp} H\|_{\infty}C_S \big) \bigg)  - 1}{2\big(\|D_{xp} H\|_{\infty} + \|D_{pp} H\|_{\infty}C_S\big)}. 
\end{align*}
Finally, for $\bm m \in \spt^n$ having Poincar\'e constant $c_P$ and satisfying a $T_2$ inequality with constant $c_{T_2}$, we denote by $C_P(\bm m)$ and $C_{T_2}(\bm m)$ the constants
\begin{align} 
C_P(\bm m) &:= C_P + c_P \exp\bigg(2T \big(\|D_{xp} H\|_{\infty} + \|D_{pp} H\|_{\infty}C_S \big) \bigg), \label{def:CP(m)} \\
C_{T_2}(\bm m)  &:= 3(c_{T_2} \wedge 2T) \exp\Big(3T(\|D_{xp} H\|_{\infty} + \|D_{pp} H\|_{\infty} C_S)^2\Big). \label{def:CWconst}
\end{align}
\end{convention}

\begin{remark}
Note that Dirac measures have Poincar\'e constant zero, so that $C_P(\delta_{x^1},\ldots,\delta_{x^n})=C_P$ for any $\bx \in (\R^d)^n$.
\end{remark}

\begin{remark}
We include the factor of $1/n$ in \eqref{asmp:specFG} so that the constants $C_F$ and $C_G$ are dimension-free in our main examples. The constant $C_S$, as we will see in Lemma \ref{lem.spectral}, gives an upper bound on the Hessian of $V$, in the sense that $D^2 V(t,\bx) \leq \frac{C_S}{n}$ for each $t$. The meaning of $C_P$ will become clear in Lemma \ref{lem.cp}, which shows that $C_P$ provides an upper bound on the Poincar\'e constant of certain diffusions with which we will be working.
\end{remark}

\subsection{Statements of main estimates on value functions and optimal state processes}
\label{se:mainestimates}

Our main estimates will be stated in terms of the following distributed state process.
For $(s,x,\bm m) \in [0,T] \times \R^d \times \spt^n$, define
\begin{align} \label{def:hatalpha}
\widehat{\alpha}^i(s,x,\bm m) = - D_p H^i\big(x, n\langle \bm m^{-i}, D_i V(s,\cdot) \rangle(x) \big).
\end{align}
where we recall that notation $\langle \bm m^{-i}, D_i V(s,\cdot) \rangle(x)$ indicates integrating over the variables $j \neq i$ with $x$ plugged into the $i^\text{th}$ variable; that is, $\langle \bm m^{-i}, D_i V(s,\cdot) \rangle(x) = \E[D_iV(s,\bm\xi)\,|\,\xi^i=x]$, for $\bm\xi=(\xi^1,\ldots,\xi^n) \sim \bm m$.
Given $(t,\bm m) \in [0,T] \times \spt^n$,  consider the McKean-Vlasov SDE
\begin{align} \label{xhatdef}
\begin{split}
d\widehat{X}_s^i &= \widehat{\alpha}^i(s,\widehat{X}^i_s, \bm m_s) \, ds + dW_s^i, \ \ s \in (t,T), \ \ i=1,\ldots,n, \\
\bm m_s &= (\sL(\widehat{X}_s^1),...,\sL(\widehat{X}_s^n)), \quad \bm m_t = \bm m.
\end{split}
\end{align}
We may write \eqref{xhatdef} more concisely as
\begin{align*}
d\widehat{X}_s^i = - D_p H^i\big(\widehat{X}_s^i, n \E[D_iV(s,\widehat{\bX}_s) \,|\, \widehat{X}^i_s] \big) \, ds + dW_s^i, \ \ s \in (t,T), \ \ i=1,\ldots,n, \ \ \widehat{\bX}_t \sim \bm m. 
\end{align*}

\begin{lemma} \label{le:Yeq-wellposed}
For $(t,\bm m) \in [0,T] \times \spt^n$, there exists a unique strong solution of the SDE \eqref{xhatdef}.
\end{lemma}

The proof is deferred to the next section.
We now state our main estimate between the (lifted) full-information and distributed value functions $\cv$ and $\vdt$, which were defined in \eqref{def:Vlift} and \eqref{controldist}.

\begin{theorem} \label{thm.est1}
Suppose Assumption \ref{assump.conv} holds.
Let $(t,\bm m) \in [0,T] \times \spt^n$, and let $\widehat{\bX}$ be the corresponding solution of \eqref{xhatdef}. We have 
\begin{align*}
0 \leq \vdt(t,\bm m) - \cv(t,\bm m) \leq \sR(t,\bm m), 
\end{align*}
where we define
\begin{align*}
\sR(t,\bm m) &= n C_t(\bm m)\sum_{1 \le i < j \le n} \bigg((T-t)\E|D_{ij} G(\widehat{\bX}_T)|^2 +    \int_t^T (s-t)\E|D_{ij} F(\widehat{\bX}_s)|^2\,ds\bigg), \\
C_t(\bm m) &= \|D_{pp} H\|_{\infty} C_P(\bm m) \exp\Big( (T-t)(1 + 2  C_S \|D_{pp} H\|_{\infty} + 2\|D_{xp} H\|_{\infty})\Big).
\end{align*}
\end{theorem}

\begin{remark}
Of course, we have $\E|D_{ij}G(\widehat{\bX}_T)|^2 \le \|D_{ij}G\|_\infty^2$, and similarly for $F$, which is in fact good enough for the applications we have in mind. 
But we prefer to state the sharpest version of the result, because in certain situations we do not appear to need boundedness of the second derivatives as was imposed in Assumption \ref{assump.conv}; see Remark \ref{re:unboundedness} for discussion of this point.
\end{remark}

In the case of quadratic costs $L^i$, we can obtain sharper constants and a particularly simple bound. We give the We state this as a corollary of the proof Theorem \ref{thm.est1}, with details given after the proof of Theorem \ref{thm.est1} below. See also Remark \ref{re:unboundedness} below for some discussion of relaxing the assumption of bounded second derivatives of $(F,G)$.

\begin{corollary} \label{co:quadratic}

Suppose Assumption \ref{assump.conv} holds, and $L^i(x,a)=|a|^2/2$ for each $i=1,\ldots,n$. Let $(t,\bm m) \in [0,T] \times \spt^n$, with $\bm m$ having Poincar\'e constant $c_0$. Let $\widehat{\bX}$ be the corresponding solution of \eqref{xhatdef}. Then
\begin{align*}
0 \leq \vdt(t,\bm m) - \cv(t,\bm m) \le n(T-t) \bigg[ &\bigg( (T-t+c_0)\sum_{1 \le i < j \le n} \E|D_{ij}G(\widehat{\bX}_T)|^2\bigg)^{1/2} \\
	&+ \int_t^T\bigg( (s-t+c_0)\sum_{1 \le i < j \le n} \E|D_{ij}F(\widehat{\bX}_s)|^2\bigg)^{1/2}\,ds\bigg]^2.
\end{align*}

\end{corollary}

Our next result shows how to approximate the optimal state process $\bX$ from the full-information problem by a distributed state process.
Recall the optimizer $\bX$ for the unconstrained control problem satisfies the SDE
\begin{align} \label{xdef2}
dX_s^i = - D_p H^i(X_s^i, n D_iV(s,\bX_s) ) ds + dW_s^i,
\end{align}
where $V$ is the full-information value function defined in \eqref{controlstandard}.
The next result shows, in a quantitative sense, that the low-dimensional marginals of $\bX$ are close to those of $\widehat{\bX}$, the latter defined in \eqref{xhatdef}.
It will apply when the given initial distribution $\bm m$ is assumed to obey a transport inequality.

\begin{theorem} \label{thm.est2}
Suppose Assumption \ref{assump.conv} holds.
Fix $(t,\bm m) \in [0,T] \times \spt^n$. Let $\bX$ denote the optimal state process as in \eqref{xdef2} initialized from $\bX_t \sim \bm m$.
Write $X^i_{[t,T]}$ for the corresponding $C([t,T];\R^d)$-valued random variable, for each $i=1,\ldots,n$.
Let $\widehat{\bX}$ be the solution \eqref{xhatdef}, initialized from $\widehat{\bX}_t \sim \bm m$.
Then, for each $k=1,\ldots,n$,
\begin{align}
\frac{1}{{n \choose k}} \sum_{S \subset [n], \, |S|=k} \wass_2^2\big(\sL((\widehat{X}^i_{[t,T]})_{i \in S}),\sL((X^i_{[t,T]})_{i \in S})\big) \le   k C_{T_2}(\bm m) \sR(t,\bm m), \label{ineq:W2bound}
\end{align}
where $\sR(t,\bm m)$ was defined in Theorem \ref{thm.est1}, and $C_{T_2}(\bm m)$ in \eqref{def:CWconst}.
\end{theorem}

Although Theorem \ref{thm.est2} is nonasymptotic, it is helpful to understand it by imagining that $n\to\infty$ and $\sR(t,\bm m)\to 0$.
The meaning of \eqref{ineq:W2bound} is that, for $k$ fixed as $n\to\infty$, ``most" $k$-state marginals of the $n$-state vector $(X^1,\ldots,X^n)$ are close to the corresponding $k$-state marginals of $\widehat{\bX}$.
In the symmetric case, when $L^i=L$ does not depend on $i$ and when $F$ and $G$ are symmetric functions of their $n$ variables, the optimal state process $\bX=(X^1,\ldots,X^n)$ is exchangeable, and so too is $\widehat{\bX}$. The inequality \eqref{ineq:W2bound} is then equivalent to
\begin{align*}
\wass_2^2\big(\sL((\widehat{X}^1,\ldots,\widehat{X}^k)_{[t,T]} ),\sL((X^1,\ldots,X^k)_{[t,T]} )\big) \le C_{T_2}(\bm m)  k \sR(t,\bm m),
\end{align*}
which implies a more traditional form of propagation of chaos, again if $\sR(t,\bm m) \to 0$.
In general, the state vector $(X^1,\ldots,X^n)$ is not exchangeable, and the bound \eqref{ineq:W2bound} instead averages over all choices of $k$ states out of the $n$.

In this section, we denote by $\bm\alpha$ and $\bX$ the optimal control and state process for the full-information problem \eqref{controlstandard}. Let $\overline{\bm\alpha}$ and $\overline{\bX}$ denote the optimal control and state process for the distributed problem \eqref{controldist}, which is unique by Proposition \ref{prop.mpfbsde}. In each case, we start from time $t=0$ and with non-random initial positions $(x^1,\ldots,x^n)$, for simplicity.

\begin{theorem} \label{thm.est3}
Let $\bm\alpha = (\alpha^1,...,\alpha^n)$  and $\overline{\bm\alpha} = (\overline{\alpha}^1,...,\overline{\alpha}^n)$ respectively denote the unique optimal controls for the full-information problem \eqref{controlstandard} and the distributed problem \eqref{controldist}, the latter being unique by Proposition \ref{prop.mpfbsde}.
Under Assumption \ref{assump.conv}, we have 
\begin{align*}
\E \int_0^T |\bm\alpha_t - \overline{\bm\alpha}_t|^2 dt  \leq  C_1 n^2\sum_{1 \le i < j \le n} \|D_{ij} F\|_{\linf}^2 + C_2 n^2 \sum_{1 \le i < j \le n} \|D_{ij} G\|_{\linf}^2,
\end{align*}
where $C_1 = C_PT^3/2C_L^2$ and $C_2= C_P T/C_L^2$.
\end{theorem}

We note that Theorem \ref{thm.est3} easily implies a corresponding estimate between the state processes: Let $\bm\bX$  and $\overline{\bm\bX}$ respectively denote the optimal state processes for the full-information and distributed problems. Then
\begin{align*}
 \E\Big[\sup_{0 \leq t \leq T} |\bX_t-\overline{\bX}_t|^2\Big] \leq 
TC_1 n^2 \sum_{i \neq j} \|D_{ij} F\|_{\linf}^2 + TC_2 n^2 \sum_{i \neq j} \|D_{ij} G\|_{\linf}^2. 
\end{align*}

The rest of the section is devoted to the proofs of Theorems \ref{thm.est1}, \ref{thm.est2}, and \ref{thm.est3}, following some preparations related to estimates on the value function $V$ and its lift $\cv$, as well as Poincar\'e inequalities for some relevant controlled state processes.

\subsection{Spectral bounds on value functions}

As a first preparation, we derive bounds on the Hessian of the value function $V$ of the full-information control problem, defined in \eqref{controlstandard}. The following lemma shows that convexity and semi-concavity of $V$ can be efficiently deduced in terms of the convexity and semi-concavity of the data $L^i$, $F$, and $G$. 

\begin{lemma} \label{lem.spectral}
Suppose that Assumption \ref{assump.conv} holds, and recall the definition of $C_S$ from \eqref{def:CS}. Then for each $0 \leq t \leq T$ and $\bx \in (\R^d)^n$, $V(t,\cdot)$ is twice differentiable with
\begin{align*}
0 \leq D^2 V(t,\bx) \leq \frac{C_S}{n} I_{nd \times nd}.
\end{align*}
\end{lemma}
\begin{proof}[Proof of Lemma \ref{lem.spectral}]
We will use the fact that under Assumption \ref{assump.conv}, the control problem \eqref{controlstandard} is equivalent when posed over open-loop controls. More precisely, we have 
\begin{align} \label{ol}
V(t,\bx) = \inf_{\bm\alpha = (\alpha^1,...,\alpha^n)} \E\bigg[\int_t^T \bigg(\frac{1}{n} \sum_{i = 1}^n L^i(X_s^i, \alpha_s^i) + F(\bX_s) \bigg) ds + G(\bX_T)\bigg]
\end{align}
where the infimum is taken over \textit{open-loop} controls, i.e. square integrable $\mathbb{F}$-adapted $(\R^d)^n$-valued processes $\alpha = (\alpha_s)_{t \leq s \leq T}$ and $\bX = (X^1,...,X^n)$ is given by
\begin{align} \label{olxdef}
X_s^i = x^i + \int_t^s \alpha_u^i du + (W_s^i - W_t^i), \quad t \leq s \leq T.
\end{align}
We will also use the fact that a $C^2$ function $g$ on a Euclidean space satisfies $D^2g(x) \le C I_{nd \times nd}$ for all $x$ if and only if
\begin{align*}
g(ry + (1-r)z) \geq rg(y) + (1-r)g(z) - \frac{C}{2} r(1-r) |y-z|^2, \quad \forall  x,y,z, \,\, \forall r \in (0,1).
\end{align*}

For convexity of $V(t,\cdot)$, we refer to (the proof of) Lemma 10.6 of \cite{flemingsoner}. The upper bound on $D^2V(t,\cdot)$ is also proved by a simple control-theoretic argument, as in Lemma 9.1 of \cite{flemingsoner}, but in order to track the constants explicitly we provide a full proof. Fix $t \in [0,T]$ as well as $\bx,\by,\bz \in (\R^d)^n$ such that $\bx = r \by + (1-r) \bz$ for some $r \in (0,1)$. Let $\bm\alpha$ denote an optimizer in \eqref{ol}. Define $\bX$ by \eqref{olxdef} and $\bY$, $\bZ$ by 
\begin{align*}
Y_s^i = y^i + \int_t^s \alpha_u^i du + (W_s^i - W_t^i), \quad t \leq s \leq T, \\
Z_s^i = z^i + \int_t^s \alpha_u^i du + (W_s^i - W_t^i), \quad t \leq s \leq T.
\end{align*}
Use the relations $\bX = r \bY + (1-r) \bZ$ and $\bY - \bZ = \by - \bz$ and the optimality of $\alpha$ to deduce
\begin{align*}
V(t,\bx) &= \E\bigg[\int_t^T \bigg(\frac{1}{n} \sum_{i = 1}^n L^i(X_s^i, \alpha_s^i) + F(\bX_s) \bigg) ds + G(\bX_T)\bigg] \\
&\geq \E\bigg[\int_t^T \bigg(\frac{1}{n} \sum_{i = 1}^n \big(rL^i(Y_s^i, \alpha_s^i) + (1-r)L^i(Z_s^i, \alpha_s^i)\big)  - \frac{\|D_{xx} L\|_{\infty}}{2n}r(1-r) |\by - \bz|^2 
\\
& \qquad\quad + rF(\bY_s) + (1-r) F(\bZ_s) - \frac{C_F}{2n} r(1-r) |\by - \bz|^2 \bigg) ds \\
&\qquad\quad + rG(\bY_T) + (1-r)G(\bZ_T) - \frac{C_G}{2n} r(1-r) |\by - \bz|^2 \bigg] \\
&\geq rV(t,\by) + (1-r) V(t,\bz) - \frac{C_S}{2n}r(1-r) |\by - \bz|^2, 
\end{align*}
and thus the claimed estimate holds.  
\end{proof}

As quick corollary of  Lemma \ref{lem.spectral}, we may prove Lemma \ref{le:Yeq-wellposed} by appealing to known results on Lipschitz McKean-Vlasov equations:

\begin{proof}[Proof of Lemma \ref{le:Yeq-wellposed}]
Lemma \ref{lem.spectral} implies that $D_i V(t,x)$ is Lipschitz in $x$, uniformly in $t$. Moreover, $D_{pp} H^i$ is bounded by assumption, so the map
\begin{align*}
[0,T] \times \R^d \times \P_2((\R^d)^n) \ni (t,x,m) \mapsto - D_p H^i(x, n \langle m^{-i}, D_i V(s,\cdot) \rangle(x) \big)
\end{align*}
is Lipschitz and of linear growth in $(x,m)$, uniformly in $t$, with the measure argument given the quadratic Wasserstein distance. Here we write
\begin{align*}
\langle m^{-i}, D_i V(s,\cdot) \rangle(x) := \int_{(\R^d)^n} D_iV(s,y^1,\ldots,y^{i-1},x,y^{i+1},\ldots,y^n) m(d\by)
\end{align*}
for the integral over the coordinates $j \neq i$, with $x$ plugged into the $i^\text{th}$ argument.
Hence, the McKean-Vlasov SDE
\begin{align*}
    d\widehat{X}_s^i &= - D_p H^i(\widehat{X}_s^i, n\langle \bm m_s^{-i}, D_i V(s,\cdot) \rangle(\widehat{X}^i_s)) ds + dW_s^i, \quad \bm m_s = \big(\sL(\widehat{X}_s^1),\ldots, \sL(\widehat{X}_s^n)\big),
\end{align*}
is uniquely solvable from any initial law with finite second moment; see, e.g., \cite[Theorem 1.7]{carmonabsde}. Because the $i^\text{th}$ equation depends only on the $i^\text{th}$ variable $\widehat{X}^i_s$, we deduce that $(\widehat{X}^1_s,\ldots,\widehat{X}^n_s)$ must be independent for each $s \in (t,T]$ if the time-$t$ positions are independent. This proves the claimed well-posedness. 
\end{proof}

We can obtain similar bounds for any optimizer of the distributed control problem by using the representation given by \eqref{vieqn}, described in Section \ref{sec:dstr-opt-necessary}.
This is summarized by the following lemma, which follows from the discussion of Section \ref{sec:dstr-opt-necessary} combined with the same argument appearing in the proof of Lemma \ref{lem.spectral}.

\begin{lemma} \label{lem.vichar}
Suppose that Assumption \ref{assump.conv} holds.
Suppose that $\bm\alpha \in \ad$ is optimal in the definition of $\vdt$ from \eqref{controldist}. Then, using the notation of Section \ref{sec:dstr-opt-necessary}, we have
\begin{align*}
\alpha^i(t,x) = - D_p H^i(x, nDv^i(t,x)), 
\end{align*}
where $v^i$ is the unique classical solution of \eqref{vieqn}, which satisfies 
\begin{align*}
D^2 v^i(t,x) \leq \frac{C_S}{n} I_{d \times d}. 
\end{align*}
\end{lemma}

\subsection{Functional inequalities for diffusions}

For $b = b(t,x) : [0,T] \times \R^k \to \R^k$, consider a diffusion $\bX$ defined by
\begin{align} \label{sdepoin}
dX_t = b(t,X_t) dt + dW_t, \quad X_0 \sim m_0 \in \mathcal{P}_2(\R^k),   
\end{align}
set on some filtered probability space hosting a $k$-dimensional Brownian motion $W$ and
an independent random vector $X_0$. Set $m_t = \sL(X_t)$. The following lemma states known Poincar\'e and $T_2$ inequalities for this process $X$.

\begin{lemma} \label{lem.poincare}
Suppose that $b$ satisfies 
\begin{itemize}
\item $b$ is jointly continuous and $x \mapsto b(t,x)$ is $L$-Lipschitz and of linear growth, uniformly in $t$.
\item $\big(b(t,x) - b(t,\bar{x})\big) \cdot  (x - \bar{x}) \leq \gamma |x- \bar{x}|^2$ for some $\gamma \in \R$ and all $x, \bar{x} \in \R^d$, $t \in [0,T]$. 
\end{itemize}
\begin{enumerate}[(i)]
\item If $m_0$ satisfies the Poincar\'e inequality with constant $c_0$, then, for $0 \leq t \leq T$, $m_t$ satisfies a Poincar\'e inequality with constant 
\begin{align*}
    \frac{e^{2\gamma t} - 1}{2 \gamma} + c_0 e^{2 \gamma t},
\end{align*} 
\item If $m_0$ satisfies the $T_2$ inequality with constant $c_0$, then $\sL(\bX_{[0,T]}) \in \P(C([0,T];\R^k))$ satisfies a $T_2$ inequality with constant 
\begin{align*}
3(c_0 \wedge 2T) e^{3TL^2}.
\end{align*}
\end{enumerate}
\end{lemma}

These results are fairly well-known. See \cite[Proposition C.1]{lackerhierarchies} for this precise form of (ii). Part (i) is shown in \cite[Theorem 4.2]{cattiaux-guillin} under stronger regularity assumptions on $b$, which are easily relaxed by an approximation argument after noting that the Poincar\'e inequality is preserved under weak convergence.
Note that the second assumption in  Lemma \ref{lem.poincare} follows from the first, with $\gamma \ge -L^2$, but we state it separately in order to emphasize that the Poincar\'e constant depends only on $\gamma$, not $L$. If we are interested only in the time-$t$ laws, instead of the path-space law as in (ii), then a $T_2$ inequality can be derived for $m_t$ with a constant depending only on $\gamma$, not $L$. The constant, however, is somewhat messy; see \cite[Proposition 2.19]{cattiaux-guillin}. In our applications, we have access to finite Lispchitz constants, so we favor the stronger path-space inequality.

We next apply Lemma \ref{lem.poincare} to estimate the Poincar\'e constant for two diffusions related to our control problems.

\begin{lemma} \label{lem.cp}
Fix $t \in [0,T]$ and $\bm m \in \spt^n$. Consider the three control/state pairs:
\begin{itemize}
\item Let $\bm\alpha=(\alpha^1,\ldots,\alpha^n) \in \sA$ be optimal for $\cv(t,\bm m)$ as defined in \eqref{def:Vlift-control}, with corresponding state process $\bX$ given as in \eqref{xdef2}.
\item Let $\widehat{\bm\alpha}$ and $\widehat{\bX}$ be as in \eqref{def:hatalpha} and \eqref{xhatdef}.
\item Let $\overline{\bm\alpha}$ be optimal for $\vdt(t,\bm m)$ as defined in \eqref{controldist}, and let  $\overline{\bX} =(\overline{X}^1,\ldots,\overline{X}^n)$ satisfy
\begin{align*}
&d\overline{X}_s^i = \overline{\alpha}^i(s,\overline{X}_s^i) ds + dW_s^i, \ \ s \in (t,T], \ \  \overline{\bX}_t \sim \bm m.
\end{align*}
\end{itemize}
Then the following hold:
\begin{enumerate}[(i)]
\item For each $t \leq s \leq T$, the measures $\sL(\bX_s)$, $\sL(\widehat{\bX}_s)$, and $\sL(\overline{\bX}_s)$ each satisfy a Poincar\'e inequality with constant $C_P(\bm m)$ defined in \eqref{def:CP(m)}.
\item If $\bm m$ satisfies a $T_2$ inequality with constant $c_0$, then the measures $\sL(\bX_{[t,T]})$, $\sL(\widehat{\bX}_{[t,T]})$, and $\sL(\overline{\bX}_{[t,T]})$ on $C([t,T];(\R^d)^n)$ satisfy a $T_2$ inequality with constant $C_{T_2}(\bm m)$ as defined in \eqref{def:CWconst}.
\end{enumerate}
\end{lemma}
\begin{proof}
We simply compute the derivative of the drift in each case. The optimal control $\bm\alpha$ for the full-information problem is given by
\begin{align*}
\alpha^i(t,\bx) = -D_pH^i(x^i,nD_iV(t,\bx)).
\end{align*}
Its derivatives are given by
\begin{align*}
D_j\alpha^i(t,\bx) = -1_{i=j}D_{xp} H^i\big(x^i, nD_iV(t,\bx)\big) - nD_{pp} H^i\big(x^i, nD_iV(t,\bx)\big)D_{ij}V(t,\bx).
\end{align*}
From Lemma \ref{lem.spectral} we have $0 \le D^2V \leq \frac{C_S}{n}$, which implies 
\begin{align*}
\|D \bm\alpha\|_{\infty}  \leq \|D_{xp} H\|_{\infty} + \|D_{pp} H\|_{\infty} C_S,
\end{align*}
The results for $\bX$ now follow from Lemma \ref{lem.poincare}.
Similarly, for $x \in \R^d$, we have
\begin{align*}
D \widehat{\alpha}^i(s,x, \bm m) = \ &- D_{xp} H^i\big(x, n\langle \bm m^{-i}, D_i V(s,\cdot) \rangle(x)\big) \\
	&- n D_{pp} H^i\big(x, n\langle \bm m^{-i}, D_i V(s,\cdot)\rangle(x) \big) \langle \bm m^{-i}, D_{ii} V(s,\cdot) \rangle(x).
\end{align*}
Again using $0 \le D^2V \leq \frac{C_S}{n}$, we get
\begin{align*}
 \|D \widehat\alpha^i\|_{\infty}  \leq \|D_{xp} H\|_{\infty} + \|D_{pp} H\|_{\infty} C_S,
\end{align*}
for each $i=1,\ldots,n$, and the results for $\widehat{\bX}$ follow from Lemma \ref{lem.poincare}. For $\overline{\bX}$, the proof is the same, but with Lemma \ref{lem.vichar} taking the place of Lemma \ref{lem.spectral}. 
\end{proof}

\subsection{Analysis of the lift of $V$}

Recall that the value function $\cv : [0,T] \times \spt^n \to \R$ was defined by \ref{def:Vlift} and satisfies $\cv(t,\bm m) = \langle \bm m, V(t,\cdot) \rangle$.
Thus $\cv$ inherits differentiability from $V$. We have the explicit formulas 
\begin{align} \label{dmi}
D_{m^i} \cv(t,\bm m,y) = \langle \bm m^{-i}, D_i V \rangle(y), \quad 
D_{y}D_{m^i} \cv(t,\bm m,y) = \langle \bm m^{-i}, D_{ii} V \rangle(y),
\end{align}
for $y \in \R^d$. In particular, this reveals the following alternative expression for the controls $\widehat{\alpha}^i$ defined in \eqref{def:hatalpha}:
\begin{equation}
\widehat{\alpha}^i(s,x,\bm m) = - D_p H^i\big(x, n D_{m^i} \cv(t,\bm m,x) \big). \label{eq:hatalpha-alt}
\end{equation}
The following Lemma states that $\cv$ in fact solves an equation similar to \eqref{eq.hjbdist}, recalling also the abbreviations defined immediately thereafter.

\begin{lemma} \label{lem.vlifteqn}
Suppose that Assumption \ref{assump.conv} holds.
For $t \in (0,T)$ and $\bm m \in \spt^n$, the lift $\cv$ satisfies
\begin{align*}
- \partial_t \cv - \frac{1}{2} \sum_{i=1}^n \langle m^i, \tr( D_{y} D_{m^i} \cv) \rangle + \frac{1}{n} \sum_{i = 1}^n \langle m^i, H^i(\cdot, nD_{m^i} \cv ) \rangle = \langle \bm m, F \rangle - E(t,\bm m), 
\end{align*}
where
\begin{align*}
E(t,\bm m) &:= \frac{1}{n} \sum_{i = 1}^n\Big( \langle \bm m,  H^i(\cdot, nD_iV) \rangle  -   \langle m^i, H^i\big(\cdot, n\langle \bm m^{-i}, D_i V \rangle \big) \rangle\Big) \\
&= \frac{1}{n} \sum_{i = 1}^n \E\Big[H^i\big(\xi^i, n D_i V(t,\bm\xi)\big) - H^i\big(\xi^i, n \E[D_i V(t,\bm\xi) | \xi^i]\big) \Big], 
\end{align*}
for $\bm\xi = (\xi^1,...,\xi^n) \sim \bm m$.
Moreover, the error term $E$ satisfies 
\begin{align}
&0 \le E(t,\bm m) \leq  \|D_{pp} H\|_{\infty} E_Q(t,\bm m) 
\end{align} 
where 
\begin{align} 
\label{eqdef}
E_Q(t,\bm m) &\coloneqq \frac{n}{2}\sum_{i = 1}^n \bigg( \langle \bm m, |D^i V|^2 \rangle  - \langle m^i, \big|\langle \bm m^{-i}, D_i V \rangle \big|^2 \rangle \bigg) \nonumber \\
&= \frac{n}{2} \sum_{i = 1}^n \E\big[ |D_i V(t,\bm\xi)|^2 - |\E[D_i V(t,\bm\xi) | \xi^i]|^2 \big]
\end{align}
\end{lemma}
\begin{proof}
The equation for $\cv$ is obtained by integrating the equation \eqref{eq.hjbclassic} for $V$ against $\bm m=m^1 \otimes ... \otimes m^n$ and then applying the identities in \eqref{dmi}. To obtain the upper bound for $E$ we use the fact that under Assumption \ref{assump.conv} 
\begin{align*}
p \mapsto \frac{\|D_{pp} H\|_{\infty} n^2}{2} |p|^2 - H^i(x,np)
\end{align*}
is convex for each $i$, so that Jensen's inequality gives
\begin{align*}
\frac{\|D_{pp} H\|_{\infty} n^2}{2}\big|& \E[D_i V(t,\xi)\,|\,\xi^i] \big|^2  - H^i\big(\xi^i, n\E[D_i V(t,\bm\xi) \,|\, \xi^i] \big) \\
&\leq \frac{\|D_{pp} H\|_{\infty} n^2}{2} \E\big[|D_i V(t,\bm\xi)|^2 \,|\, \xi^i\big] -\E\big[ H^i\big(\xi^i, n   D_i V(t,\bm\xi) \big) \,|\, \xi^i\big].
\end{align*}
Rearranging and using convexity of $H^i$ we have 
\begin{align*}
0 &\leq \E\big[ H^i\big(\xi^i, n   D_i V(t,\bm\xi) \big) \,|\, \xi^i\big] - H^i\big(\xi^i, n\E[D_i V(t,\bm\xi) \,|\, \xi^i] \big) \\
&\leq \frac{\|D_{pp} H\|_{\infty} n^2}{2} \Big(\E\big[|D_i V(t,\bm\xi)|^2 \,|\, \xi^i\big] - \big| \E[D_i V(t,\xi)\,|\,\xi^i] \big|^2\Big).
\end{align*}
Take expectations and then sum over $i=1,\ldots,n$ to get the desired bound.
\end{proof}

\begin{lemma} \label{lem.suffest}
Suppose that Assumption \ref{assump.conv} holds.
For $(t,\bm m) \in [0,T) \times \spt^n $, we have 
\begin{align*}
0 \leq \vdt(t,\bm m) - \cv(t,\bm m) \leq  \|D_{pp} H\|_{\infty} \int_t^T E_Q(s,\bm m_s)ds, 
\end{align*}
where $E_Q$ is given by \eqref{eqdef}, and $\bm m_s := (\sL(\widehat{X}_s^1),...,\sL(\widehat{X}_s^n))$ is the law of the solution $\widehat{\bX}$  of \eqref{xhatdef}.
\end{lemma}
\begin{proof}
First, we need to check that the lift $\cv$ of $V$ is regular enough to apply the verification result Proposition \ref{prop.verification}. The explicit formula for $D_{m^i} \cv$ appearing in \eqref{dmi} together with the fact that $|D V(t,\bx)| \leq C(1 + |\bx|)$ for some constant $C$ gives the estimate \eqref{itosuff}. The McKean-Vlasov SDE \eqref{xhatdef} is well-posed by Lemma \ref{le:Yeq-wellposed}.

Thus we can indeed apply Lemma \ref{lem.vlifteqn} and the verification result Proposition \ref{prop.verification}, with $\sF(\bm m) := \langle \bm m,F\rangle - E(t,\bm m)$ and $\sG(\bm m) = \langle \bm m,G\rangle$, to get
\begin{align*} 
    \cv(t,\bm m) = \inf_{\bm\alpha \in \ad} \E\bigg[\int_t^T \bigg(\frac{1}{n} \sum_{i = 1}^n L^i(X_s^i, \alpha^i(s,X_s^i)) + F(\bX_s) - E(s, \sL(X_s^1),...,\sL(X_s^n)) \bigg) ds + G(\bX_T)\bigg], 
\end{align*}
where $\bX$ is given by $dX_s^i = \alpha^i(s,X_s^i) ds + dW_s^i$, with $\bX_t \sim \bm m$.
Moreover, we deduce from the optimality criterion in Proposition \ref{prop.verification} and the formula \eqref{eq:hatalpha-alt} that the control $(\widehat{\alpha}^1,\ldots,\widehat{\alpha}^n) \in \ad$ defined in \eqref{def:hatalpha} attains this infimum. In particular,
\begin{align*}
\cv(t,\bm m) &=  \E\bigg[\int_t^T \bigg(\frac{1}{n} \sum_{i = 1}^n L^i(\widehat{X}_s^i, \widehat\alpha^i(s,\widehat{X}_s^i)) + F(\widehat{\bX}_s) - E(s, \bm m_s) \bigg) ds + G(\widehat{\bX}_T)\bigg] \\
&\geq \vdt(t,\bm m) - \int_t^T E(s, \bm m_s) ds, 
\end{align*}
where $\widehat{\bX}$ solves the claimed McKean-Vlasov SDE  and $\bm m_s = (\sL(\widehat{X}_s^1),...,\sL(\widehat{X}_s^n))$.
\end{proof}

\subsection{Proof of Theorem \ref{thm.est1}}
We now begin the main line of the proof of Theorem \ref{thm.est1}.
By Lemma \ref{lem.suffest}, we can focus on estimating the quantity
$\int_t^T E_Q(s, \bm m_s) ds$, with $\bm m_s = (\sL(\widehat{X}_s^1),...,\sL(\widehat{X}_s^1))$, and with $\widehat{\bX}$ being the solution to the McKean-Vlasov SDE \eqref{xhatdef}. To simplify notation, we abbreviate
\begin{align}
\widehat{\alpha}^{i}(s,x) &= \widehat{\alpha}^{i}(s,x,\bm m_s) =  - D_p H^i(x, nD_{m^i} \cv(t, \bm m_s, x)), \quad (s,x) \in [t,T] \times \R^d, \\
\alpha^{i}(s,\bx) &= - D_p H^i(x^i, n D_i V(s,\bx)), \quad (s,\bx) \in [t,T] \times (\R^d)^n.
\end{align}
So $\alpha \in \sA$ is the optimizer in the full-information control problem, and $\widehat{\alpha} \in \ad$ is the control corresponding to $\widehat{\bX}$.

Our goal is now to estimate
\begin{align*}
E_Q(s, \bm m_s) = \frac{n}{2} \sum_{i = 1}^n \E\Big[ |D_i V(s,\widehat{\bX}_s)|^2 - \big|\E[D_i V(s,\widehat{\bX}_s)\,|\,\widehat{X}^i_s]\big|^2\Big].  
\end{align*}
We will derive our estimate first under the additional assumption that $DV$ is $C^{1,2}$, which holds when in addition to Assumption \ref{assump.conv} we have $G \in C^{3,\alpha}_{\text{loc}}$ (this can be inferred from the interior Schauder estimates, see for example Theorem 9 in Chaper 3.4 of \cite{friedman}). We will then explain how to remove this additional assumption with a mollification procedure.

Our strategy is to compute the differential of the It\^o process $|D_i V(s,\widehat{\bX}_s)|^2 - |\E[D_i V(s,\widehat{\bX}_s)\,|\,\widehat{X}^i_s]|^2$. First, we differentiate the PDE \eqref{eq.hjbclassic} to identify a PDE satisfied by $D_i V$: 
\begin{align*} 
- \partial_t D_i V - D_i \Delta  V + \sum_{j = 1}^n D_{ij} VD_p H^j(x^j, nD_j V)  + \frac{1}{n} D_x H^i(x^i, nD_i V) = D_i F .
\end{align*}
Use Ito's formula followed by this PDE, recalling the definitions of $\alpha^i$ and $\widehat{\alpha}^i$ above, to get
\begin{align} \label{ddiv}
d D_i V(s,\widehat{\bX}_s) &= \bigg(\sum_{j = 1}^n D_{ij} V(s,\widehat{X}^j_s) \big(\widehat{\alpha}^j(s,\widehat{X}^j_s) - \alpha^j(s,\widehat{\bX}_s)\big) + \frac{1}{n} D_x H^i(\widehat{X}^i_s, n D_i V(s,\widehat{\bX}_s)) \nonumber \\
	&\qquad \  - D_i F(s, \widehat{\bX}_s) \bigg) ds   + \sum_{j = 1}^n D_{ij} V(s,\widehat{\bX}_s) dW_s^j. 
\end{align}
Let $(\F^i_s)_{s \in [t,T]}$ be the filtration generated by the process $(\widehat{X}^i_t,W^i_s)_{s \in [t,T]}$. By independence of $(\widehat{X}^j)_{j \neq i}$ and $\F^i_s$, and by $\F^i_s$-measurability of $\widehat{X}^i_s$, we have $\E[D_iV(s,\widehat{\bX}_s)\,|\,\widehat{X}^i_s] = \E[D_iV(s,\widehat{\bX}_s)\,|\,\F^i_s]$ for $s \in [t,T]$. Taking conditional expectations with respect to $\F^i_s$ in \eqref{ddiv}, we find
\begin{align} \label{ddivcond}
d\E [D_iV(s,\widehat{\bX}_s)\,|\,\F^i_t]  
&= \E\bigg[\sum_{j = 1}^n D_{ij} V(s,\widehat{\bX}_s) \big(\widehat{\alpha}^j(s,\widehat{X}^j_s) - \alpha^j(s,\widehat{\bX}_s)\big) + \frac{1}{n} D_x H^i(\widehat{X}^i_s, n D_i V(s,\widehat{\bX}_s)) \nonumber  \\
 &\qquad \ \ - D_i F(s, \widehat{\bX}_s)  \ \bigg| \ \widehat{X}_s^i \bigg] ds  + \E[D_{ii} V(s,\widehat{\bX}_s) \,|\, \widehat{X}_s^i] dW_s^i. 
\end{align}
with the stochastic integrals for $(W^j)_{j \neq i}$ all vanishing by independence with $\F^i_s$. Using It\^o's formula, we can now use \eqref{ddiv} to compute 
\begin{align*} 
d |D_iV(s,\widehat{\bX}_s)|^2 &= \bigg(2\sum_{j = 1}^n D_iV(s,\widehat{\bX}_s)^\top D_{ij} V(s,\widehat{\bX}_s) \big(\widehat{\alpha}^j(s,\widehat{X}^j_s) - \alpha^j(s,\widehat{\bX}_s)\big)   \\
&\quad \quad + \frac{2}{n} D_iV(s,\widehat{\bX}_s) \cdot D_x H^i(\widehat{X}^i_s, n D_i V(s,\widehat{\bX}_s))   \\
&\quad  \quad - 2 D_i V(s,\widehat{\bX}_s) \cdot D_i F(s, \widehat{\bX}_s) + \sum_{j = 1}^n |D_{ij} V(s,\widehat{\bX}_s)|^2 \bigg) ds + dM^i_s,  
\end{align*}
with $M^i$ being a martingale. Similarly, we use \eqref{ddivcond} to find
\begin{align*}  
d |\E[D_iV(s,\widehat{\bX}_s) | \widehat{X}_s^i]|^2 &= \bigg(2\sum_{j = 1}^n\E[D_iV(s,\widehat{\bX}_s) | \widehat{X}_s^i]^\top \E[D_{ij} V(s,\widehat{\bX}_s) \big(\widehat{\alpha}^j(s,\widehat{X}^j_s) - \alpha^j(s,\widehat{\bX}_s)\big) | \widehat{X}_s^i]  \\
&\quad \quad + \frac{2}{n} \E[D_iV(s,\widehat{\bX}_s) | \widehat{X}_s^i] \cdot \E[D_x H^i(\widehat{X}^i_s, n D_i V(s,\widehat{\bX}_s)) | \widehat{X}_s^i]  \\
&\quad \quad - 2 \E[D_iV(s,\widehat{\bX}_s) | \widehat{X}_s^i] \cdot \E[D_i F(s, \widehat{\bX}_s) | \widehat{X}_t^i] + |\E[D_{ii} V(s,\widehat{\bX}_s) | \widehat{X}_s^i]|^2 \bigg) ds + dN^i_s,
\end{align*}
with $N^i$ being another martingale.

Taking expectations in the previous two equations allows us to compute $\frac{d}{ds} \E\big[ |D_i V(s,\widehat{\bX}_s)|^2 - \E[D_i V(s,\widehat{\bX}_s)]^2\big]$. Summing over $i$ and multiplying by $n/2$ leads to 
\begin{align*}
\frac{d}{ds} E_Q(s, \bm m_s) = A_1+A_2+A_3+A_4,  
\end{align*}
where we define
\begin{align*}
A_1 &\coloneqq n \E\sum_{i,j=1}^n \big(D_i V(s,\widehat{\bX}_s) - \E[D_i V(s,\widehat{\bX}_s) | \widehat{X}_s^i]\big)^\top D_{ij} V(s,\widehat{\bX}_s) \big(\widehat{\alpha}^j(s,\widehat{X}_s^j) - \alpha^j(s,\widehat{\bX}_s) \big)  , \\
A_2 &\coloneqq \E \sum_{i = 1}^n \big(D_i V(s,\widehat{\bX}_s) - \E[D_i V(s,\widehat{\bX}_s) | \widehat{X}_s^i]\big) \cdot \big(D_x H^i(\widehat{X}_s^i, n D_i V(s,\widehat{\bX}_s) - \E[D_x H^i(\widehat{X}_s^i, n D_i V(s,\widehat{\bX}_s)|\widehat{X}_s^i] \big), \\
A_3 &\coloneqq - n \E\sum_{i = 1}^n \big(D_i V(s,\widehat{\bX}_s) - \E[D_i V(s,\widehat{\bX}_s) | \widehat{X}_s^i]\big) \cdot \big(D_i F(\widehat{\bX}_s) - \E[D_i F(\widehat{\bX}_s) | \widehat{X}_s^i] \big), \\
A_4 &\coloneqq \frac{n}{2} \E\bigg[\sum_{i,j = 1}^n |D_{ij} V(s,\widehat{\bX}_s)|^2 - \sum_{i = 1}^n \big|\E[D_{ii} V(s,\widehat{\bX}_s)  | \widehat{X}_s^i]\big|^2 \bigg].
\end{align*}
Our goal is to bound $\frac{d}{ds} E_Q(s, \bm m_s)$ from below. We handle each term separately.

\vskip.2cm

\noindent\textbf{First term:}
To estimate $A_1$, we use 
\begin{align*}
|\widehat{\alpha}^j(s,\widehat{X}^j_s) - \alpha^j(s,\widehat{\bX}_s)| \leq n \|D_{pp} H\|_{\infty} |D_i V(s,\widehat{\bX}_s) - \E[D_i V(s,\widehat{\bX}_s) | \widehat{X}_s^i]|
\end{align*}
together with the bound $0 \le D^2 V \le C_S/n$ to get
\begin{align*} 
|A_1| \leq nC_S \|D_{pp} H\|_{\infty} \E \sum_{i=1}^n \big|D_i V(s,\widehat{\bX}_s) - \E[D_iV(s,\widehat{\bX}_s) | \widehat{X}_s^i] \big|^2  = 2C_S \|D_{pp} H\|_{\infty} E_Q(s,\bm m_s). 
\end{align*}

\noindent\textbf{Second term:}
For $A_2$, we note that $D_xH^i$ is Lipschitz in its second argument:
\begin{align*}
\E\Big|&\big(D_i V(s,\widehat{\bX}_s) - \E[D_i V(s,\widehat{\bX}_s) | \widehat{X}_s^i]\big) \cdot \big(D_x H^i(\widehat{X}_s^i, n D_i V(s,\widehat{\bX}_s) - \E[D_x H^i(\widehat{X}_s^i, n D_i V(s,\widehat{\bX}_s)|\widehat{X}_s^i] \big)\Big| \\
&\leq n\|D_{xp}H\|_{\infty} \E \big|D_i V(s,\widehat{\bX}_s) - \E[D_i V(s,\widehat{\bX}_s) | \widehat{X}_s^i] \big|^2.
\end{align*}
Sum over $i=1,\ldots,n$ to get
\begin{align*}
|A_2| \le 2\|D_{xp}H\|_{\infty}E_Q(s,\bm m_s).
\end{align*}


\noindent\textbf{Third term:}
For $A_3$, we use Young's inequality to get
\begin{align*}
|A_3| &\leq \frac{n}{2} \sum_{i = 1}^n \E\big|D_i V(s,\widehat{\bX}_s) - \E[D_i V(s,\widehat{\bX}_s) | \widehat{X}_s^i] \big|^2  + \frac{n}{2} \sum_{i = 1}^n \E\big|D_i F(s,\widehat{\bX}_s) - \E[D_i F(s,\widehat{\bX}_s) | \widehat{X}_s^i]\big|^2.
\end{align*}
For the second term, use the Poincar\'e's inequality (Lemma \ref{lem.poincare}) satisfied by $(\widehat{X}^j_s)_{j \neq i}$, which are independent of $\widehat{X}^i_s$, to get
\begin{align*}
\E\Big[\big|D_i F(s,\widehat{\bX}_s) - \E[D_i F(s,\widehat{\bX}_s) | \widehat{X}_s^i]\big|^2 \,\Big|\, \widehat{X}^i_s\Big] = C_P(\bm m)\sum_{j \neq i} \E\big[|D_{ij} F(\widehat{\bX}_s)|^2 \,|\,\widehat{X}^i_s\big],
\end{align*}
for each $i$.
Hence, by symmetry,
\begin{align}  \label{a3est}
|A_3| &\leq E_Q(s,\bm m_s) + n C_P(\bm m) \E\sum_{1 \le i < j \le n} |D_{ij}F(\widehat{\bX}_s)|.
\end{align}

\noindent\textbf{Fourth term:} We clearly have $A_4 \ge 0$, because $\E|D_{ii} V(s,\widehat{\bX}_s)|^2 \ge \E|\E[D_{ii} V(s,\widehat{\bX}_s)  | \widehat{X}_s^i]|^2$.

{\ }

Combining the above calculations, we find that, if $\widetilde{C} := 2C_S \|D_{pp} H\|_{\infty} + 2 \|D_{xp} H\|_{\infty} + 1$, then
\begin{align*}
\frac{d}{ds} E_Q(s,\bm m_s) \geq -\widetilde{C} E_Q(s,\bm m_s) - n C_P(\bm m) \E\sum_{1 \le i < j \le n} |D_{ij}F(\widehat{\bX}_s)|^2.
\end{align*}
By Gronwall, we have
\begin{align*}
E_Q(s,\bm m_s) &\le e^{\widetilde{C}(T-s)}E_Q(T,\bm m_T) + n C_P(\bm m)\int_s^T e^{\widetilde{C}(u-s)}\E\sum_{1 \le i < j \le n} |D_{ij}F(\widehat{\bX}_u)|^2 \,du.
\end{align*} 
Arguing as in \eqref{a3est}, we deduce from Poincar\'e inequality (Lemma \ref{lem.poincare}) that
\begin{align*}
E_Q(T,\bm m_T) &= \frac{n}{2} \sum_{i = 1}^n \E\Big[ |D_i G(\widehat{\bX}_T)|^2 - \big|\E[D_i G(\widehat{\bX}_T)\,|\,\widehat{X}^i_T]\big|^2\Big] \\
	&\le nC_P(\bm m) \E\sum_{1 \le i < j \le n} |D_{ij}G(\widehat{\bX}_T)|^2.
\end{align*}
Bounding $e^{\widetilde{C}(T-s)}$ and $e^{\widetilde{C}(u-s)}$ by $e^{\widetilde{C}(T-t)}$, and using Fubini's theorem in the form of $\int_t^T\int_s^T h(u)\,duds = \int_t^T (u-t)h(u)\,du$ with $h(u)=|D_{ij}F(\widehat{\bX}_u)|^2$, we get
\begin{align}
\begin{split}
\int_t^T E_Q(s,\bm m_s)\,ds \le \ &(T-t)e^{\widetilde{C}(T-t)}n C_P(\bm m)  \sum_{1 \le i < j \le n} \E|D_{ij}G(\widehat{\bX}_T)|^2  \\
	&+ e^{\widetilde{C}(T-t)}n C_P(\bm m)  \int_t^T \sum_{1 \le i < j \le n} (s-t)\E|D_{ij}F(\widehat{\bX}_s)|^2\,ds .
\end{split} \label{eqest}
\end{align}
Combining this with Lemma \ref{lem.suffest} completes the proof of Theorem \ref{thm.est1}.

We have now established \eqref{eqest} under the additional assumption that $DV \in C^{1,2}$, which holds in particular when $G \in C^{3,\alpha}_{\text{loc}}$ in addition to Assumption \ref{assump.conv}. We now describe how to remove this condition via a mollification argument. We set $G^{\epsilon} = \int_{(\R^d)^n} G(\bx - \by) \rho^{\epsilon}(\by)$, where $(\rho^{\epsilon})_{\epsilon > 0}$ is a standard approximation to the identity on $(\R^d)^n$. Define $V^{\epsilon}$ exactly like $V$, but with $G^{\epsilon}$ replacing $G$. Define $\widehat{\bX}^{\epsilon}$, $\bm m^{\epsilon}_s$ and $E^{\epsilon}_Q(s,\bm m_s)$ just like $\widehat{\bX}$, $\bm m_s$ and $E_Q(s,\bm m_s)$ but with $V^{\epsilon}$ replacing $V$. Then from the identity 
\begin{align*}
D_{ij} G^{\epsilon}(\bx) = \int_{(\R^d)^n} D_{ij} G(\bx - \by) \rho^{\epsilon}(\by), 
\end{align*}
it is easy to verify that for each $\epsilon > 0$, we have 
\begin{align*}
\|D_{ij} G^{\epsilon}\|_{\linf} \leq \|D_{ij} G\|_{\linf}, \quad D^2 G^{\epsilon}(\bx) \leq \frac{C_G}{n} I_{nd \times nd}
\end{align*}
for all $\bx \in (\R^d)^n$. For each $\epsilon > 0$, we know $E^{\epsilon}_Q(s,\bm m^{\epsilon}_s)$ is bounded by the right-hand side of \eqref{eqest} but with $G$ and $\widehat{\bX}$ replaced by $G^\epsilon$ and $\widehat{\bX}^\epsilon$. To conclude that in fact \eqref{eqest} holds, we just need to be sure that, as $\epsilon \downarrow 0$,
\begin{align}
E^{\epsilon}_Q(s,\bm m^{\epsilon}_s) &\to E_Q(s,\bm m_s),  \label{eqapprox1} \\
 \E|D_{ij}G^\epsilon(\widehat{\bX}^\epsilon_T)|^2 &\to \E|D_{ij}G(\widehat{\bX}_T)|^2, \quad \E|D_{ij}F(\widehat{\bX}^\epsilon_s)|^2 \to \E|D_{ij}F(\widehat{\bX}_s)|^2. \label{eqapprox2}
\end{align}
But it is standard to show that the sequences of functions $V^{\epsilon}$ are in $C^{2,\alpha}_{\text{loc}}$, uniformly in $\epsilon$, and so a compactness argument shows
\begin{align*}
V^{\epsilon} \to V, \quad DV^{\epsilon} \to DV, \quad D^2 V^{\epsilon} \to D^2V, 
\end{align*}
locally uniformly on $[0,T] \times (\R^d)^n$. Together with a standard stability estimate for SDEs, this implies that $\widehat{\bX}^{\epsilon}_s \xrightarrow{L^2} \widehat{\bX}_s$, and in particular $\bm m_s^{\epsilon} \to \bm m_s$ in $\spt^n$, from which \eqref{eqapprox1} and \eqref{eqapprox2} follow. Thus \eqref{eqest} in fact holds without the additional regularity assumption. \hfill \qedsymbol

\subsection{Proof of Corollary \ref{co:quadratic}}
The assumption $L^i(x,a)=|a|^2/2$ implies that $H^i(x,p)=|p|^2/2$, and so $D_{pp}H^i(x,p)=I_d$. The spectral bound $D\bm\alpha \le -nD^2V \le 0$ and Lemma \ref{lem.poincare} imply that $\widehat{X}_s$ satisfies the Poincar\'e inequality with constant $s-t+c_0$, for each $s \in [t,T]$.
Following the proof of Theorem \ref{thm.est1}, and using $\alpha^j = -nD_jV$ and $\widehat\alpha^j(s,\widehat{X}^j_s) = -n\E[D_jV(s,\widehat{\bm X}_s)\,|\,\widehat{X}^j_s)$, we find $A_2 =0$ and $A_1,A_4 \ge 0$, and thus
\begin{align*}
\frac{d}{ds}E_Q(s,\bm m_s) \ge - n \E\sum_{i = 1}^n \big(D_i V(s,\widehat{\bX}_s) - \E[D_i V(s,\widehat{\bX}_s) | \widehat{X}_s^i]\big) \cdot \big(D_i F(\widehat{\bX}_s) - \E[D_i F(\widehat{\bX}_s) | \widehat{X}_s^i] \big).
\end{align*}
By Cauchy-Schwarz and the Poincar\'e inequality, we get
\begin{align*}
\frac{d}{ds}E_Q(s,\bm m_s) &\ge - 2 E_Q(s,\bm m_s)^{1/2} \bigg( \frac{n}{2}\E\sum_{i = 1}^n \big|D_i F(\widehat{\bX}_s) - \E[D_i F(\widehat{\bX}_s) | \widehat{X}_s^i] \big|^2 \bigg)^{1/2} \\
	&\ge - 2 E_Q(s,\bm m_s)^{1/2} \bigg(  n(s-t+c_0)\sum_{1 \le i < j \le n} \E|D_{ij}F(\widehat{\bX}_s)|^2\bigg)^{1/2}
\end{align*}
Thus,
\begin{align*}
\frac{d}{ds}\big(E_Q(s,\bm m_s)^{1/2}\big) \ge -\bigg(  (s-t+c_0)\sum_{1 \le i < j \le n} \E|D_{ij}F(\widehat{\bX}_s)|^2\bigg)^{1/2}.
\end{align*}
Integrate to find, for $s \in [t,T]$,
\begin{align*}
E_Q(s,\bm m_s)^{1/2} \le E_Q(T,\bm m_T)^{1/2} + \int_s^T\bigg( n(u-t+c_0)\sum_{1 \le i < j \le n} \E|D_{ij}F(\widehat{\bX}_u)|^2\bigg)^{1/2}\,du.
\end{align*}
Using the Poincar\'e inequality again, we have
\begin{align*}
E_Q(T,\bm m_T) &= \frac{n}{2} \sum_{i = 1}^n \E\Big[ |D_i G(\widehat{\bX}_T)|^2 - \big|\E[D_i G(\widehat{\bX}_T)\,|\,\widehat{X}^i_T]\big|^2\Big] \\
	&\le n(T-t+c_0)\E\sum_{1 \le i < j \le n} |D_{ij}G(\widehat{\bX}_T)|^2.
\end{align*}
Noting that $E=E_Q$ in the quadratic case, 
\begin{align*}
\vdt(t,\bm m) - \cv(t,\bm m) \le n\int_t^T \bigg[ &\bigg( (T-t+c_0)\sum_{1 \le i < j \le n} \E|D_{ij}G(\widehat{\bX}_T)|^2\bigg)^{1/2} \\
	&+ \int_s^T\bigg( (u-t+c_0)\sum_{1 \le i < j \le n} \E|D_{ij}F(\widehat{\bX}_u)|^2\bigg)^{1/2}\,du\bigg]^2\,ds.
\end{align*}
This completes the proof. \hfill \qedsymbol

\begin{remark} \label{re:unboundedness}
It is interesting to note that in some cases our results can go beyond the case where $(F,G,H,L)$ have bounded second derivatives.
In the case of quadratic Hamiltonian covered by Corollary \ref{co:quadratic}, the above proof did not really use any upper bound on the Hessians of $F$ or $G$ except to justify the well-posedness of the HJB equation for $V$, though this could likely be worked around in many cases.
Similarly, in dimension $d=1$ with general (convex) Hamiltonians depending only on $p$, we have $D_{pp}H^i \ge 0$ for all $i$ and thus $D\bm\alpha \le 0$, and again the Poincar\'e inequality can be improved to $C_P(\bm m) = T+c_0$ even without bounded  second derivatives of $(F,G,H,L)$.
In this case, though, it appears that the proof of Theorem \ref{thm.est1} still requires an upper bound on $D^2V$, for which we do require bounded second derivatives of $(F,G,L)$.
\end{remark}

\subsection{Proof of Theorem \ref{thm.est2}}

By a well-known entropy estimate for diffusions (e.g., Lemma 4.4 in \cite{lackerhierarchies}) we can estimate
\begin{align*}
H( \sL(\widehat{\bX}_{[t,T]}) \,|\, \sL(\bX_{[t,T]})) &\leq \frac12\E\sum_{i=1}^n \int_t^T  \big|D_p H^i(\widehat{X}_s^i, n D_iV(s,\widehat{\bX}_s)) )  -D_p H^i(\widehat{X}_s^i, n \E[D_iV(s,\widehat{\bX}_s) | \widehat{X}^i_s] ) \big|^2 \,ds  \\
	&= \frac12 n \int_t^T E(s,\bm m_s)\,ds \le \frac12 n\|D_{pp}H\|_\infty \int_t^T E_Q(s,\bm m_s)\,ds \\
	&\le \frac{n}{2}\sR(t,\bm m),
\end{align*}
with $\bm m_s = \sL(\widehat{X}_s^1,...,\widehat{X}_s^n)$ and $E$ and $E_Q$ defined as in Lemma \ref{lem.vlifteqn}, and where the last inequality was shown in \eqref{eqest}. 

An application of Lemma \ref{lem.cp}(ii) shows that $\L(\bX_{[t,T]})$ obeys the $T_2$ inequality with constant  $C_{T_2}(\bm m)$ defined in \eqref{def:CWconst}.
In particular,

\begin{align*}
\wass_2^2\big(\sL(\widehat{\bX}_{[t,T]}), \sL(\bX_{[t,T]})\big) \le C_{T_2}(\bm m)  H( \sL(\widehat{\bX}_{[t,T]}) \,|\, \sL(\bX_{[t,T]})) \le C_{T_2}(\bm m)  \frac{n}{2}\sR(t,\bm m).
\end{align*}
Next, we use a well known subadditivity inequality for $\wass_2^2$, which states that
\begin{align*}
\frac{1}{{n \choose k}} \sum_{S \subset [n], \, |S|=k} \wass_2^2\big(\sL((\widehat{X}^i_{[t,T]})_{i \in S}),\sL((X^i_{[t,T]})_{i \in S})\big) \le \frac{1}{\lfloor n/k \rfloor}\wass_2^2\big(\sL(\widehat{\bX}_{[t,T]}), \sL(\bX_{[t,T]})\big).
\end{align*}
For the short proof, see \cite[Section 3.4]{LacMukYeu}.
Combining the last two inequalities and noting that $\lfloor n/k\rfloor \ge n/2k$ completes the proof. \hfill \qedsymbol

\subsection{Proof of Theorem \ref{thm.est3}}

Let us denote by $(\bX,\bY,\bZ)$ the solution of \eqref{eq.mpstand} and by $(\overline{\bX}, \overline{\bY}, \overline{\bZ})$ the solution of \eqref{eq.mpdist}. We further denote by $\alpha^i_t = - D_p H^i(X^i_t, N\widehat{X}^i_t)$ the optimal (open loop)
control for the full-information control problem and $\overline{\alpha}^i_t = - D_p H^i(\overline{X}^i_t, N \overline{Y}^i_t)$ the optimal (open loop) control for the distributed problem. Finally, we set $\bm m_t = (\sL(\overline{X}_t^1),...,\sL(\overline{X}_t^n))$. We can rewrite the equation \eqref{eq.mpdist} satisfied by $(\overline\bX, \overline\bY, \overline\bZ)$ as 
\be \label{eq.mpperturbed}
\begin{cases}
\ds d\overline{X}_t^i = - \alpha_t^i dt + dW_t^i, \
\ds, \\
d\overline{Y}_t^i = -\big( \frac{1}{n}D_x L^i\big(\overline{X}_t^i, \alpha_t^i \big) + D_i F(\overline\bX_t) + E^{F,i}_t \big) dt + \overline{Z}_t^i dW_t^i, \\
\overline{X}_0^i = x^i, \quad \overline{Y}_T^i = G^i(\overline{X}_T) + E^{G,i}_T, 
\end{cases}
\ee
where we define
\begin{align*}
E^{F,i}_t = \sF^i(\overline\bX_t, \bm m_t) - D_iF(\overline\bX_t), \quad E^{G,i}_T = \sG^i(\overline\bX_T, \bm m_T) - D_i G(\overline\bX_T).
\end{align*}
By Lemmas \ref{lem.poincare} and \ref{lem.cp}, $\sL(\overline\bX_t)$ satisfies a Poincar\'e inequality with constant $C_P$. Thus, recalling the definition of $\sF^i$ and $\sG^i$, we have
\begin{align} \label{ferror}
\E |E^{F,i}_t|^2  \leq C_P \E \sum_{j=1, \,j \neq i}^n |D_{ij} F(\bX_t)|^2 ,   \\ \label{gerror}
\E |E^{G,i}_T|^2  \leq C_P \E \sum_{j=1, \,j \neq i}^n |D_{ij} G\bX_T)|^2 .
\end{align}
Now set 
\begin{align*}
\Delta \bX_t = \overline\bX_t - \bX_t, \quad \Delta \bY_t = \overline\bY_t - \bY_t,  \quad \Delta \bm\alpha_t = \overline{\bm\alpha}_t - \bm\alpha_t. 
\end{align*}
Then, for a certain martingale $M^i$, we find
\begin{align} \label{xycomp}
d(\Delta X_t^i \cdot \Delta Y_t^i) &=  - \Delta X_t^i  \cdot \bigg( \frac{1}{n} D_x L^i(\overline{X_t}^i, \overline{\alpha}_t^i) - \frac{1}{n} D_x L^i(X_t^i, \alpha_t^i) + D_i F(\overline{\bX}_t) - D_i F(\bX_t) + E^{F,i}_t \bigg) dt \nonumber  \\
	&\qquad + \Delta Y_t^i \cdot  \Delta \alpha_t^i dt  + dM^i_t \nonumber \\
&=  -\bigg( \frac{1}{n} \big(D_a L^i(\overline{X}_t^i, \overline{\alpha}_t^i) - D_a L^i(X_t^i, \alpha_t^i)\big) \cdot  \Delta \alpha^i_t 
 + \frac{1}{n} \big(  D_x L^i(\overline{X}_t^i, \overline{\alpha}_t^i) - D_x L^i(X_t^i, \alpha_t^i) \big) \cdot  \Delta X_t^i 
 \nonumber \\ &\qquad
 +  \big( D_i F(\overline{\bX}_t) -D_i F(\bX_t) \big)\cdot \Delta X_t^i  + E^{F,i}_t\cdot \Delta X_t^i \bigg) dt + dM^i_t, 
\end{align}
where the second equality comes from the the fact that $\alpha_t^i$ maximizes $a \mapsto - L^i(X_t^i, a) - n Y_t^i$, so that $Y_t^i = -\frac{1}{n} D_a L^i(X_t^i, \alpha_t^i)$, and likewise $\overline{Y}_t^i = -\frac{1}{n} D_aL^i(\overline{X}_t^i, \overline{\alpha}_t^i)$. Because $\Delta \bX_0 = 0$, we can integrate \eqref{xycomp} in time and take expectations to get
\begin{align*}
&\E\int_0^T \bigg(\frac{1}{n} \big(  D_a L^i(\overline{X}_t^i, \overline{\alpha}_t^i) - D_a L^i(X_t^i, \alpha_t^i)\big) \cdot \Delta \alpha^i_t  + \frac{1}{n} \big(D_x L^i(\overline{X}_t^i, \overline{\alpha}_t^i) - D_x L^i(X_t^i, \alpha_t^i) \big) \cdot \Delta X_t^i  \\
&\qquad+  \big(D_i F(\overline{\bX}_t) -  D_i F(\bX_t)  \big) \cdot\Delta X_t^i +  E^{F,i}_t\cdot \Delta X_t^i \bigg) dt  
\\ &= - \E\big[\Delta X_T^i \cdot \Delta Y^i_T\big] = -\E\big[\Delta X^i_T \cdot (D_iG(\bX_T) - D_iG(\overline{\bX}_T)) +  E^{G,i}_T\cdot \Delta X_T^i\big].
\end{align*}
Let $\bm{E}^F_t=(E^{F,i}_t,\ldots,E^{F,n}_t)$ and $\bm{E}^G_T=(E^{G,i}_T,\ldots,E^{G,n}_T)$.
Summing over $i$ and using convexity of $F$ and $G$ along with \eqref{lsrict}, we obtain the estimate 
\begin{align} \label{youngs}
\frac{C_L}{n} \E \int_0^T |\Delta \bm\alpha_t|^2 dt  &\leq \E\bigg[\int_0^T |\Delta \bX_t| |\bm{E}_t^{F}| dt + |\Delta \bX_T||\bm{E}^{G}_T|\bigg] \nonumber \\
&\leq 
\frac{\delta}{2} \E \int_0^T |\Delta \bX_t|^2 dt  + \frac{\epsilon}{2} \E |\Delta \bX_T|^2  + \frac{1}{2\delta} \E \int_0^T |\bm{E}_t^F|^2 dt  + \frac{1}{2\epsilon} \E |\bm{E}^{G}|^2 , 
\end{align}
for each $\epsilon, \delta > 0$.
To proceed, we note that $|\Delta \bX_t| = |\int_0^t \Delta \bm\alpha_s ds|$, from which it follows that 
\begin{align} \label{xalphaest1}
\E |\Delta \bX_t|^2  &\leq t \E \int_0^T |\Delta \bm\alpha_t|^2 dt , \\ \label{xalphaest2}
\E \int_0^T |\Delta \bX_t|^2 dt  &\leq \frac{T^2}{2} \E \int_0^T |\Delta \bm\alpha_t|^2 dt . 
\end{align}
Plugging \eqref{xalphaest1} and \eqref{xalphaest2} into \eqref{youngs} yields 
\begin{align*}
\frac{C_L}{n} \E \int_0^T |\Delta \bm\alpha_t|^2 dt  \leq 
\bigg(\frac{T^2\delta}{4} + \frac{T \epsilon}{2}\bigg) \E \int_0^T |\Delta \bm\alpha_t|^2 dt  +  \frac{1}{2\delta} \E \int_0^T |\bm{E}_t^F|^2 dt  + \frac{1}{2\epsilon} \E |\bm{E}^{G}_T|^2 .
\end{align*}
Now we choose $\delta = \frac{C_L}{T^2 n}$ and $\epsilon =\frac{C_L}{2nT}$ to conclude that 
\begin{align*}
\frac{C_L}{2n} \E\int_0^T |\Delta \bm\alpha_t|^2 dt 
\leq \frac{T^2 n}{2C_L} \E \int_0^T |\bm{E}_t^F|^2 dt  + \frac{nT}{C_L} \E |\bm{E}^{G}_T|^2 .
\end{align*}
Plugging in the estimates \eqref{ferror} and \eqref{gerror} completes the proof. \hfill\qedsymbol

\section{Application to mean field control}  \label{sec.mfc}

In this section, we explain the implications of the results in Section \ref{sec.approx} in the mean field case.
We work on a filtered probability space $(\Omega, \sF, \bP, \mathbb{F})$ with $\mathbb{F}$ satisfying the usual conditions and hosting independent Brownian motions $W, W^1,W^2,W^3,...$, and with the (augmented) filtration $\mathbb{F} = (\sF_t)_{0 \leq t \leq T}$ with $\sF_t$. We assume furthermore that $\sF_0$ is atomless. For the mean field case, we are given  functions
\begin{align*}
L = L(x,a) : \R^d \times \R^d \to \R, \quad \sF = \sF(m) : \sP_2(\R^d) \to \R, \quad \sG = \sG(m) : \sP_2(\R^d) \to \R
\end{align*}
and consider a sequence of control problems indexed by the number of particles $n \in \N$. Once again, it is useful to define the Hamiltonian 
\begin{align*}
H(x,p) = \sup_{a \in \R^d} \big( - a \cdot p - L(x,a) \big).
\end{align*}
Recall that $m^n_{\bx} := \frac{1}{n}\sum_{i=1}^n\delta_{x^i}$ denotes the empirical measure of a vector $\bx \in (\R^d)^n$.

The value function of the $n^{th}$ control problem is the map $V^n = V^n(t,\bx) : [0,T] \times (\R^d)^n \to \R$ given by 
\begin{align} 
V^n(t,\bx) = \inf_{\alpha = (\alpha^1,...,\alpha^n) \in \sA^n} &\E\bigg[\int_t^T \bigg(\frac{1}{n} \sum_{i = 1}^n L(X_s^i, \alpha^i(s,X_s)) + \sF(m_{\bX_s}^n)\bigg) ds + \sG(m_{\bX_T^n})\bigg], \label{costnpart}
\end{align}
subject to 
\begin{align}
 \label{dynamicsnpart}
dX_s^i = \alpha^i(s,\bX_s) dt + dW_s^i, \quad X_t^i = x^i,
\end{align}
where $\sA^n$ denotes the set of Borel functions $\bm\alpha = (\alpha^1,...,\alpha^n) : [0,T] \times (\R^d)^n \to \R$ such that the SDE \eqref{dynamicsnpart} admits a unique strong solution from any initial position.
We also consider as above the lift of $V^n$, i.e., the function $\cv^n = \cv^n(t,\bm m) : [0,T] \times \spt^n \to \R$ given by 
\begin{align*}
\cv^n(t,\bm m) = \langle \bm m, V^n(t,\cdot) \rangle.
\end{align*}
For large $n$, this $n$-particle problem is expected to be well-approximated by a mean field control problem, whose value function is the map $\cu = \cu(t,m) : [0,T] \times \sP_2(\R^d) \to \R$ given by
\begin{align}
\cu(t,m) = \inf_{\alpha \in \sA^1} \E\bigg[\int_t^T \bigg(L(X_s,\alpha(s,X_s)) + \sF(\sL(X_s))\bigg) dt + \sG(\sL(X_s))]\bigg], \label{mfproblem}
\end{align}
where
\begin{align}
 \label{dynamicsmf}
dX_s = \alpha(s,X_s)ds + dW_s, \quad X_t \sim m.
\end{align}

Finally, for each $n$ we can also introduce the corresponding distributed control problem, whose value function $\vdt^n = \vdt^n(t,\bm m) : [0,T] \times \spt^n \to \R$ is defined
\begin{align} \label{vndist}
\vdt^n(t,\bm m) = \inf_{\bm\alpha  \in \ad^n } E\bigg[\int_t^T \bigg(\frac{1}{n} \sum_{i = 1}^n L(X_s^i, \alpha^i(s,X_s)) + \sF(m_{\bX_s}^n)\bigg) ds + \sG(m_{\bX_T^n})\bigg], 
\end{align}
subject to 
\begin{align} \label{dynamicsnpartdist}
dX_s^i = \alpha^i(s,X_s^i) dt + dW_s^i, \quad \bX_t \sim \bm m,
\end{align}
where $\ad^n \subset \sA^n$ is the subset of distributed controls, i.e., $\alpha^i(t,\bx) = \alpha^i(t,x^i)$.

In order to state our main assumptions for the mean field setting, we consider the following condition for a function $f = f(m) : \spt \to \R$: 
\begin{align} \label{cstcondition} \tag{CST}
\begin{cases} \vspace{.2cm} \text{There exists an increasing function $\kappa : \R_+ \to \R_+$ such that } \\
\vspace{.2cm}
\quad \quad \E\bigg[\sup_{\overline{m} \in \spt} \Big|\frac{\delta f}{\delta m}(\overline{m}, \xi^1) \Big| + \sup_{\overline{m} \in \spt}\Big|\frac{\delta^2 f}{\delta m^2}(\overline{m}, \xi^1,\xi^2) \Big| \bigg] \leq \kappa\Big(\int_{\R^d} |x|^2 dm\Big), \\
\text{whenever $(\xi^1,\xi^2) \sim m\otimes m$.}
\end{cases}
\end{align}
Here $\frac{\delta}{\delta m}$ denotes the linear derivative; see \cite[Section 2.2.1]{cdll} or \cite[Section 2.1.1]{CST} for the definition. The point of the condition \eqref{cstcondition} is that by (the proof of) \cite[Theorem 2.11]{CST} we have the estimate 
\begin{align} \label{cst}
|f(m) - \E[f(m_{\bm\xi}^n)]| \leq \frac{C}{n}
\end{align}
for a constant $C$ depending on $m$ only through its second moment, whenever $\bm\xi = (\xi^1,...,\xi^n) \sim m^{\otimes n}$.

Recall that a function $\sG = \sG(m) : \sP_2(\R^d) \to \R$ is called displacement convex if its lift $\tilde{\sG} : L^2(\Omega) \to \R$ given by $\tilde{\sG}(X) = \sG(\sL(X))$ is convex in the usual sense. We recall that when $\sG$ is smooth and displacement convex, its Lions derivative $D_m \sG$ satisfies
\begin{align} \label{dispmonotone}
\E\bigg[\big(D_m \sG(\sL(\xi^1), \xi^1) - D_m \sG(\sL(\xi^2), \xi^2) \big) \cdot (\xi^1 - \xi^2) \bigg] \geq 0
\end{align}
for any square-integrable random variables $\xi^1,\xi^2$. 

We now make the following assumption, which is a sort of symmetric version of the Assumption \ref{assump.conv} used above.

\begin{assumption} \label{assump.convmfc}
The functions $L, H : \R^d \times \R^d \to \R$ are $C^2$ with bounded derivatives of order two (but not necessarily of order one). Moreover $L$ satisfies the coercivity conditions
\begin{align} \label{coercive}
L(x,a) \geq -C + \frac{1}{C} |a|^2, \quad (x,a) \in \R^d \times \R^d
\end{align}
for some $C > 0$ and satisfies the spectral bound
\begin{align}
    D^2 L(x,a) = \begin{pmatrix} 
    D_{xx} L(x,a) & D_{xa} L(x,a) \\
    D_{ax} L(x,a) & D_{aa} L(x,a)
    \end{pmatrix}
    \geq C_L \begin{pmatrix} 0 & 0 \\
    0 & I_{d \times d}    
    \end{pmatrix}, \quad \text{for all } x,a \in \R^d
\end{align}
or equivalently 
\begin{align}
\big(D_x L(x,a) - D_x L(\bar{x}, \bar{a})\big) \cdot (x - \bar{x}) + \big(D_a L^i(x,a) - D_a L^i(\bar{x}, \bar{a})\big) \cdot (a - \bar{a}) \geq C_L |a - \bar{a}|^2, 
\end{align}
for all $x,\bar{x}, a, \bar{a} \in \R^d$ and for some constant $C_L > 0$. Moreover, the functions $\sF$ and $\sG$ are bounded from below and displacement convex and $C^2$ with $D_y D_m \sF$, $D_{mm} \sF$, $D_y D_m \sG$, $D_{mm} \sG$ all bounded. Moreover, we assume that the maps 
\begin{align}
\spt \times \R^d \ni (m,x) \mapsto D_y D_m \sG(m,x), \quad \spt \times \R^d \times \R^d \ni (m,x,y) \mapsto D_{mm} \sG(m,x,y)
\end{align}
are locally $\alpha$-H\"older continuous for some $\alpha \in (0,1)$.
Finally, we assume that the condition \eqref{cstcondition} is satisfied for $f = \sF, \sG$. 
\end{assumption}

For some of the results stated in the remainder of this section, we also need the following assumption. 

\begin{assumption} \label{assump.mfc2}
The functions 
\begin{align*}
f(m) = D_m \sF(m,x), \quad f(m) = D_m \sG(m,x) 
\end{align*}
satisfy the condition \eqref{cstcondition} uniformly in $x$, i.e. with the increasing function $\kappa : \R_+ \to \R_+$ independent of $x$.
\end{assumption}

In order to make our estimates explicit, we employ in this section the following notational conventions.

\begin{convention} \label{conv-MF}
We will use the notation
\begin{align*}
\|D_{pp} H\|_{\infty} = \| |D_{pp} H|_{\ope} \|_{\linfty(\R^d \times \R^d)}, \\
\|D_{xp} H\|_{\infty} = \| |D_{pp} H|_{\ope} \|_{\linfty(\R^d \times \R^d)},  \\
\|D_{xx} L\|_{\infty} = \| |D_{xx} L^i|_{\ope} \|_{\linfty(\R^d \times \R^d)}. 
\end{align*}
We will denote by $C_{\sF}$ and $C_{\sG}$ the constants
\begin{align*}
C_{\sF} = \| |D_{mm} \sF|_{\text{op}} \|_{\linfty(\sP_2(\R^d) \times \R^d \times \R^d)}  + \| |D_y D_{m} \sF|_{\text{op}} \|_{\linfty(\sP_2(\R^d) \times \R^d)}, \\
C_{\sG} = \| |D_{mm} \sG|_{\text{op}} \|_{\linfty(\sP_2(\R^d) \times \R^d \times \R^d)}  + \| |D_y D_{m} \sG|_{\text{op}} \|_{\linfty(\sP_2(\R^d) \times \R^d)}.
\end{align*}
We denote by $C_S$ the constant 
\begin{align*}
C_S = C_{\sG} + T(\|D_{xx} L\|_{\infty} + C_{\sF}),
\end{align*}
and by $C_P$ the constant
\begin{align*} 
C_P = \frac{\exp\Big(2T \big(\|D_{xp} H\|_{\infty} + \|D_{pp} H\|_{\infty}C_S \big) \Big)  - 1}{2\big(\|D_{xp} H\|_{\infty} + \|D_{pp} H\|_{\infty}C_S\big)}.
\end{align*}
For $m \in \spt$ with Poincar\'e constant $c_0$, we denote by $C_P(m)$ the constant 
\begin{align*} 
C_P(m) = C_P + c_0 \exp \Big(2T \big(\|D_{xp} H\|_{\infty} + \|D_{pp} H\|_{\infty}C_S \big) \Big).
\end{align*}
\end{convention}

\subsection{Convergence of the value functions}

In this section, we show that $V^n$ converges to $U$ in a certain sense. We start by applying the estimate in Theorem \ref{thm.est1} to get the following: 

\begin{corollary} \label{cor.mf}
Under Assumption \ref{assump.convmfc}, for each $(t,\bm m) \in [0,T] \times \spt^n$, we have
\begin{align*}
0 \leq \vdt^n(t,\bm m) - \cv^n(t,\bm m) \leq \frac{C(\bm m)}{n}, 
\end{align*}
with the constant $C$ given by 
\begin{align*}
C(\bm m) = \bigg(\frac{T-t}{2}\|D_{mm} \sG\|^2_{\linfty}   + \frac{(T-t)^2}{4} \|D_{mm} \sF\|_{\linfty}^2   \bigg) \|D_{pp} H\|_{\infty}C_P(\bm m)e^{(T-t)(1+2C_S + 2 \|D_{xp} H\|_{\infty})}.
\end{align*}
\end{corollary}
\begin{proof}
The maps $G^n$, $F^n: (\R^d)^n \to \R$ given by
\begin{align*}
G^n(\bx) = \sG(m_{\bx}^n), \quad F^n(\bx) = \sF(m_{\bx}^n)
\end{align*}
are convex (since $\sG$ and $\sF$ are displacement convex). Moreover, the computation
\begin{align*}
D_{ij} G^n(\bx) = \frac{1}{n^2}D_{mm} \sG(\bx, x^i,x^j) + \frac{1}{n}D_{y}D_m \sG(m_{\bx}^n, x^i) 1_{i = j},
\end{align*}
and a similar computation for $F^n$, easily reveal the spectral estimates
\begin{align*}
D^2 G^n \leq \frac{C_{\sG}}{n}I_{nd \times nd}, \quad D^2 F^n \leq \frac{C_{\sF}}{n} I_{nd \times nd}. 
\end{align*}
Furthermore, we can estimate 
\begin{align*} 
n \sum_{1 \le i < j \le n} \|D_{ij} G^n\|_{\linf}^2 \leq \frac{1}{2n} \|D_{mm} \sG\|_{\linf}^2, \quad 
n \sum_{1 \le i < j \le n} \|D_{ij} F^n\|_{\linf}^2 \leq \frac{1}{2n} \|D_{mm} \sF\|_{\linf}^2. 
\end{align*}
Applying Theorem \ref{thm.est1} gives the result.
\end{proof}

It is also possible to compare $\vdt^n$ to $\cu$, in the following sense:

\begin{proposition} \label{prop.uandvnd}
Suppose that Assumption \ref{assump.convmfc} holds. For each $m \in \spt$, there is a constant $C$, depending only on the second moment of $m$ and the constants listed in Assumption \ref{assump.convmfc} and Convention \ref{conv-MF}, such that
\begin{align*}
|\vdt^n(t,m,...,m) - \cu(t,m)| \leq \frac{C}{n}.
\end{align*}
\end{proposition}
\begin{proof}
Fix $(t,\bm m)$, and $\alpha = \alpha(s,x) : [0,T] \times \R^d \to \R$ be an optimizer for the mean field control problem started from $(t,\bm m)$. For $n \in \N$, define $\bm\alpha^{n} \in \sA^n$ by $\alpha^{n,i}(s,x) = \alpha(s,x)$, i.e., $\bm\alpha^{n} = (\alpha,...,\alpha)$. Let $\bX = (X^1,...,X^n)$ be the corresponding state processes, 
\begin{align*}
dX_s^i = \alpha(s,X_s^i) ds + dW_s^i, \quad \bX_t^i \sim m^{\otimes n},
\end{align*}
Furthermore, let $X$ denote the optimal state process for the mean field control problem:
\begin{align*}
dX_s = \alpha(s,X_s) ds + dW_s, \quad X_t  \sim m. 
\end{align*}
Notice that $(X^1_s,...,X^n_s) \sim (\sL(X_s))^{\otimes n}$ for each $s$.
Then we have 
\begin{align*}
\vdt^n(t,m,...,m) &\leq \E\bigg[\int_t^T \bigg(\frac{1}{n} \sum_{i = 1}^n L(X_s^i, \alpha(s,X^i_s)) + \sF(m_{\bX_s}^n)\bigg) ds + \sG(m_{\bX_T}^n)\bigg] \\
&= \E\bigg[\int_t^T \Big(L(X_s, \alpha(s,X_s)) + \sF(m_{\bX_s}^n)\Big) ds + \sG(m_{\bX_T^n})\bigg]
\\ &= \E\bigg[\int_t^T \Big(L(X_s, \alpha(s,X_s)) + \sF(\sL(X_s))\Big) ds + \sG(\sL(X_s))\bigg] \\
&\qquad + \int_t^T \E[\sF(m_{\bX_s}^n) - \sF(\sL(X_s))] ds + \E[ \sG(m_{\bX_T}^n) - \sG(\sL(X_T))] \\
&= \cu(t,m) + \int_t^T \E[\sF(m_{\bX_s}^n) - \sF(\sL(X_s))] ds + \E[ \sG(m_{\bX_T}^n) - \sG(\sL(X_T))]
\end{align*}
It is straightforward to use the coercivity \eqref{coercive} to obtain $\sup_{0 \leq t \leq T} \E |X_t|^2 < \infty$. We can then apply \eqref{cst} to get
\begin{align*}
\vdt^n(t,m,...,m) \leq \cu(t,m) + \frac{C}{n}. 
\end{align*}
For the other direction, we note that by uniqueness (see Proposition \ref{existunique}), the optimizer for the distributed control problem must be of the form $\bm\alpha^n = (\alpha,...,\alpha)$ for some $\alpha : [0,T] \times \R^d \to \R$. Define $X^i$, $X$ exactly as above. This time, we can estimate 
\begin{align*}
\cu(t,m) &\leq \E\bigg[\int_t^T \Big(L(X_s, \alpha(s,X_s)) + \sF(\sL(X_s))\Big) ds + \sG(\sL(X_s))\bigg] \\
&= \E\bigg[\int_t^T \bigg(\frac{1}{n} \sum_{i = 1}^n L(X_s^i, \alpha(s,X^i_s)) + \sF(m_{\bX_s}^n)\bigg) ds + \sG(m_{\bX_T}^n)\bigg] \\
&\qquad + \int_t^T \E[\sF(\sL(X_s)) -\sF(m_{\bX_s}^n) ] ds + \E[\sG(\sL(X_T)) -  \sG(m_{\bX_T}^n) ] \\
&= \vdt^n(t,m,...,m) + \int_t^T \E[\sF(\sL(X_s)) -\sF(m_{\bX_s}^n) ] ds + \E[\sG(\sL(X_T)) -  \sG(m_{\bX_T}^n) ] 
\end{align*}
It is again straightforward to use the coercivity \eqref{coercive} to obtain an estimate on $\sup_{0 \leq t \leq T} \E |X_t^i|^2$, independent of $n$, so we can once again apply \eqref{cst} to get
\begin{align*}
\cu(t,m) \leq \vdt^n(t,m,...,m) + \frac{C}{n}, 
\end{align*}
which completes the proof.
\end{proof}

Combining Corollary \ref{cor.mf} with Proposition \ref{prop.uandvnd}, we get the following rate of convergence of $\cv^n$ to $\cu$:

\begin{theorem} \label{thm.mf}
Suppose that Assumption \ref{assump.convmfc} holds. Fix $m \in \spt$ satisfying a Poincar\'e inequality. Then there exists a constant $C$, depending only on the second moment and Poincar\'e constant of $m$, and the constants listed in Assumption \ref{assump.convmfc} and Convention \ref{conv-MF}, such that for each $n \in \N$,
\begin{align*}
|\cv^n(t,m,...,m) - \cu(t,m)| \leq \frac{C}{n}.
\end{align*}
\end{theorem}

\subsection{Propagation of chaos}
Next, we turn to propagation of chaos. We start by introducing some notation which will be in force throughout this subsection. We work with initial time $0$ for simplicity. We fix $m \in \spt$ and for each $n$, we fix $\bm \xi^n = (\xi^{1},...,\xi^{n})$ such that 
\begin{align*}
\bm \xi^n \sim m \otimes ... \otimes m. 
\end{align*}
We let $\bm \alpha^n = (\alpha^{n,1},...,\alpha^{n,n}) \in \sA^n$ and $\overline{\bm{\alpha}}^n = (\overline{\alpha}^{n,1},...,\overline{\alpha}^{n,n}) \in \ad^n$ denote the optimizers for the full-information and distributed $n$-particle problems (starting from $m$ at time $0$). We let $\bX^n = (X^{n,1},...,X^{n,n})$ and $\overline{\bX}^n = (\overline{X}^{n,1},...,\overline{X}^{n,n})$ denote the corresponding optimal state processes, i.e. the processes satisfying 
\begin{align*}
X^{n,i}_t = \xi^i + \int_0^t \alpha^{n,i}(s,\bX_s^{n}) ds + W_t^i, \\
\overline{X}^{n,i}_t = \xi^i + \int_0^t \overline{\alpha}^{n,i}(s,\overline{X}_s^{n,i}) ds + W_t^i.
\end{align*}
We view $\bm \alpha^n$ and $\overline{\bm \alpha}^n$ as processes (open-loop optimizers) by setting 
\begin{align*}
\bm \alpha^n_t = \bm \alpha^n(t,\bX^n_t), \quad \overline{\bm \alpha}_t^n = \overline{\bm \alpha}^n(t,\overline{\bm X}_t^n). 
\end{align*}

We will also need some notation for independent copies of the mean field optimizers, which we will denote by $(\alpha^{\mf,i})_{i \in \N}$. More precisely, for the remainder of the section we denote by $\alpha^{\text{MF}} = \alpha^{\text{MF}}(t,x) \in \sA^1$ the optimizer for the mean field control problem \eqref{mfproblem} (started at time $0$ from initial measure $m$), and then we set
\begin{align*}
\alpha^{\mf,i}_t = \alpha^{\mf}(t,X_t^{\mf,i}), 
\end{align*}
where $X^{\mf,i}$ solves 
\begin{align*}
X^{\mf,i}_t= \xi^i + \int_0^t \alpha^{\mf}(s,X_s^{\mf,i}) ds + W_t^i. 
\end{align*}
We start with a corollary obtained by applying Theorem \ref{thm.est3} in the mean field case in order to bound the difference between $\alpha^{n,i}$ and $\overline{\alpha}^{n,i}$. 

\begin{corollary} \label{cor.propofchaos}
Suppose that $m$ satisfies a Poincar\'e inequality and Assumption \ref{assump.convmfc} is in force. 
Then, for each $n \in \N$ and $i \in \{1,...,n\}$, we have
\begin{align*}
\E  \int_0^T |\alpha_t^{n,i} - \overline{\alpha}_t^{n,i}|^2 \, dt  \leq \frac{C(m)}{n}, 
\end{align*}
where the constant $C$ is given by 
\begin{align*}
C(m) = \frac{C_P T^3 \|D_{mm} \sF\|_{\linf}^2}{ C_L^2} + \frac{2 C_P T \|D_{mm} \sG\|_{\linf}^2}{C_L^2}
\end{align*}
\end{corollary}
\begin{proof}
The fact that 
\begin{align*}
\E  \sum_{i=1}^n \int_0^T |\alpha_t^{n,i} - \overline{\alpha}_t^{n,i}|^2 \,dt \leq C(m) 
\end{align*}
follows directly from Theorem \ref{thm.est3} just as Corollary \ref{cor.mf} follows from Theorem \ref{thm.est1}. The proof is completed by symmetry considerations.
\end{proof}

Corollary \ref{cor.propofchaos} compares the optimal full-information $n$-particle control $\bm \alpha^n$ to the optimal distributed $n$-particle control $\overline{\bm \alpha}^n$. We next compare $\overline{\bm \alpha}^n$ to independent copies of the mean field optimizers.

\begin{proposition} \label{prop.propofchaos}
Suppose that $m \in \spt$ and that Assumption \ref{assump.convmfc} is in force. Then there is a constant $C$ depending on $m$ only through its second moment such that for each $n \in \N$ and $i \in \{1,...,n \}$
\begin{align*}
\E  \int_0^T |\overline{\alpha}^{n,i}_t - \alpha^{\mf,i}_t|^2 \, dt  \leq \frac{C}{n^2}
\end{align*}
\end{proposition}

\begin{proof}
In this proof $C$ denotes a constant independent of $n$ which may change from line to line.
We know thanks to Proposition \ref{prop.mpnec} that the distributed optimizer $\overline{\bm \alpha}$ satisfies 
\begin{align*}
\overline{\alpha}_t^i = - D_p H^i(\overline{X}_t^i, n\overline{Y}_t^i), 
\end{align*}
where $(\overline{ \bm X}^n, \overline{ \bm Y}^n, \overline{\bm Z}^n)$ satisfy
\be \label{bareqn}
\begin{cases}
\ds d\overline{X}_t^{n,i} = \overline{\alpha}_t^{n,i} dt + dW_t^i, \
\ds \\
\ds d\overline{Y}_t^{n,i} = -\bigg( \frac{1}{n}D_x L^i\big(\overline{X}_t^{n,i}, \overline{\alpha}_t^{n,i}) + \frac{1}{n} \E\big[D_m \sF(m_{\overline{\bX}^n_t}^n, \overline{X}_t^{n,i}) \,|\, \overline{X}_t^{n,i} \big] \bigg) dt + \overline{Z}_t^{n,i} dW_t^i, \\
\overline{X}_0^{n,i} = \xi^i, \quad \overline{Y}_T^{n,i} = \frac{1}{n} \E\big[D_m \sG(m_{\overline{\bX}^n_T}^n, \overline{X}_T^{n,i}) \,|\, \overline{X}_T^{n,i} \big].
\end{cases}
\ee
Meanwhile, the maximum principle for the optimal control of McKean-Vlasov dynamics (see e.g. Theorem 4.5 in \cite{cardelmfc}) reveals that $\alpha_t^{\mf,i}$ satisfies 
\begin{align*}
\alpha_t^{\mf,i} = - D_p H(X_t^{\mf,i}, Y_t^{\mf,i}), 
\end{align*}
where $(X_t^{\mf,i}, Y_t^{\mf,i}, Z_t^{\mf,i})$ are such that 
\be 
\begin{cases}
\ds dX^{\mf,i}_t = \alpha_t^{\mf,i} dt + dW_t^i, \
\ds \\
\ds dY_t^{\mf,i} = -\bigg( D_x L^i\big(X_t^{\mf,i} , \alpha_t^{\mf,i} \big) + D_m\sF(m^{\mf}_t, X_t^{\mf,i}) \bigg) dt + Z_t^{\mf,i} dW_t^i, \\
X_0^{\mf,i} = \xi^i, \quad Y_T^{\mf,i} = D_m\sG(m^{\mf}_T, X_T^{\mf,i}), 
\end{cases}
\ee
where $m^{\mf}_t$ is the common law of the random variables $X_t^{\mf,i}$. We rewrite \eqref{bareqn} as 
\be \label{bareqn2}
\begin{cases}
\ds d\overline{X}_t^{n,i} = \overline{\alpha}_t^{n,i} dt + dW_t^i,
\ds \\
\ds d\overline{Y}_t^{n,i} = -\bigg( \frac{1}{n}D_x L^i\big(\overline{X}_t^{n,i}, \overline{\alpha}_t^{n,i}) + \frac{1}{n} D_m\sF\big(\overline{m}^n_t, \overline{X}_t^{n,i} \big) + \frac{1}{n} E_t^{\sF,i} \bigg) dt + \overline{Z}_t^{n,i} dW_t^i, \\
\overline{X}_0^{n,i} = \xi^i, \quad \overline{Y}_T^{n,i} = \frac{1}{n} D_m \sG\big(\overline{m}^n_T, \overline{X}_T^{n,i} \big) + \frac{1}{n} E^{\sG,i}_T,
\end{cases}
\ee
where $\overline{m}^n_t$ is the common law of the random variables $\overline{X}_t^{n,i}$ and 
\begin{align*}
E_t^{\sF,i} = \E\big[D_m \sF(m_{\overline{\bX}^n_t}^n, \overline{X}_t^{n,}) | \overline{X}_t^{n,i} \big] -  D_m\sF\big(\overline{m}^n_t, \overline{X}_t^{n,i}\big) , \\
E^{\sG,i}_T = \E\big[D_m \sG(m_{\overline{\bX}^n_T}^n, \overline{X}_T^{n,}) | \overline{X}_T^{n,i} \big] -  D_m\sG\big(\overline{m}^n_T, \overline{X}_T^{n,i}\big).
\end{align*}
We note that Assumption \ref{assump.mfc2} and \eqref{cst} imply easily that 
\begin{align} \label{fgerror}
\E |E^{\sF,i}_t |^2  \leq C/n^2, \quad \E |E^{\sG,i}_T |^2   \leq C/n^2. 
\end{align}
Now set 
\begin{align*}
\Delta X_t^i = \overline{X}_t^{n,i} - X_t^{\mf,i}, \quad \Delta Y_t^i = n \overline{Y}_t^{n,i} - Y_t^{\mf,i}, \quad \Delta \alpha_t^i = \overline{\alpha}_t^{n,i} - \alpha_t^{\mf,i},
\end{align*}
noticing the factor of $n$ in the definition of $\Delta Y^i_t$.
We now perform a computation very similar to the one appearing in the proof of Theorem \ref{thm.est3}, except instead of looking at the dynamics of $\sum_i \Delta X_t^i \cdot \Delta Y_t^i$, it turns out in this case we need only study the dynamics of $\Delta X_t^i \cdot \Delta Y_t^i$. Following the computation \eqref{xycomp}, we find that 
\begin{align} \label{xycompmfc}
&\E \bigg[ \int_0^T \bigg( \big(  D_a L^i(\overline{X}_t^{n,i}, \overline{\alpha}_t^{n,i}) - D_a L^i(X_t^{\mf,i}, \alpha_t^{\mf,i})\big) \cdot \Delta \alpha^i_t  + \big(D_x L^i(\overline{X}_t^{n,i}, \overline{\alpha}_t^{n,i}) - D_x L^i(X_t^{n,i}, \alpha_t^{n,i}) \big) \cdot \Delta X_t^i \nonumber \\
&\qquad+  \big(D_m \sF(\overline{m}^n_t, \overline{X}_t^{n,i}) -  D_m \sF(m_t^{\mf}, X_t^{\mf,i})  \big) \cdot\Delta X_t^i +  E^{\sF,i}_t \cdot \Delta X_t^i \bigg) dt \bigg] \nonumber
\\ &= - \E\big[\Delta X_T^i \cdot \Delta Y^i_T\big] = -\E\big[\big(D_m \sG(\overline{m}^n_T, \overline{X}_T^{n,i}) -  D_m \sG(m_T^{\mf}, X_T^{\mf,i})  \big) \cdot\Delta X_T^i +   E^{\sG,i}_T\cdot \Delta X_T^i\big]. 
\end{align}
The convexity in Assumption \ref{assump.convmfc} gives 
\begin{align*}
C_L \E\int_0^T |\Delta \alpha_t^i|^2 dt \le \E \bigg[ \int_0^T \bigg(& \big(  D_a L^i(\overline{X}_t^{n,i}, \overline{\alpha}_t^{n,i}) - D_a L^i(X_t^{\mf,i}, \alpha_t^{\mf,i})\big) \cdot \Delta \alpha^i_t  \\
	&+   \big(D_x L^i(\overline{X}_t^{n,i}, \overline{\alpha}_t^{n,i}) - D_x L^i(X_t^{n,i}, \alpha_t^{n,i}) \big) \cdot \Delta X_t^i \bigg)dt  \bigg] . 
\end{align*}
Combining this with \eqref{xycompmfc} and the displacement convexity of $\sF$ and $\sG$ (see \eqref{dispmonotone}) gives 
\begin{align*}
C_L \E \int_0^T |\Delta \alpha_t^i|^2 dt   \leq 
-\E\bigg[\int_0^T  E^{\sF,i}_t \cdot \Delta X_t^i  dt + E^{\sG,i}_T \cdot \Delta X_T^i \bigg].  
\end{align*}
The proof is now completed by an application of Young's inequality together with the estimates \eqref{fgerror}, exactly as in the proof of Theorem \ref{thm.est3}.
\end{proof}

Combining Corollary \ref{cor.propofchaos} and Proposition \ref{prop.propofchaos} gives the main estimate of this subsection. 

\begin{theorem} \label{thm.propofchaos}
Suppose that $m \in \spt$ satisfies a Poincar\'e inequality, and that Assumptions \ref{assump.convmfc} and \ref{assump.mfc2} are in force. Then there is a constant $C$ which depends on $m$ only through its second moment and Poincar\'e constant such that for each $n \in \N$ and $i \in \{1,...,n\}$,
\begin{align*}
\E \int_0^T |\alpha_t^{n,i} - \alpha_t^{\mf,i}|^2 \, dt  \leq \frac{C}{n}. 
\end{align*}
\end{theorem}

We note that Theorem \ref{thm.propofchaos} immediately implies an estimate on $X^{\mf,i} - X^{n,i}$ more in the spirit of propagation of chaos, i.e. we get 
\begin{align*}
\E\sup_{0 \leq t \leq T} |X^{\mf,i}_t - X^{n,i}_t|^2 \leq \frac{C}{n}.
\end{align*}
In particular, this implies an estimate on the $k$-particle marginals in the quadratic Wasserstein distance: 
\begin{align} \label{propofchaos}
\wass_2^2\Big( \sL\big(X^{n,1}_{[0,T]}, ...., X^{n,k}_{[0,T]}\Big), (m^{\mf}_{[0,T]})^{\otimes k}\bigg) \leq  \frac{Ck}{n}, 
\end{align}
where $m^{\mf}_{[0,T]} = \sL(X^{\mf,1}_{[0,T]})$ is the law of the mean field optimal state process.

\begin{remark} \label{re:differentapproximations}
As was discussed in Section \ref{se:intro:MFC}, the two approximations $\cu \approx \vdt^n$ and $\vdt^n \approx \cv^n$ are quite different in nature:
The convergence rate of $\vdt^n(t,m,\ldots,m) \to \cu(t,m)$ is essentially dictated by the convergence rate of an i.i.d.\ empirical measure to its limit.
We obtain a bound of $O(1/n)$ in Proposition \ref{prop.uandvnd} because we impose smoothness assumptions on the functionals $\sF$ and $\sG$ and rely on results of \cite{CST}. Under the weaker assumption that $\sF$ and $\sG$ are Lipschitz with respect to a Wasserstein distance, our convergence rate for $|\vdt-V_{\mathrm{MF}}|$ would be the same as that of empirical measure of i.i.d.\ random variables in $\R^d$ in expected Wasserstein distance, which is well known to deteriorate with the dimension \cite{fournier2015rate}. 
Convexity plays very little role in Proposition \ref{prop.uandvnd}.
An inspection of the proof  shows that convexity is used only to guarantee the existence of an optimizer $\overline{\bm \alpha}^n$ of \eqref{vndist} which is symmetric, i.e. such that $\overline{\alpha}^{n,i}(t,x^i) = \overline{\alpha}^n(t,x^i)$ for some $\overline{\alpha}^n \in \sA^1$. Thus if the existence of a symmetric optimizer is proved (or assumed), the conclusion of Proposition \ref{prop.uandvnd} remains valid, which reveals that the bound $|\vdt^n(t,m,\ldots,m) - \cu(t,m)| = O(1/n)$ should be expected even without convexity. 
\end{remark}

\begin{remark}
The reader may have noticed that Corollary \ref{cor.propofchaos} and Proposition \ref{prop.propofchaos} have a different dependence on $n$. In particular, when translated to an estimate on the state processes, Proposition \ref{prop.propofchaos} yields
\begin{align} \label{distpropofchaos}
\wass_2^2\Big( \sL\big(\overline{X}^{n,1}_{[0,T]}, ...., \overline{X}^{n,k}_{[0,T]}\big), (m^{\mf}_{[0,T]})^{\otimes k}\Big) \leq  \frac{Ck}{n^2}, 
\end{align}
which shows propagation of chaos for the distributed problems with a better rate than we have obtained in the full-information setting (Proposition \ref{prop.propofchaos}). In particular, \eqref{distpropofchaos} has the same dependence on $n$ as the second author's recent work \cite{lackerhierarchies} on propagation of chaos in the uncontrolled setting. It is an open question whether the same rate can be obtained in the full-information regime, i.e., whether or not \eqref{propofchaos} can be improved to $O(1/n^2)$ for each fixed $k$.
\end{remark}

\section{Application to heterogeneous doubly stochastic interactions}  \label{sec.hetero}

In this section, we explain the implications of the results in Section \ref{sec.approx} for a control problem in which interactions are governed by a graph. We work in the same filtered probability space $(\Omega, \sF, \bP, \mathbb{F})$  as in the previous section.
We are given  functions $L : \R^d \times \R^d \to \R$ and $G_1,G_2,F_1,F_2 : \R^d \to \R$, as well as an \emph{interaction matrix} $J^n$ which is an $n \times n$ symmetric matrix of nonnegative entries with zeros on the diagonal. We define
\begin{align*}
F^n(\bx) &:= \frac{1}{n}\sum_{i=1}^n F_1(x^i) + \frac{1}{n }\sum_{i,j=1}^n J^n_{ij} F_2(x^i-x^j), \\
G^n(\bx) &:= \frac{1}{n}\sum_{i=1}^n G_1(x^i) + \frac{1}{n }\sum_{i,j=1}^n J^n_{ij}G_2(x^i-x^j).
\end{align*}
We define the Hamiltonian again by
\begin{align*}
H(x,p) = \sup_{a \in \R^d} \big( - a \cdot p - L(x,a) \big).
\end{align*}

We consider the following sequence of control problems indexed by $n \in \N$.
The value function of the $n^{th}$ control problem is the map $V^n = V^n(t,\bx) : [0,T] \times (\R^d)^n \to \R$ given by 
\begin{align} 
V^n(t,\bx) = \inf_{\alpha = (\alpha^1,...,\alpha^n) \in \sA^n} &\E\bigg[\int_t^T \bigg(\frac{1}{n} \sum_{i = 1}^n L(X_s^i, \alpha^i(s,X_s)) + F^n(\bX_s)\bigg) ds + G^n(\bX_T)\bigg], \label{costnpart-hetero}
\end{align}
subject to  the dynamics \eqref{dynamicsnpart}.
We also consider as above the lift of $V^n$, i.e. the function $\cv^n = \cv^n(t,\bm m) : [0,T] \times \spt^n \to \R$ given by 
\begin{align*}
\cv^n(t,\bm m) = \langle \bm m, V^n(t,\cdot) \rangle.
\end{align*}
We define the mean field value function $\cu = \cu(t,m) : [0,T] \times \sP_2(\R^d) \to \R$  exactly as in  \eqref{mfproblem}, with the functions $\sF$ and $\sG$ given by
\begin{align}
\begin{split}
\sF(m) &:= \langle m, F_1\rangle + \int_{\R^d}\int_{\R^d} F_2(x-y) \,m(dx)m(dy) \\
\sG(m) &:= \langle m, G_1\rangle + \int_{\R^d}\int_{\R^d} G_2(x-y) \,m(dx)m(dy).
\end{split} \label{def:hetero-sF,sG}
\end{align}

Finally, for each $n$ we can also introduce the corresponding distributed control problem, whose value function $\vdt^n = \vdt^n(t,\bm m) : [0,T] \times \spt^n \to \R$ is defined
\begin{align} 
\vdt^n(t,\bm m) = \inf_{\bm\alpha  \in \sA^n_d} &\E\bigg[\int_t^T \bigg(\frac{1}{n} \sum_{i = 1}^n L(X_s^i, \alpha^i(s,X^i_s)) + F^n(\bX_s)\bigg) ds + G^n(\bX_T)\bigg], \label{costnpartdist-hetero}
\end{align}
subject to the dynamics \eqref{dynamicsnpartdist}.

We make the following assumption, similar to \eqref{assump.convmfc}.

\begin{assumption} \label{assump.convmfc-hetero}
The functions $(L,H)$ satisfy the conditions of \eqref{assump.convmfc}. The functions $(F_1,F_2,G_1,G_2)$ are convex and $C^2$ with bounded derivatives of order two, and in addition $G^1, G^2 \in C^{2,\alpha}_{\text{loc}}$ for some $\alpha \in (0,1)$. The interaction matrix $J^n$ is symmetric, has nonnegative entries and zeros on the diagonal, and is doubly stochastic; i.e., $\sum_{j=1}^n J^n_{ij} = \sum_{j=1}^n J^n_{ji}=1$ for all $i=1,\ldots,n$.
\end{assumption}


Perhaps the most important point here is the assumption that $J^n$ is doubly stochastic. This could likely be relaxed to an approximate form, in which the empirical measure of row sums should converge in some Wasserstein distance to $\delta_1$; we will not pursue this generalization here, but see \cite{basak2017universality} for an implementation of this idea in the context of Ising and Potts models. The most important special case is when $J^n=A^n/d_n$, where $d_n \in \N$ and $A^n$ is the adjacency matrix of a $d_n$-regular graph.
The following lemma proves a first remarkable point in the doubly stochastic, which is that the distributed control problem actually coincides with the mean field one. There is a close analogy with the identity (4.5) in \cite[Proof of Theorem 2.5]{LacMukYeu}, which essentially covers the case $F_1\equiv F_2\equiv 0$ and $m=\delta_0$.

\begin{lemma} \label{le:dst=mf-hetero}
Under Assumption \ref{assump.convmfc-hetero}, for each $(t,m) \in [0,T] \times \spt$, we have
\begin{align*}
\vdt^n(t,m,m,\ldots,m)= \cu(t,m) .
\end{align*}
\end{lemma}

Once we know this, we deduce a convergence rate by specializing Theorem \ref{thm.est1}:

\begin{corollary} \label{cor.mf-hetero}
Under Assumption \ref{assump.convmfc-hetero},  for each $(t,m) \in [0,T] \times \spt^n$ with $\bm m$ obeying a Poincar\'e inequality, we have
\begin{align*}
0 \leq \cu(t,m,m,\ldots,m) - \cv^n(t,m) \leq C \tr((J^n)^2)/n, 
\end{align*}
where the constant $C$ depends only on $T-t$, the bounds on second derivatives of $(F_1,F_2,G_1,G_2,L,H)$, and the Poincar\'e constant of $m$.
\end{corollary}

Again, a typical case is when $J^n_{ij}=A^n/d_n$, where $d_n \in \N$ and $A^n$ is the adjacency matrix of some $d_n$-regular graph. Then $\tr((J^n)^2)=n/d_n$, and Corollary \ref{cor.mf-hetero} yields $|\cu - \cv^n| \le C/d_n$, which vanishes as long as $d_n \to \infty$. In the very sparse regime where $\sup_n d_n < \infty$, we do not expect the usual mean field approximation to hold; see \cite{lacker2022case} for a discussion of dense versus sparse regimes in the context of mean field games and control.

\begin{proof}[Proof of Corollary \ref{cor.mf-hetero}]
In light of Lemma \ref{le:dst=mf-hetero}, it suffices to show the claim with $\vdt^n$ in place of $\cu$. We will check Assumption \ref{assump.conv} to apply Theorem \ref{thm.est1}. We first check the spectral bounds on $G^n$, with the bounds on $F^n$ checked analogously.
Recalling that $J^n$ is symmetric, we compute
\begin{align*}
D_{ij} G^n(\bx) &= -\frac{1}{n}  J^n_{ij}\big(D^2 G_2(x^i-x^j) + D^2 G_2(x^j-x^i)\big), \ \ i \neq j, \\
D_{ii} G^n(\bx) &= \frac{1}{n}D^2G_1(x^i) + \frac{1}{n}\sum_{j=1}^n J_{ij}^n \big( D^2 G_2(x^i-x^j) + D^2 G_2(x^j-x^i)\big).
\end{align*}
For a vector $\bz \in (\R^d)^n$,
\begin{align*}
n &\bz^\top D^2 G^n(\bx) \bz \\
	&= \sum_{i=1}^n z_i^\top D^2 G_1(x^i) z_i + \sum_{i,j=1}^n z_i^\top J^n_{ij} \big(D^2 G_2(x^i-x^j) +  D^2 G_2(x^j-x^i)\big) (z_i-z_j) \\
	&= \sum_{i=1}^n z_i^\top D^2 G_1(x^i) z_i + \frac12 \sum_{i,j=1}^n (z_i-z_j)^\top J^n_{ij}  \big(D^2 G_2(x^i-x^j) +  D^2 G_2(x^j-x^i)\big) (z_i-z_j).
\end{align*}
Recalling that $J^n_{ij} \ge 0$ and $G_2$ is convex, for each $(i,j)$ we deduce
\begin{align*}
0 &\le (z_i-z_j)^\top J^n_{ij} \big( D^2 G_2(x^i-x^j) +  D^2 G_2(x^j-x^i)\big) (z_i-z_j) \\
	&\le 2 J^n_{ij}   \|D^2G_2\|_\infty |z_i-z_j|^2 \\
	&\le 4J^n_{ij} \|D^2G_2\|_\infty (|z_i|^2 + |z_j|^2).
\end{align*}
Recalling that $\sum_{j=1}^nJ^n_{ij}= \sum_{j=1}^n J^n_{ji}=1$ and $G_1$ is convex, we deduce that
\begin{align*}
0 &\le n  D^2 G^n(\bx) \le \big(\|D^2 G_1\|_\infty + 4 \|D^2G_2\|_\infty\big)I_{nd \times nd}. 
\end{align*}
Finally, recalling the form of $D_{ij}G^n$ for $i \neq j$, and that $J^n_{ii}=0$ for all $i$, we have 
\begin{align*} 
n \sum_{1 \le i < j \le n} \|D_{ij} G^n\|_{\linf}^2  &\leq \frac{4}{n}\|D^2 G_2\|_{\linf}^2 \sum_{1 \le i < j \le n} (J^n_{ij})^2 = \frac{4}{n}\|D^2 G_2\|_{\linf}^2 \tr((J^n)^2).
\end{align*}
We are now in a position to apply Theorem \ref{thm.est1} to complete the proof.
\end{proof}

\begin{proof}[Proof of Lemma \ref{le:dst=mf-hetero}]
We first check the easier inequality $\vdt^n(t,m,m,\ldots,m) \le \cu(t,m)$. Let $\alpha \in \A^1$ be any control for $\cu$. Define $\bm\alpha=(\alpha^1,\ldots,\alpha^n) \in \ad^n$ by setting $\alpha^i(t,x^i)=\alpha(t,x^i)$. The corresponding state processes $\bX=(X^1,\ldots,X^n)$, solving \eqref{dynamicsnpart}, are i.i.d.\ with the same law as $X$ solving \eqref{dynamicsmf}. Let $\widetilde{X}$ denote an independent copy of $X$. Then, recalling the form of $\sG$ from \eqref{def:hetero-sF,sG},
\begin{align*}
\sG(\sL(X_T)) &= \E\bigg[G_1(X_T) + \frac{1}{n}\sum_{i=1}^n G_2(X_T-\widetilde{X}_T) \bigg] \\
	&= \E\bigg[\frac{1}{n}\sum_{i=1}^n G_1(X_T) + \frac{1}{n}\sum_{i,j=1}^n J^n_{ij} G_2(X_T-\widetilde{X}_T) \bigg] \\
	&= \E\bigg[\frac{1}{n}\sum_{i=1}^n G_1(X^i_T) + \frac{1}{n}\sum_{i,j=1}^n J^n_{ij} G_2(X^i_T-X^j_T) \bigg] = \E[G^n(\bX_T)].
\end{align*}
Indeed, the second step comes from the assumption that the row sums of $J^n$ are $1$, and the third from the assumption that $J^n_{ii}=0$ for all $i$. Arguing similarly for $\sF(\sL(X_t))$, we deduce that
\begin{align*}
\E&\bigg[\int_t^T \bigg(L(X_s,\alpha(s,X_s)) + \sF(\sL(X_s))\bigg) ds + \sG(\sL(X_T))\bigg] \\
	&= \E\bigg[\int_t^T \bigg(\frac{1}{n} \sum_{i = 1}^n L(X_s^i, \alpha^i(s,X^i_s)) + F^n(\bX_s)\bigg) ds + G^n(\bX_T)\bigg] \\
	&\ge \vdt^n(t,m,m,\ldots,m).
\end{align*}

To prove the reverse inequality $\vdt^n(t,m,m,\ldots,m) \ge \cu(t,m)$, fix an optimal control $\bm\alpha=(\alpha^1,\ldots,\alpha^n) \in \ad^n$. Let $\bX=(X^1,\ldots,X^n)$ be the corresponding state process, solving \eqref{dynamicsnpart}, driven by independent Brownian motions $W^1,\ldots,W^n$. 
Recall that $W$ denotes an additional independent Brownian motion, and fix some random variable $\xi \sim m$. We now construct $\overline{\bX}=(\overline{X}^1,\ldots,\overline{X}^n)$ as the solutions of
\begin{align*}
d\overline{X}^i_s = \alpha^i(s,\overline{X}^i_s)ds + dW_s, \ \ \overline{X}^i_t=\xi.
\end{align*}
This way, $\overline{X}^i \stackrel{d}{=} X^i$ for each $i=1,\ldots,n$, although the $\overline{X}^i$ are not independent. 
Let $\widetilde{\bX}=(\widetilde{X}^1,\ldots,\widetilde{X}^n)$ be an independent copy of $\bX=(X^1,\ldots,X^n)$, and let $\widetilde{\overline{\bX}}=(\widetilde{\overline{X}}^1,\ldots,\widetilde{\overline{X}}^n)$ be an independent copy of $\overline{\bX}=(\overline{X}^1,\ldots,\overline{X}^n)$.
Define
\begin{align*}
\overline{X} = \frac{1}{n}\sum_{i=1}^n \overline{X}^i, \quad \widetilde{\overline{X}} = \frac{1}{n}\sum_{i=1}^n \widetilde{\overline{X}}^i, \quad \overline{\alpha}_s = \frac{1}{n}\sum_{i=1}^n \alpha^i(s,\overline{X}^i_s).
\end{align*}

By convexity of $G_2$, and because $J^n_{ii}=0$ for all $i$ and $\frac{1}{n}\sum_{i,j=1}^nJ^n_{ij}=1$, we have
\begin{align*}
\E[G^n(\bX_T)] &= \E\bigg[\frac{1}{n}\sum_{i=1}^n G_1(X^i_T) + \frac{1}{n}\sum_{i,j=1}^n J^n_{ij} G_2(X^i_T-X^j_T) \bigg] \\
	&= \E\bigg[\frac{1}{n}\sum_{i=1}^n G_1(\overline{X}^i_T) + \frac{1}{n}\sum_{i,j=1}^n J^n_{ij} G_2(\overline{X}^i_T-\widetilde{\overline{X}}^j_T) \bigg] \\
	&\ge \E\bigg[ G_1\bigg(\frac{1}{n}\sum_{i=1}^n \overline{X}^i_T \bigg) + G_2\bigg(\frac{1}{n}\sum_{i,j=1}^n J^n_{ij} (\overline{X}^i_T-\widetilde{\overline{X}}^j_T)\bigg) \bigg] \\
	&= \E\bigg[ G_1(\overline{X}_T) + G_2(\overline{X}_T-\widetilde{\overline{X}}_T) \bigg] \\
	&= \sG(\sL(\overline{X}_T)).
\end{align*} 
Similarly, $\E[F^n(\bX_s)] \ge \sF(\sL(\overline{X}_s))$ for each $s \in [t,T]$, and
\begin{align*}
\E\bigg[ \frac{1}{n}\sum_{i=1}^n L(X^i_s,\alpha^i(s,X^i_s)) \bigg] &= \E\bigg[ \frac{1}{n}\sum_{i=1}^n L(\overline{X}^i_s,\alpha^i(s,\overline{X}^i_s)) \bigg] \ge \E\big[ L(\overline{X}_s, \overline{\alpha}_s) \big].
\end{align*}
We deduce that
\begin{align*}
\vdt^n(t,m,m,\ldots,m) &= \E\bigg[\int_t^T \bigg(\frac{1}{n} \sum_{i = 1}^n L(X_s^i, \alpha^i(s,X^i_s)) + F^n(\bX_s)\bigg) ds + G^n(\bX_T)\bigg]  \\
	&\ge \E\bigg[\int_t^T \bigg(L(\overline{X}_s,\overline{\alpha}_s) + \sF(\sL(\overline{X}_s))\bigg) ds + \sG(\sL(\overline{X}_T))\bigg].
\end{align*}
Now, if we could find some $\alpha \in \sA^1$ such that $\overline{\alpha}_s=\alpha(s,\overline{X}_s)$, then we could conclude that the right-hand side would be at least $\cu(t,m)$, completing the proof.

It is not clear, however, that we may find such $\alpha \in \sA^1$ in general. But a simple approximation will remedy this. Note that $\overline{X}$ satisfies
\begin{align*}
d\overline{X}_s = \overline\alpha_s ds + dW_s, \quad \overline{X}_t= \xi.
\end{align*}
Let $\overline\alpha^n_s$ be the projection of $\overline\alpha_s$ onto the ball of radius $n$, and define $\overline{X}^n$ by
\begin{align*}
d\overline{X}^n_s = \overline\alpha^n_s ds + dW_s, \quad \overline{X}^n_t=\xi.
\end{align*}
It is clear that $\E \sup_{s \in [t,T]}|\overline{X}^n_s-\overline{X}_s|^2 \to 0$ and $\E\int_s^T|\overline\alpha^n_s-\overline\alpha_s|^2\,ds \to 0$. Using continuity of $(L,F,G)$ along with their quadratic growth (implied by boundedness of the second derivatives), it is straightforward to deduce that
\begin{align*}
\lim_{n\to\infty} &\E\bigg[\int_t^T \bigg(L(\overline{X}^n_s,\overline{\alpha}^n_s) + \sF(\sL(\overline{X}^n_s))\bigg) ds + \sG(\sL(\overline{X}^n_T))\bigg] \\
&= \E\bigg[\int_t^T \bigg(L(\overline{X}_s,\overline{\alpha}_s) + \sF(\sL(\overline{X}_s))\bigg) ds + \sG(\sL(\overline{X}_T))\bigg].
\end{align*}
Hence, to complete the proof, it suffices to show for each $n$ that 
\begin{align}
\E\bigg[\int_t^T \bigg(L(\overline{X}^n_s,\overline{\alpha}^n_s) + \sF(\sL(\overline{X}^n_s))\bigg) ds + \sG(\sL(\overline{X}^n_T))\bigg] \ge \cu(t,m). \label{pf:hetero-goal}
\end{align}
To this end, let us fix $n$, and find a Borel measurable function $\widehat\alpha : [0,T] \times \R^d \to \R^d$ satisfying 
\begin{align}
\widehat{\alpha}(s,x) = \E[\overline\alpha^n_s\,|\,\overline{X}^n_s=x]. \label{pf:hetero-alphahat}
\end{align}
Indeed, a jointly measurable version exists by \cite[Proposition 5.1]{brunick2013mimicking}, and we may take $\widehat\alpha$ to be bounded because $\overline{\alpha}^n$ is. By the Markovian projection theorem \cite[Corollary 3.7]{brunick2013mimicking}, we may find a weak solution $\widehat{X}$ of the SDE
\begin{align*}
d\widehat{X}_s = \widehat{\alpha}(s,\widehat{X}_s) ds + dW_s,
\end{align*}
satisfying $\widehat{X}_s \stackrel{d}{=} \overline{X}^n_s$ for all $s \in [t,T]$. By boundedness of $\widehat{\alpha}$ and a result of Veretennikov \cite{veretennikov}, this SDE is in fact well posed in the strong sense, and so $\widehat{\alpha} \in \sA^1$. Using \eqref{pf:hetero-alphahat}, Fubini's theorem, and convexity of $(L,F,G)$, we have
\begin{align*}
\E \int_t^T L(\overline{X}^n_s,\overline{\alpha}^n_s)  ds   &\ge \E \int_t^T L(\overline{X}^n_s,\widehat{\alpha}(s,\overline{X}^n_s))  ds.
\end{align*}
Using  the fact that $\widehat{X}_s \stackrel{d}{=} \overline{X}^n_s$ for all $s$, we deduce
\begin{align*}
\E&\bigg[\int_t^T \bigg(L(\overline{X}^n_s,\overline{\alpha}^n_s) + \sF(\sL(\overline{X}^n_s))\bigg) ds + \sG(\sL(\overline{X}^n_T))\bigg] \\
	&\ge \E\bigg[\int_t^T \bigg(L(\widehat{X}_s,\widehat{\alpha}(s,\widehat{X}_s)) + \sF(\sL(\widehat{X}_s))\bigg) ds + \sG(\sL(\widehat{X}_T))\bigg].
\end{align*}
The right-hand side is at least $\cu(t,m)$, and the proof of \eqref{pf:hetero-goal} is complete.
\end{proof}

\section{Tradeoffs between convexity and smallness}

\label{sec.nonconvex}

The goal of this section is to explain how versions our main estimates can be obtained when convexity of the data is traded for some form of smallness. 

We start with an informal discussion of what results we can hope to obtain without convexity. First, let us mention that we cannot expect versions of our main estimates to hold in general. Indeed, if an estimate like the one in Theorem \ref{thm.est1} was obtained in a general non-convex setting, it could be combined with the arguments from Section \ref{sec.mfc} to give an estimate of the form 
\begin{align} \label{toogood}
|\vdt^n(t,m,...,m) - \cu(t,m)| \leq \frac{C}{n}
\end{align}
for mean field control problems with costs $\sF$ and $\sG$ which are neither convex nor displacement convex. While no counterexample is currently known, obtaining \eqref{toogood} in the non-convex setting would be surprising since the convergence problem for mean field control problems with non-convex data presents serious difficulties, and the rate $1/n$ would be a significant improvement over existing estimates. See \cite{cdjs} and the references therein for a discussion of the convergence problem for mean field control in the non-convex setting, and \cite{Cecchin2021FiniteSN} for results in the finite state-space setting.

Nevertheless, an inspection of the proofs of Theorems \ref{thm.est1} and \ref{thm.est2} show that the estimates obtained depend on convexity of $L^i$, $F$, and $G$ only through the bound 
\begin{align} \label{operatorest}
\sup_{(t,\bx) \in [0,T] \times (\R^d)^n} |D^2 V(t,\bx)|_{\text{op}} \leq \frac{C_S}{n}
\end{align}
which was proved in Lemma \ref{lem.spectral}.
Thus if an analogous bound on the operator norm of $D^2V$ can be obtained without convexity, extensions of Theorems \ref{thm.est1} and \ref{thm.est2} would follow. This strategy can be executed under the assumption that $T$ is small (relative to the regularity of the data), but we do not pursue this generalization for brevity.

We will, however, prove a version of Theorem \ref{thm.est3} under the following assumption, which replaces convexity of $F$ and $G$ with Lipschitz continuity, together with a smallness condition.

\begin{assumption} \label{assump.nonconvex}
The functions $F, G : (\R^d)^n \to \R$ are bounded from below, $C^2$ with bounded derivatives of order 1 and 2. The functions $L^i, H^i : \R^d \times \R^d \to \R$ are $C^2$ with bounded derivatives of order two (but not necessarily of order one).
 Moreover $L^i$ is bounded from below and satisfies
\begin{align} \label{lsrict2}
    D^2 L^i(x,a) = \begin{pmatrix} 
    D_{xx} L^i(x,a) & D_{xa} L^i(x,a) \\
    D_{ax} L^i(x,a) & D_{aa} L^i(x,a)
    \end{pmatrix}
    \geq C_L \begin{pmatrix} 0 & 0 \\
    0 & I_{d \times d}    
    \end{pmatrix}, \quad \text{for all } x,a \in \R^d
\end{align}
or equivalently 
\begin{align} \label{lstrictequiv2}
\big(D_x L^i(x,a) - D_x L^i(\bar{x}, \bar{a})\big) \cdot (x - \bar{x}) + \big(D_a L^i(x,a) - D_a L^i(\bar{x}, \bar{a})\big) \cdot (a - \bar{a}) \geq C_L |a - \bar{a}|^2, 
\end{align}
for all $x,\bar{x}, a, \bar{a} \in \R^d$ and for some constant $C_L > 0$.
We denote by $C_F$ and $C_G$ two constants such that the following spectral lower bounds hold:
\begin{align*}
 D^2 F(\bx) \ge -\frac{C_F}{n} I_{nd \times nd}, \quad D^2 G(\bx) \ge -\frac{C_G}{n} I_{nd \times nd}.
\end{align*}
\end{assumption} 

Let us mention that under Assumption \ref{assump.nonconvex}, the existence of a solution to the McKean-Vlasov FBSDE \eqref{eq.mpdist} can be obtained directly from \cite{cardelmkvfbsde}, and it provides a necessary condition for optimality. However, without the additional convexity assumptions on $\sF$ and $\sG$, it is not clear that the maximum principle is a sufficient condition; that is, the conclusion of Proposition \ref{prop.mpnec} holds but not necessary the conclusion of Proposition \ref{prop.mpsuff}. Thus, in the following Theorem we do not assert the existence of an optimizer for the distributed problem, we simply show that if one exists then it is close to the optimizer for the standard control problem (whose existence and uniqueness is well-known under Assumption \ref{assump.nonconvex}). We additionally assume a Poincar\'e inequality for the optimal state process, with the subsequent Proposition \ref{prop.poinnonconvex} giving some sufficient conditions.

\begin{theorem} \label{thm.nonconvex}
Suppose that Assumption \ref{assump.nonconvex} holds, and that $\bm\alpha$ is an optimizer for \eqref{controldistol} and $\overline{\bm\alpha}$ is an optimal (open loop) control for the full-information control problem.
Suppose $\sL(\bX_t)$ satisfies a Poincar\'e inequality with constant $C_P$, for each $t \in [0,T]$.
 Suppose further that
\begin{align}
\frac12 C_F T^2 + C_G T < C_L  . \label{asmp:smallness}
\end{align}
Then we have 
\begin{align*}
\E \int_0^T |\bm\alpha_t - \overline{\bm\alpha}_t|^2 dt  \leq  C n^2 \sum_{1 \le i < j \le n} \big(\|D_{ij} F\|_{\linf}^2 + \|D_{ij} G\|_{\linf}^2 \big),
 \end{align*}
 where $C$ depends only on $(C_P,C_L,C_F,C_S,T)$.
\end{theorem}
\begin{proof}
The start of the proof is exactly like the proof of Theorem \ref{thm.est3}. We denote by $(\bX,\bY,\bZ)$ the solution of \eqref{eq.mpstand} such that $\alpha^i = - D_p H^i(X^i, NY^i)$ and by $(\overline{\bX}, \overline{\bY}, \overline{\bZ})$ the solution of \eqref{eq.mpdist} such that $\overline{\alpha} = - D_p H^i(\overline{X}^i, N \overline{Y}^i)$. Finally, we set $\bm m_t = (\sL(\overline{X}_t^1),...,\sL(\overline{X}_t^n))$, and set
\begin{align*}
\Delta \bX_t = \overline{\bX}_t - \bX_t , \quad \Delta \bY_t = \overline{\bY} _t - \bY_t , \quad \Delta \bm\alpha_t =  \overline{\bm\alpha}_t - \bm\alpha_t.  
\end{align*}
As in the proof of Theorem \ref{thm.est3}, we study the dynamics of the process $\Delta \bX_t \Delta \bY_t$ to get
\begin{align*} 
&\E\int_0^T \bigg(\frac{1}{n} \big(  D_a L^i(\overline{X}_t^i, \overline{\alpha}_t^i) - D_a L^i(X_t^i, \alpha_t^i)\big) \cdot \Delta \alpha^i_t  + \frac{1}{n} \big(D_x L^i(\overline{X}_t^i, \overline{\alpha}_t^i) - D_x L^i(X_t^i, \alpha_t^i) \big) \cdot \Delta X_t^i   \\
&\qquad+  \big(D_i F(\overline{\bX}_t) -  D_i F(\bX_t)  \big) \cdot\Delta X_t^i +  E^{F,i}_t\cdot \Delta X_t^i \bigg) dt  
 \\ &= - \E\big[\Delta X_T^i \cdot \Delta Y^i_T\big] = -\E\big[\Delta X^i_T \cdot (D_iG(\bX_T) - D_iG(\overline{\bX}_T)) +  E^{G,i}_T\cdot \Delta X_T^i\big].
\end{align*}
where $E_t^{F,i}$ and $E^{G,i}_T$ are defined as in the proof of Theorem \ref{thm.est3} and satisfy
\begin{align} \label{ferror2}
\E|E^{F,i}_t|^2 \leq C_P \E[\sum_{j \neq i} |D_{ij} F(\bX_t)|^2],   \\ \label{gerror2}
\E|E^{G,i}_T|^2 \leq C_P \E[\sum_{j \neq i} |D_{ij} G\bX_T)|^2].
\end{align}
Unlike in the proof of Theorem \ref{thm.est3}, this time we cannot use convexity of $F$ and $G$ to control the $D_iF$ and $D_iG$ terms. Instead, summing over $i$ and using Assumption \ref{assump.nonconvex} only leads to 
\begin{align} \label{youngs2}
\frac{C_L}{n} \E\int_0^T |\Delta \bm\alpha_t|^2 dt  &\leq \E\bigg[\int_0^T |\Delta \bX_t| |\bm{E}_t^{F}| dt + |\Delta \bX_T||\bm{E}^{G}_T| \bigg] \nonumber  \\
&\quad -\sum_{i=1}^n\E\bigg[\int_0^T  \Delta X_t^i \cdot(D_i F(\overline{\bX}_t) - D_iF(\bX_t)) \,dt -   \Delta X_t^i \cdot (D_i G(\overline{\bX}_T) - D_i G(\bX_T))\bigg] \nonumber \\
&\leq 
\frac{\delta}{2n} \E\int_0^T |\Delta \bX_t|^2 dt  + \frac{\epsilon}{2n} \E |\Delta \bX_T|^2 + \frac{n}{2\delta} \E \int_0^T |\bm{E}_t^F|^2 dt  + \frac{n}{2\epsilon} \E |\bm{E}^{G}_T|^2 \nonumber \\
&\quad + \frac{C_F}{n} \E \int_0^T |\Delta \bX_t|^2 dt  + \frac{C_G}{n} \E |\Delta \bX_T|^2 ,
\end{align}
for any $\delta,\epsilon > 0$.
We use the estimates 
\begin{align} \label{xalphaestnc1}
\E|\Delta \bX_t|^2 &\leq t \E\int_0^T |\Delta \bm\alpha_t|^2 dt, \\ \label{xalphaestnc2}
\E\int_0^T |\Delta \bX_t|^2 dt &\leq \frac{T^2}{2} \E\int_0^T |\Delta \bm\alpha_t|^2 dt. 
\end{align}
to get 
\begin{align*}
\frac{C_L}{n} \E\int_0^T |\Delta \bm\alpha_t|^2 dt &\leq 
\bigg(\frac{T^2\delta}{4n} + \frac{T \epsilon}{2n} +  \frac{C_F T^2}{2n} + \frac{C_G T}{n} \bigg) \E\int_0^T |\Delta \bm\alpha_t|^2 dt  \\
&\quad +  \frac{n}{2\delta} \E\int_0^T |\bm{E}_t^F|^2 dt + \frac{n}{2\epsilon} \E|\bm{E}^{G}_T|^2 .
\end{align*}
Using the Assumption  \eqref{asmp:smallness}, 
we may choose $\epsilon,\delta$ sufficiently small so that $C < C_L$, where $C$ is defined by
\begin{align*}
C = \frac{T^2\delta}{4} + \frac{T \epsilon}{2} +  \frac{C_F T^2}{2} + C_G T.
\end{align*}
This leads to
\begin{align*}
\frac{C_L-C}{n} \E\int_0^T |\Delta \bm\alpha_t|^2 dt  \leq \frac{n}{2\delta} \E\int_0^T |\bm{E}_t^F|^2 dt + \frac{n}{2\epsilon} \E|\bm{E}^{G}_T|^2.
\end{align*}
Recalling the estimates \eqref{ferror2} and \eqref{gerror2} completes the proof.
\end{proof}

Of course, for this result to be useful, one needs a (dimension-free) estimate on the Poincar\'e constant of the optimally controlled state process, without relying on convexity as in the proof of Lemma \ref{lem.poincare}. The following Proposition gives such a result, under the assumption that 
\begin{align}
\begin{split}
\|D_{i} F \|_{\linf} \leq \frac{C_0}{n}, \quad \|D_{ii} F \|_{\linf} \leq \frac{C_0}{n},  \\ 
\|D_{i} G \|_{\linf} \leq \frac{C_0}{n}, \quad \|D_{ii} G \|_{\linf} \leq \frac{C_0}{n}
\end{split} \label{fglip}
\end{align}
for each $i = 1,...,n$,
and 
\begin{align} \label{dxh}
|D_x H^i(x,0)| \leq C_0, \quad x \in \R^d
\end{align}
for some constant $C_0$. 
Note that these conditions are all satisfied in the mean field case for functions with bounded first and second derivatives; e.g., if $F(\bx) = \sF(m^n_{\bx})$ then $\|D_iF\|_\infty \le \frac{1}{n}\|D_m\sF\|_\infty$.

\begin{proposition} \label{prop.poinnonconvex}
Suppose that Assumption \ref{assump.nonconvex} is in force and that \eqref{fglip} and \eqref{dxh} hold. Let $\bm\alpha$ be an optimizer for \eqref{controldistol}, and let $\bX$ be the corresponding state process. Then there is a constant $C_P$ depending only on $C_0, \|D_{xx} H\|_{\infty}, \|D_{xp}H\|_{\infty}, \|D_{pp} H\|_{\infty}$, $d$, and $T$ such that $\sL(\bX_t)$ satisfies the Poincar\'e inequality with constant $C_P$, for each $0 \leq t \leq T$.
\end{proposition}
\begin{proof}
In this proof, $C$ denotes a constant which may change from line to line but depends only on the parameters stated in the Proposition \ref{prop.poinnonconvex}. For simplicity, we give the proof in the special case $d = 1$, the argument in the general case is the same but more notationally intensive. We also assume that $DV$ and $D^2V$ are both $C^{1,2}$, an assumption which is easily removed by a mollification argument as in the proof of Theorem \ref{thm.est1}. 

We use again the observation that the optimal control $\bm\alpha$ must satisfy $\alpha^i(t,x) = - D_p H^i(x, nDv^i(t,x))$ for some solution $v^i$ to the equation \eqref{vieqn}. Fixing $i$ and setting $w^1 = D v^i$, we differentiate \eqref{vieqn} in space to get 
\begin{align*}
\begin{cases}
- \partial_t w^1 - \Delta w^1 + \frac{1}{n} D_{x} H^i(x,nw^1) + D_p H^i(x, n w^1) Dw^1 = D F^i(t,x), \quad (t,x) \in [0,T) \times \R^d, \\
w^1(T,x) = D G^i(t,x). 
\end{cases}
\end{align*}
Recalling that $dX_t^i = \alpha_t^i dt + dW_t^i$, we see that $Y_t^1 \coloneqq  w^1(t,X_t^i)$ and $Z_t^1 \coloneqq Dw^1(t,X_t^i)$ satisfy
\begin{align*}
dY_t^1 = \bigg(\frac{1}{n} D_x H^i(X_t^i,nY_t^1) - D F^i(t,X_t^i) \bigg) dt + Z_t^1 dW_t^i, 
\end{align*}
with terminal condition $Y_T = D_i G(X_T^i)$.
Since $|D_i G| \leq C_0/n$, $|D_i F| \leq C_0/n$, and $|\frac{1}{n} D_x H^i(X_t^i,n Y_t)| \leq C_0/n + \|D_{xp} H^i\|_{\infty}Y^i_t$, a standard BSDE argument (expanding $d|Y_t|^2$) gives 
\begin{align} \label{bmoest}
\bigg\|\E\bigg[\int_{t}^T |Z_s^1|^2 ds \,\Big|\, \sF_t\bigg]\bigg\|_{\linf} \leq \frac{C}{n^2}, \quad \forall t \in [0,T].
\end{align}
Now we differentiate the PDE again to find that $w^2 \coloneqq D^2 v^i$ satisfies 
\begin{align*}
\begin{cases}
- \partial_t w^2- \Delta w^2 + \frac{1}{n} D_{xx} H^i(x,nw^1) + D_p H^i(x, nw^1) Dw^{2} \\
\qquad + 2D_{xp} H^i(x, nDv^i)w^2 + n D_{pp} H^i(x,nw^1) |w^2|^2  = D^2 F^i(t,x), \quad (t,x) \in [0,T) \times \R^d, \\
w^2(T,x)  = D^2 G^i(x), \quad x \in \R^d. 
\end{cases}
\end{align*}
Thus the processes $Y^2 = w^2(t,X_t)$ and $Z^2 = Dw^2(t,X_t)$ satisfy 
\begin{align*}
dY_t^2 = \bigg( \frac{1}{n} D_{xx} H^i(x, nY_t^1) + 2 D_{xp} H^i(x,n Y^1_t) Y_t^2  + nD_{pp} H^i(x,nY_t^1) |Z^1|^2 - D^2 F^i(t,X_t) \bigg) dt + Z_t^2 dW_t^i, 
\end{align*}
with terminal condition $Y_T^2 = D^2 G^i(X_T^i)$. 
Integrating and taking conditional expectations gives
\begin{align} \label{yrep}
Y_t^2 = \ &\E\bigg[D^2 G(T,X^i_T) - \int_t^T \bigg(\frac{1}{n} D_{xx} H^i(X_s^i, n Y_s^i)  \nonumber \\
& \quad + 2 D_{xp} H^i(X_s^i, nY_s^1)Y_s^2 + n D_{pp} H^i(X_s^i, nY_s^1) |Z_s^1|^2 - D^2 F^i(s,X_s^i) \bigg) ds \,\Big|\, \sF_t \bigg]. 
\end{align}
Using \eqref{bmoest} we have
\begin{align*}
    \E\bigg[\int_t^T nD_{pp} H^i(X_s^i, nY_s^i) |Z_s^1|^2\,ds \,\Big|\, \sF_t \bigg] \leq \frac{C}{n},
\end{align*}
and also by \eqref{fglip} and \eqref{dxh} we have
\begin{align*}
    \|D^2 F^i(t,X_t)\|_{\linf} \le C_0/n, \quad  \|D^2 G^i(t,X_t)\|_{\linf} \le C_0/n. 
\end{align*}
Thus from \eqref{yrep} and Gronwall's inequality we get the estimate $|Y^2_t| \leq C/n$ for all $t$ a.s., from which we infer $\|D^2 v^i\|_{\linf} \leq C/n$. 
The proof is completed by differentiating the optimal control $\alpha^i$, to find 
\begin{align*}
D \alpha^i(t,x) = - D_{xp} H^i(x, nDv^i(t,x)) - n D_{pp}H^i(x, nDv^i(t,x)) D^2 v^i(t,x) \leq C, 
\end{align*}
and then applying Lemma \ref{lem.poincare}. 
\end{proof}

\appendix

\section{Well-posedness of maximum principle FBSDE} \label{sec.appendix}

This appendix contains a proof of Proposition \ref{prop.mpfbsde}, which states that the McKean-Vlasov FBSDE \eqref{prop.mpfbsde} in fact has a unique solution under Assumption \ref{assump.conv}. 

\begin{proof}[Proof of Proposition \ref{prop.mpfbsde}]

For uniqueness, we assume that we have two solutions $(\bX, \bY, \bZ)$ and $(\overline{\bX}, \overline{\bY}, \overline{\bZ})$. Let $\Delta \bX = \bX - \overline{\bX}$ and $\Delta \bY = \bY - \overline{\bY}$. Moreover, we set $\alpha_t^i = - D_p H^i(X_t^i, n Y_t^i)$, $\overline{\alpha}_t^i = - D_p H^i(\overline{X}_t^i, n \overline{Y}_t^i)$, and $\Delta \bm\alpha = \bm\alpha - \overline{\bm\alpha}$. Let $(\sF^i_t)_{t \in [0,T]}$ denote the filtration generated by $(W^i_t)_{t \in [0,T]}$. We compute 
\begin{align*}
d(\Delta X^i_t \cdot \Delta Y^i_t) &= \bigg(\Delta Y^i_t \Delta \alpha_t^i - \frac{1}{n} ( D_x L^i(X_t^i, \alpha_t^i) - D_x L^i(\overline{X}_t^i, \overline{\alpha}_t) ) \cdot  \Delta X_t^i \\
&\qquad - \E[D_i F(X_t) - D_i F(\overline{X}_t) \,|\, \sF_t^i]\cdot  \Delta X_t^i \bigg) dt + dM_t^i \\
&= \bigg(- \frac{1}{n} (D_a L^i(X_t^i, \alpha_t^i) - D_a L^i(\overline{X}_t^i, \overline{\alpha}_t^i))\cdot  \Delta \alpha_t^i - \frac{1}{n} ( D_x L^i(X_t^i, \alpha_t^i) - D_x L^i(\overline{X}_t^i, \overline{\alpha}_t) ) \cdot  \Delta X_t^i \\
&\qquad - \E[D_i F(X_t) - D_i F(\overline{X}_t) \,|\, \sF_t^i] \cdot \Delta X_t^i \bigg) dt + dM_t^i 
\end{align*}
with $M^i$ being a martingale. Here we have used the fact that $\alpha_t^i$ maximizes $a \mapsto - a \cdot nY_t^i - L^i(X_t^i,a)$ so that $Y_t^i = - \frac{1}{n} D_a L^i(X_t^i, \alpha_t^i)$, and likewise $\overline{\alpha}_t^i = - \frac{1}{n} D_a L^i(\overline{X}_t^i, \overline{\alpha}_t^i)$. We have also used the identities
\begin{align*}
\sF^i(X_t^i, \bm m_t) = \E[D_i F(\bX_t) \,|\, X_t^i] = \E[D_i F(\bX_t) \,|\, \sF_t^i],
\end{align*}
the first coming from the definition of $\sF^i$, and the second from the independence of $(X^j)_{j \neq i}$ and $\sF^i_t$-measurability of $X^i_t$.
Integrating, summing over $i$, and taking expectations, we get
\begin{align*}
C_L &\E \int_0^T |\Delta \bm\alpha_t|^2 dt  \\
	&\leq -\E\bigg[ \int_0^T \sum_{i=1}^n  \E[D_i F(\bX_t) - D_i F(\overline{\bX}_t) \,|\, \sF_t^i] \cdot \Delta X_t^i \, dt  + \sum_{i=1}^n \big(\E[D_i G(\bX_T) - D_i F(\overline{\bX}_T) \,|\, \sF_T^i] \cdot \Delta X_T^i \big)  \bigg] \\ &= -\E\bigg[\int_0^T \sum_{i=1}^n (D_iF(\bX_t) -  D_i F(\overline{\bX}_t)) \cdot \Delta X_t^i \, dt + \sum_{i=1}^n (D_i G(\bX_T) - D_i G(\overline{\bX}_T)) \cdot \Delta X_T^i \bigg] \\
&\leq 0, 
\end{align*}
so that $\Delta \bm\alpha = 0$. But once we know that $\Delta \alpha = 0$, clearly $\Delta \bX = 0$, which easily implies that in fact $(\bX, \bY, \bZ) = (\overline{\bX}, \overline{\bY}, \overline{\bZ})$. 

Now we turn to existence. If in addition to Assumption \ref{assump.conv} we assume that $F$ and $G$ are Lipschitz, then existence follows directly from the main result of \cite{cardelmkvfbsde}. When $F$ and $G$ are not Lipschitz but are convex with bounded second derivatives, as in Assumption \ref{assump.conv}, we can approximate $F$ and $G$ by sequences $F^{(k)}$, $G^{(k)}$ in such a way that 
\begin{enumerate}
\item $F^{(k)}$ and $G^{(k)}$ are in $C^2((\R^d)^n)$ and Lipschitz
\item $0 \leq D^2 F^{(k)} \leq C$, $0 \leq D^2 G^{(k)} \leq C$, for each $k \in \N$ and some $C$ independent of $k$
\item $F^{(k)}$ and $G^{(k)}$ are bounded from below, uniformly in $k$
\item $F = F^{(k)}$ and $G = G^{(k)}$ on the ball of radius $k$ in $(\R^d)^n$. 
\end{enumerate}
Then we can, for each $k \in \N$, produce a triple $(\bX^{(k)}, \bY^{(k)}, \bZ^{(k)})$ satisfying
\be
\begin{cases}
\ds dX_t^{(k),i} = - D_p H^i(X_t^{(k),i}, N Y_t^{(k),i}) dt + dW_t^i, \
\ds, \\
dY_t^{(k),i} = -\big( \frac{1}{n}D_x L^i\big(X_t^{(k),i}, -D_pH^i(X_t^{(k),i},NY_t^{(k),i})\big) + \sF^{(k),i}(X_t^{(k),i}, \bm m^{(k)}_t)\big) dt +  Z_t^{(k),i}dW_t^i, \\
X_0^{(k),i} = \xi^i, \quad Y_T^i = \sG^{(k),i}(X_T^{(k),i}, \bm m^{(k)}_t), 
\end{cases}
\ee
where 
\begin{align*}
\bm m^{(k)}_t = (\sL(X_t^{(k),1}),...,\sL(X_t^{(k),n})), 
\end{align*}
and with $\sF^{(k),i}, \sG^{(k),i} : \R^d \times \spt^n \to \R^d$ given by
\begin{align*}
\sF^{(k),i}(x,\bm m^{(k)}) = \langle \bm m^{(k),-i}, D_i F^{(k)} \rangle(x), \quad \sG^{(k),i}(x,\bm m^{(k)}) = \langle \bm m^{(k),-i}, D_i G^{(k)} \rangle(x)
\end{align*}
Our goal will be to show that the sequence $(\bX^{(k)}, \bY^{(k)}, \bZ^{(k)})$ is Cauchy in $\stwo \times \stwo \times \ltwo$, which will clearly imply existence. In fact to do this we will introduce an additional assumption, namely that $\xi^i \in L^p$ for some $p > 2$. We will later remove this assumption by approximation.
We set $\alpha_t^{(k),i} = - D_p H^i(X_t^{(k),i}, nY_t^{(k),i})$, and for $k,j \in \N$, we set $\Delta^{j,k} \bX = \bX^{(j)} - \bX^{(k)}$, and likewise for $\Delta^{j,k} \bY $, $\Delta^{j,k} \bm\alpha$, $\Delta^{j,k} \bZ$. By expanding $d (\Delta^{j,k} \bX \cdot  \Delta^{j,k} \bY)$ exactly as in the uniqueness argument above, we obtain the estimate 
\begin{align}  \label{jkcomp} \nonumber 
C_L \E \int_0^T |\Delta^{j,k} \bm\alpha_t|^2 dt  \leq \E \bigg[ \int_0^T \sum_{i=1}^n \Delta^{j,k} X_t^{i} \cdot (\sF^{(j),i}(X_t^{(k),i}, \bm m_t^{(k)}) - \sF^{(k),i}(X_t^{(k),i}, \bm m_t^{(k)}) ) \big) dt  \\
+ \sum_{i=1}^n  \Delta^{j,k} X_T^{i} \cdot (\sG^{(j),i}(X_T^{(k),i}, \bm m_t^{(k)}) - \sG^{(k),i}(X_t^{(k),i}, \bm m_t^{(k)}) ) \bigg]
\end{align}
Applying Young's inequality to \eqref{jkcomp} together with the observation $\Delta^{j,k} X_t^i = \int_0^t \Delta^{j,k} \alpha_t^i\,dt$, we find a constant $C$ independent of $k$ and $j$ with the property that 
\begin{align} \label{cauchyest}
\E \int_0^T |\Delta^{j,k} \bm \alpha_t|^2 dt  &\leq C \E\bigg[ \int_0^T \sum_{i=1}^n |\sF^{(j),i}(X_t^{(k),i}, \bm m_t^{(j)}) - \sF^{(k),i}(X_t^{(k),i}, \bm m_t^{(k)})|^2 \, dt 
\nonumber \\ &\qquad + \sum_{i=1}^n   |\sG^{(j),i}(X_T^{(k),i}, \bm m_T^{(j)}) - \sG^{(k),i}(X_t^{(k),i}, \bm m_T^{(k)})|^2 \bigg] \nonumber  \\
&= C\E\bigg[\int_0^T \sum_{i=1}^n |\E[D_i F^{(j)}(\bX_t^{(k)}) - D_iF^{(k)}(\bX_t^{(k)}) \,|\, \sF_t^i]|^2 \, dt 
\nonumber \\ &\qquad + \sum_{i=1}^n  |\E[D_i G^{(j)}(\bX_T^{(k)}) - D_iG^{(k)}(\bX_T^{(k)}) \,|\, \sF_T^i]|^2 \bigg] \nonumber \\
&\leq C\E\bigg[\int_0^T \sum_{i=1}^n |D_i F^{(j)}(\bX_t^{(k)}) - D_iF^{(k)}(\bX_t^{(k)})|^2 \, dt \nonumber \\ &\qquad + 
 \sum_{i=1}^n  |D_i G^{(j)}(\bX_t^{(k)}) - D_iG^{(k)}(\bX_t^{(k)}) |^2 \bigg] \nonumber \\
 &\leq C\E\bigg[\int_0^T (1 + |\bX_t^{(k)}|^2) 1_{\{|\bX_t^{(k)}| \geq k \wedge j\}} dt  +  (1 + | \bX_T^{(k)}|^2) 1_{\{|\bX_T^{(k)}| \geq k \wedge j\}}\bigg].
\end{align}
It is easy to see that 
\begin{align*}
\|\Delta^{j,k} \bX\|_{\stwo} + \|\Delta^{j,k} \bY\|_{\stwo} + \|\Delta^{j,k} \bZ\|_{\ltwo} \leq C \|\Delta^{j,k} \bm\alpha \|_{\ltwo}, 
\end{align*}
so in fact \eqref{cauchyest} shows that the sequence $(\bX^{(k)}, \bY^{(k)}, \bZ^{(k)})$ is Cauchy as soon as we show that
\begin{align*}
\E\bigg[\int_0^T (1 + |\bX_t^{(k)}|^2) 1_{\{|\bX_t^{(k)}| \geq k\}} dt  +  (1 + | \bX_T^{(k)}|^2) 1_{\{|\bX_T^{(k)}| \geq k\}}\bigg] \xrightarrow{k \to \infty} 0, 
\end{align*}
which in turn would follow (from a uniform integrability argument) from the estimate 
\begin{align} \label{cauchysuff}
    \sup_k \E\bigg[\sup_{0 \leq t \leq T} |\bX_t^{(k)}|^p\bigg] < \infty,
\end{align}
for some $p > 2$.
To do this, we first note that applying \eqref{jkcomp} with $k = 1$ shows already the weaker estimate
\begin{align} \label{cauchysuffweaker}
    \sup_k \E\bigg[\sup_{0 \leq t \leq T} |\bX_t^{(k)}|^2\bigg] < \infty.
\end{align}
From this, we can see that the functions 
\begin{align*}
F^{(k),i}(\cdot) = \langle m^{(k),-i}, F \rangle, \quad G^{(k),i}(\cdot) = \langle m^{(k),-i}, G \rangle
\end{align*}
satisfy
\begin{align} \label{fkgk}
|F^{(k),i}(x)| \leq C(1 + |x|^2), \quad |DF^{(k),i}(x)| \leq C(1 + |x|), \quad |D^2 F^{(k),i}(x)| \leq C, \nonumber \\
|G^{(k),i}(x)| \leq C(1 + |x|^2), \quad |DG^{(k),i}(x)| \leq C(1 + |x|), \quad |D^2 G^{(k),i}(x)| \leq C,
\end{align}
for some $C$ independent of $k$. But now from Lemma \ref{lem.vichar}, we know that 
\begin{align*}
    \alpha^{(k),i}_t = - D_p H^i(X_t^{(k),i}, n Dv^{(k),i}(t,X_t^{(k),i})),
\end{align*} 
where $v^{(k),i}$ satisfies the PDE 
\begin{align*}
\begin{cases}
- \partial_t v^{(k),i} - \Delta v^{(k),i} + \frac{1}{n} H^i(x,n Dv^{(k),i}) = F^{(k),i}(t,x), \quad (t,x) \in [0,T) \times \R^d, \\
v^{(k),i}(T,x) = G^i(x), \quad x \in \R^d. 
\end{cases}
\end{align*}
From \eqref{fkgk} we deduce that the $v^{(k),i}$ satisfy 
\begin{align*}
|Dv^{(k),i}(x)| \leq C(1 + |x|)
\end{align*}
for some constant $C$ independent of $k$, and so for each $k$ we have $|\alpha_t^{(k),i}| \leq C(1 + |X_t^{(k),i}|)$. Since $dX_t^{(k),i} = \alpha_t^{(k),i} dt + dW_t^i$ and $X_0^{(k),i} = x^i$, from here it is standard to show that \eqref{cauchysuff} holds. 

This completes the proof in the special case that $\xi^i \in L^p$ for some $p > 2$. To remove this additional assumption, we again approximate (and here we recycle notation from earlier in the proof), solving for each $k \in \N$ the equation
\be
\begin{cases}
\ds dX_t^{(k),i} = - D_p H^i(X_t^{(k),i}, N Y_t^{(k),i}) dt + dW_t^i, \
\ds, \\
dY_t^{(k),i} = -\big( \frac{1}{n}D_x L^i\big(X_t^{(k),i}, -D_pH^i(X_t^{(k),i},NY_t^{(k),i})\big) + \sF^{(k),i}(X_t^{(k),i}, \bm m^{(k)}_t)\big) dt +  Z_t^{(k),i}dW_t^i, \\
X_0^{(k),i} = \xi^{(k),i}, \quad Y_T^i = \sG^{(k),i}(X_T^{(k),i}, \bm m^{(k)}_t), 
\end{cases}
\ee
with $\xi^{(k),i} = \xi^{i} 1_{|\xi^i| \leq k}$. Defining $\Delta^{j,k} \bm \alpha$, $\Delta^{j,k} \bm X$, $\Delta^{j,k} \bm Y$, exactly as above, the same computation (expanding the dynamics of $\Delta^{j,k} \bm X \cdot \Delta^{j,k} \bm Y$ this time reveals 
\begin{align*}
\E\bigg[\int_0^T |\Delta^{j,k} \bm \alpha_t|^2 dt \bigg] \leq C \E[|\Delta^{j,k} \bm \xi| |\Delta^{j,k} \bm Y_0|]. 
\end{align*} 
Applying H\"older's inequality to the right-hand side and using the fact that the sequence $\bm Y_0^{(k)}$ is bounded in $L^2$ shows that the sequence of solutions in Cauchy, allowing us to pass to the limit and complete the proof.
\end{proof}

\bibliographystyle{abbrv} 
\bibliography{dist}

\end{document}